%% file: faminodal3.tex
\documentclass[12pt]{amsart}
\usepackage{amsmath}
\usepackage{amssymb}
\usepackage{amsfonts}
\usepackage{amsthm}
\usepackage{verbatim}
\usepackage{cite}
\usepackage{bookman}
\input{myheader2.tex}
\begin{document}
\title {Geometry and intersection theory\\
 on Hilbert schemes of families of nodal curves}

\author
{Ziv Ran}
\date {\today}
\address {Math Dept. UC Riverside  \nl
Surge Facility, Big Springs Road\\
Riverside CA 92521}
\email {ziv @ math.ucr.edu} 

\begin{abstract} We study the relative Hilbert scheme
of
 a family of nodal (or smooth)
curves, over a base of arbitrary dimension,
 via its (birational) {\it{ cycle map}}, going to the relative
symmetric product.  We show the cycle map is the blowing up of the
discriminant locus, which consists of cycles with multiple points.
We work out the action of the blowup or 'discriminant' polarization
on some natural cycles in the Hilbert scheme, including generalized diagonals and cycles, called 'node scrolls',
 parametrizing schemes supported on singular points.
We derive an intersection calculus for Chern classes of tautological
vector bundles, which are closely related to
enumerative geometry.

\end{abstract}
 \thanks { Research Partially supported
by NSA Grant MDA904-02-1-0094 reproduction and distribution of
reprints by US government permitted.}

\maketitle
\tableofcontents

Consider a family of curves given by a flat projective morphism \beq
\pi:X\to B \eeq over an irreducible  base,
with fibres \beq X_b=\pi\inv(b), b\in B \eeq which are irreducible
nonsingular for the generic $b$ and   at worst nodal for every $b$.
For example, $X$ could be the universal family of
automorphism-free curves over the appropriate open subset
of $\overline{\mathcal M}_g$, the moduli space of Deligne-Mumford
stable curves.
Many questions in the classical projective and enumerative geometry
of this family can be naturally  phrased, and in a formal sense
solved (see for instance \cite{R}), in the context of the
{\it{relative Hilbert scheme}} \beq X\sbr m._B=\Hilb_m(X/B), \eeq
which parametrizes length-$m$ subschemes of $X$ contained in fibres
of $\pi$, and the natural {\it {tautological vector bundle}}
$\Lambda_m(E)$, living on $X\sbr m._B$, that is associated to some
vector bundle $E$ on $X$ (for example, the relative dualizing
sheaf $\omega_{X/B}$). Typically, the questions include ones involving
relative multiple points and multisecants in the family, and the
formal solutions involve Chern numbers of the tautological bundles.
Thus, turning these formal solutions into meaningful ones requires
computing the Chern numbers in question. Aside from some low-degree
cases, this problem was left open in \cite{R}. Our main purpose here
is to solve this problem in general. More than that, we shall
in fact provide a calculus to compute arbitrary polynomials
in the Chern classes of the tautological bundles. In the
'absolute' case $E=\omega_{X/B}$, the computation ultimately reduces
these polynomials to polynomials in Mumford's tautological
classes \cite{Mu} on various boundary strata of
$B$. Fortunately
these have been computed by Kontsevich \cite{Kon}
following a conjecture by Witten.
Note that the boundary naturally carries families of
\emph{pointed} curves (of lower genus), the points being node
preimages, and the tautological classes involved will
include the cotangent or $\psi$ classes on these.
\par The calculus that we develop is
fundamentally \emph{recursive} in $m$. The recursion involves
\emph{flag Hilbert schemes}, in particular the full-flag scheme
$W^m=W^m(X/B)$ studied in \cite{R}, as well as its 'flaglet'
analogue $X\sbr\mm._B$, parametrizing flags of schemes of lengths
$m,m-1$. Its starting point is an analogue for $X\sbr \mm._B$ of the
\emph{splitting principle}, established in \cite{R}. This result
(Corollary \ref{tautbun} ) expresses the total Chern class
$c(\Lambda_m(E))$, pulled back to $X\sbr\mm._B$, as a product of
$c(\Lambda_{m-1}(E))$ and a simple 'discriminant' factor involving
Chern classes of $E$ and a certain \emph{discriminant} divisor
$\Gamma\spr m.$. In order to compute polynomials in the Chern
classes of $\Lambda_m(E)$, we are thus reduced recursively to
studying the multiplication action of powers of $\Gamma\spr m.$ on
polynomials in $c(\Lambda_{m-1}(E))$.\par
 Among other things, the
Splitting Principle suggests the central role
played by $\Gamma\spr m.$ in the study of the Hilbert scheme $X\sbr
m._B$, stemming from the fact that it effectively encodes the
information contained in $X\sbr m._B$ beyond the relative symmetric
product $X\spr m._B$. The latter viewpoint is further supported
by the \emph{Blowup Theorem} \ref{blowup} that we prove below, which
says that via the cycle (or 'Hilb-to-Chow') map \beq \frak c_m:X\sbr
m._B\to X\spr m._B, \eeq the Hilbert scheme is equivalent to the
blowing up of the {\it{discriminant locus}} \beq D^m\subset X\spr
m._B, \eeq which is the Weil divisor parametrizing nonreduced
cycles, and where $\Gamma\spr m.=\frak c_m\inv(D^m)$, so that
$-\Gamma\spr m.$ can be identified with the natural $\O(1)$
polarization of the blowup. The Blowup Theorem is valid without
dimension restrictions on $B$.

\par
Given the Blowup Theorem, our intersection calculus proceeds along
the following lines suggested by the aforementioned Splitting
Principle. On each Hilbert scheme $X\sbr m._B$, we identify a
collection of geometrically-defined \emph{tautological classes}
which, together with base classes coming from $X$, additively
generate what we call the \emph{tautological module} $T^m=T^m(X/B)$.
These classes come in two main flavors.\begin{itemize}
\item The (relative) \emph{diagonal classes}: these are loci
of various codimensions defined by diagonal conditions pulled back from the relative symmetric product of $X/B$, possibly twisted by base classes; they are analogous to the 'creation operators' in Nakajima's work in the case of smooth surfaces.
\item The \emph{node classes}: these are associated to relative
nodes $\theta$ of $X/B$, hence roughly to boundary components, and come in 2 kinds: the \emph{node scrolls}, which  parametrize schemes with a length $>1$ component at $\theta$, and are $\P^1$-bundles over a relative Hilbert scheme associated
to the normalization along $\theta$ of the boundary subfamily
of $X/B$ where $\theta$ lives; and the \emph{node sections}. which are simply (intersection) products of a node scroll by the discriminant $\Gamma\spr m.$.
\end{itemize}
Then the first main component of the calculus is the \emph{Module Theorem} \ref{taut-module},
which says that $T^m$ is
(computably!) a module over the polynomial ring $\Q[\Gamma\spr m.]$. Included in this is the nontrivial assertion that
$\Q[\Gamma\spr m.]\subset T^m$; this means we can compute, recusively at least, arbitrary powers of $\Gamma\spr m.$ as $\Q$-linear
combinations of tautological classes.
\par Rounding out the story is the
\emph{Transfer Theorem} \ref{taut-tfr}, which computes the transfer (pull-push) operation on tautological classes from
$X\sbr m-1._B$ to $X\sbr m._B$ via the flaglet Hilbert scheme
$X\sbr\mm._B$, viewed as a correspondence.\par
The conjunction of the Splitting Principle, Module Theorem and
Transfer Theorem computes all polynomials in the Chern
classes, in particular the Chern numbers, of $\Lambda_m(E)$
as $\Q$-linear combinations of tautological classes on $X\sbr m._B$.\par
\par Note that if $X$ is a
smooth surface, there is a natural closed embedding \beq j_\pi\sbr
m.:X\sbr m._B\subset X\sbr m. \eeq of the relative Hilbert scheme in
the full Hilbert scheme of $X$, which is a smooth projective
$2m$-fold. There is a large literature on Hilbert schemes of smooth
surfaces and their cohomology and intersection theory, due to
Ellingsrud-Str{\o}mme, G\"ottsche, Nakajima, Lehn and others, see
\cite{EG, L, LS, N} and references therein. In particular, Lehn
\cite{L} gives a formula for the Chern classes of the tautological
bundles on the full Hilbert scheme $X\sbr m.$, from which one can
derive a formula for the analogous classes on $X\sbr m._B$ if $X$ is
a smooth surface, but this does not, to our knowledge, yield Chern
numbers (besides the top one) on $X\sbr m.,$ much less $X\sbr m._B$
(the two sets of numbers are of course different). Going from Chern
{\it classes} to Chern {\it numbers} it a matter of working out the
top-degree multiplicative structure, i.e. the intersection calculus.
 When $X$ is a surface with trivial
canonical bundle, Lehn and Sorger \cite{LS} have given a rather
involved description of the mutiplicative structure on the
cohomology of $X\sbr m.$ in all degrees, not just the top one. While
products on $X\sbr m.$ and $X\sbr m._B$ are compatible $j_\pi\sbr
m.$, it's not clear how to compute intersection products, especially
intersection { numbers} on $X\sbr m._B$ from products on $X\sbr
m.$, even in case $X$ has trivial canonical bundle. Indeed some of
our additive generators directly involve the {{fibre nodes}} of
the family $X/B$ and do not appear to come from classes on $X\sbr
m.$.  However, note that the class of surfaces with trivial
canonical bundle that fibre (via a morphism, not a rational map)
over a smooth curve is very small (and even smaller if one assumes
at least one singular fibre), so the potential intersection between
our work and \cite{LS} is very small. Besides, our calculus
works for a higher dimensions as well.
\par

The paper is organized as follows. A preliminary \S0 defined certain
combinatorial numbers to be needed later. Then Chapter 1 is devoted
to the proof of the  Blowup Theorem. Actually it is the proof,
rather than the statement, of the Theorem, that is of principal
interest to us. The proof proceeds by first constructing an explicit
model $H_m$ for the cycle map, locally over the base in a
neighborhood of a cycle of the form the point $m[p]$ where $p$ is a
relative node (this cycle is an explicit, albeit non $\Q$-Gorenstein lci-quotient
singularity on the relative symmetric product);
then, via a 'reverse engineering' process, we identify
the cycle map as the blowup of the discriminant locus (this without
advance knowledge of the ideal of the latter). In the sequel, our
main use of the Blowup Theorem is as a convenient way of gluing the
$H_n$ models, $n\leq m$ together, globally over the base $X\spr m._
B$.
\par
Chapter 2 is devoted to the definition of the Tautological Module
$T^m$ and proof of the Module Theorem. As a convenient artifice, we
first define the appropriate classes on an ordered model of the
Hilbert scheme and subsequently pass to the quotient by the
symmetrization map.\par In Chapter 3 we study the geometry of the
flaglet Hilbert scheme $X\sbr\mm._B$, largely referring to
\cite{Hilb}, and derive properties of the transfer operation. We
then review the splitting principle form\cite{R}, which enables us
to complete our claculus.\par
For a detailed sketch of the proof of
Theorem 1 (though without all the details), and some other
applications, see \cite{R2}.

\subsection*{Acknowledgements} I thank Mirel Caibar  and Ethan Cotteril for valuable comments. A preliminary version of some
of these results was presented at conferences in Siena, Italy and
Hsinchu, Taiwan, in June 2004, and KIAS, Seoul, in September 2005,
and I thank the organizers of these conferences for this
opportunity.

\setcounter{section}{-1}

\newsection{Preliminaries} \subsection{Staircases} We define a combinatorial
function that will be important in computations to follow. Denote by
$Q$ the closed 1st quadrant in the real $(x,y)$ plane, considered as
an additive cone. We consider an integral {\it{staircase}} in $Q$.
Such a staircase is determined by a sequence of points
\beq (0,y_m),(x_1,y_m),(x_1,y_{m-1}),(x_2,y_{m-1}),...,(x_m,y_1),
(x_m,0) \eeq where $0<x_1<...<x_m, 0<y_1<...<y_m,$ are integers, and
consists of the polygon
\beq B=(-\infty,x_1)\times \{y_m\}\bigcup \{x_1\}\times
[y_{m-1},y_m]\bigcup [x_1,x_2]\times \{y_{m-1}\}\bigcup...\bigcup
\{x_m\}\times (-\infty, y_1]. \eeq The \emph{upper region} of $B$ is by
definition \beq R=B+Q=\{(b_1+u_1, b_2+u_2): (b_1,b_2)\in B, u_1,u_2\geq
0\} \eeq

We call such $R$ a \emph{special infinite polygon}. The closure of
the complement \beq S=R^c:=\overline{Q\setminus R}\subset Q \eeq has
finite (integer) area and will be called a \emph{special finite
polygon}; in fact the area of $S$ coincides with the number of
integral points in $S$ that are $Q$-interior, i.e. not in $R$; these
are precisely the integer points $(a,b)$ such that $[a,a+1]\times [b,b+1]\subset S$.\par
Note that we may associate to $R$ a monomial
ideal $\I(R)<\C[x,y]$ generated by the monomials $x^ay^b$ such that
$(a,b)\in R\cap Q$. The area of $S$ then coincides with
$\dim_\C(\C[x,y]/\I(R)).$ It is also possible to think of $S$ as a
partition or Young tableau, with $x_1$ many blocks of size $y_m$,
..., $(x_{i+1}-x_i)$ many blocks of size $y_{m-i}...$\par
Fixing a
natural number $m$, we define the \emph{basic special finite polygon associated to
$m$} as
\beq S_m=\bigcup_{i=1}^m
[0,\binom{m-i+1}{2}]\times[0,\binom{i+1}{2}]. \eeq It has area
\beq \alpha_m=\sum\limits_{i=1}^{m-1} i\binom{m+1-i}{2}
=\frac{m(m+2)(m^2-1)}{24} \eeq and associated special infinite polygon
denoted $R_m.$ Now for each integer $j=1,...,m-1$ we define a
special infinite polygon $R_{m,j}$ as follows. Set
\beq P_j=(m-j, -j)\in\R^2, \eeq
\beq R_{m,j}= R_m\cup(R_m+P_j)\cup[0,\infty)\times[j,\infty)
 \eeq (where $R_m+P_j$ denotes the translate of $R_m$ by $P_j$ in
$\R^2$).
We also define $P_j^-, P_j^+$ analogously with $m-j$
replaced by $m-j-1$ and $m-j+1$, respectively.
Then let $S_{m,j}=R_{m,j}^c,$
\beq \beta_{m,j}={\text {area}}(S_{m,j}), \eeq
\beq \beta_m=\sum_{j=1}^{m-1}\beta_{m,j}. \eeq
We also define $\beta_{m,j}^\pm$ based on $P_j^+, P_j^-$ respectively.
It is easy to see that
\beql{}{ \beta_{m,1}=\binom{m}{2},
\beta_{m,2}=\binom{m}{2}+\binom{m-1}{2}-1, \beta_{m,j}=\beta_{m,m-j}},
\beql{}{\beta^-_{m,1}=\binom{m-1}{2}-1,} but
otherwise we don't know a closed-form formula for these numbers in
general. A few small values are
\beq \beta_{2,1}=\beta_2=1 \eeq
\beq \vec{\beta}_{3}=(3,3), \beta_3=6 \eeq \beq\vec{\beta}_{4}=(6,8,6),
\beta_4=20\eeq  \beq\vec{\beta}_{5}=(10, 15,15,10), \beta_5=50 \eeq
\beq \vec{\beta}_{6}=(15,24,27,24,15), \beta_6=105. \eeq
For example, for $m=5$ the relevant finite polygons, viewed as
partitions, are
\beq S_{5,1}=1^{10}, S_{5,2}=2^61^3, S_{5,3}=3^5, S_{5,4}=3^22^2. \eeq Set
\beq J_m=\I(R_m). \eeq
For an interpretation of the $\beta_{m,j}$ as exceptional
multiplicities associated to the blowup of the monomial ideal $J_m$,
see \S 1.6 below. 
\subsection{Products, diagonals, partitions}
\subsubsection{Partitions, distributions and shapes}
\label{partitions}
The intersection calculus we aim to develop is couched
in terms certain diagonal-like loci on products,
defined in the general case in terms of partitions. To facilitate working with these loci systematically, we now establish some conventions, notations and simple remarks
related to partitions.\par
Given a natural number $m$,
a \emph{partition} in $[1,m]$ for us is a sequence of pairwise disjoint index sets
$$I_1,...,I_r\subset [1,m]\cap \Z.$$
Elsewhere this is sometimes called a 'labelled partition'.
 Two partitions are said to
be \emph{equivalent} if they differ only by singleton blocks and by renumbering of
blocks.
A partition in $[1,m]$ is said to be \emph{full} if $\bigcup\limits_\l I_\l=\{1,...,m\}$, \emph{irredundant} if it has no singleton blocks. Clearly any partition is equivalent to a full one with nonincreasing block cardinalities, and a to a irredundant one, again with nonincreasing block
cardinalities.\par
We partially order the partitions by declaring that $\Phi_1\prec\Phi_2$ if every non-singleton block of $\Phi_1$
is a union of (possibly singleton) blocks of $\Phi_2$. Thus
$\Phi_1$ and $\Phi_2$ are equivalent iff $\Phi_1
{\prec}\Phi_2$ and $\Phi_2
{\prec}\Phi_1$ .\par
Now by a \emph{length distribution} (elsewhere called a partition)
we mean a function $\un$
of finite support from the positive integers to the nonnegative integers. The \emph{total length} of $\un$ is by definition
$\sum\limits_\l\un(\l).$ To any partition $(I.)$ there is an associated length distribution
$\un=(|I.|)$, defined by $\un(\l)=|I_\l|, \forall \l$.
In light of these, there are natural notions of fullness
(with respect to $m$), precedence and equivalence for distributions.
\par

Now given a distribution $\un$, let $n_1> n_2>..,n_r\geq 1$
be its set of distinct nonzero values, i.e. the \emph{value sequence}, and
$$\mu_\l=|\{\l:\un(\l)=n\}|$$
be its \emph{frequency
sequence} (also called sometimes frequency function using
the notation $\mu_{\un}$).
The \emph{shape} of $\un$ is the (obviously finite)
sequence
\beql{}{(n.\spr{\mu.}.)=(n\spr\mu(n).)=(...,m^{\spr\mu_\un(m).},...,
1^{\spr\mu\subp{n.}(1).})=
(n_1^{\spr\mu_1.},...,n_r^{\spr\mu_r.})}
where the exponents are taken formally with respect to the juxtaposition operation (and all terms with
zero exponent in the LHS are omitted).  The shape of $\un$ obviously determines
$\un$, and indeed distribution and shape are equivalent data,
each preferable in different situations, and will be used
interchangeably. A distribution is in \emph{loose form} (resp. \emph{shape form}) if it is written as $(n_1\geq n_2\geq...)$
(resp. $(n_1\spr\mu_1.,...)$)
where each $n_i$ is the $i$-th member
(resp. the $i$-th distinct member) of it, in
nonincreasing (resp. strictly decreasing)
order.\par
The shape of a partition $(I.)$, i.e. the shape
of its distribution, is
also written $(|I.|)$ and has
the form $(n_1^{\mu\subp{I.}(n_1)},...,n_r^{\mu\subp{I.}(n_r)})$
where the $n_\l,\mu_\l=\mu\subp{I.}(n_\l)$ are precisely the heights and widths
 of the adjacent rectangles forming the Young tableau for
$(I.)$.  Thus we may think of a shape (or a distribution) as an
unspecified partition having the given shape or distribution.\par
Next we define some natural operations on distributions that we will need. If $(n'.), (n".)$
are distributions, we can define a new distribution denoted
$(n'.)\coprod (n".)$ by the condition that its frequency function
$\mu$ is the sum of that of $(n'.)$ and $(n".)$, i.e. \beql{}{
\mu_{(n'.)\coprod (n".)}=\mu_{(n'.)}+\mu_{(n".)}. } This corresponds
to the operation of disjoint union of partitions. For a distribution
$(n.)$ and an integer $k$, we define a distribution $(n.)\setminus
k$ by \beql{}{\mu_{(n.)\setminus k}=\mu_{(n.)}-\mathbf 1_k} where
corresponds to removing a block of size $k$ from a partition. By
convention, $(n.)\setminus k=\emptyset$ if $(n.)$ has no block of
size $k$; more generally, a distribution of shape is considered
empty of the corresponding frequency function has a negative
value.\par We define another operation $(n.)^{-\l}$ as follows: let
$n_1>...>n_\l>...$ be the distinct block sizes occurring in $(n.)$.
Then \beql{}{ (n.)^{-\l}=(n.)\setminus n_\l\coprod (n_\l+1)} which
corresponds to removing a block of size $n_\l$ and replacing it by
one of size $n_\l+1$.
 Also,
\beql{}{u_{j,\l}(n.)=(n.)\setminus n_j\setminus n_\l\coprod
(n_j+n_\l) } which corresponds to uniting an $n_j$ and an $n_\l$
block.\par
 Define \emph{multidistribution data} $\phi$ as
$(n_\l:\un'|\un")$ where $ (\un')\coprod (\un")=(\un)\setminus
n_\l$, where $(\un)$ is a (usually full) distribution on $[1,m]$
(this notation, motivated by the partition case, indicates removing
a single block of size $n_\l$ and breaking up the remainder into $x$
type and $y$ type).
\par
\par

\subsubsection{Products and diagonal loci}\label{prod-diag-loci}
Now given any set $X$ and partition $(I.)$, we define
the diagonal locus
$X^{(I.)}$ as the set of 'locally constant', i.e. constant on blocks, functions $\bigcup I.\to X$.
This is called the (ordered) diagonal locus corresponding to $(I.)$.
 If $X$ is endowed with a
map to $B$, there is an analogous relative notion $X^{(I.)}_B$
referring to functions such that the composite
$\bigcup I.\to B$ is constant (the extension from the
absolute to the relative case involves no new ideas and
will not be emphasized in what follows). The diagonal locus $X^{(I.)}$
is a subset (closed subscheme, if $X$ is a separated scheme)
of the cartesian product $X^m$, and can be identified
(isomorphically, if $X$ is a scheme), in terms of the above shape notation, with
$\prod\limits_{\l=1}^r X^{ \mu_\l}.$ \par
Now there is an unordered analogue of $X^{(I.)}$,
 depending only on the shape of $I.$ and denoted $X^{\un}$ or $X^{(n.^{\mu.})}$ or $X^{(\coprod n^{\mu(n)})}$.
 This is called a symmetric diagonal locus.
 It coincides with the image of $X^{(I.)}$ in the symmetric product $X^{(m)}$, and can in turn be
identified with $\prod\limits_{\l=1}^rX\spr{\mu_\l}.
=\prod\limits_{n=\infty}^1X\spr\mu(n).$.
Note the natural 'diagonal' embedding
\beql{delta-n-ell}{\delta_{n.,\l}:X\spr {n.^{-\l}}.\to X\spr {n.}.\times X}
\par
Next, multidistribution data have to do with a situation where
X comes equipped
with a representation $X=X'\cup X"$ ;
the case $X'=X"$ will require special treatment.
The corresponding (symmetric) diagonal locus
is
\beql{}{X^\phi=\prod\limits_\l (X')\spr{\mu'_\l}.\times
\prod\limits_\l (X")\spr{\mu"_\l}.}
where $\un'=(n'.^{\mu.'}), \un"=(n".^{\mu."})$.
We view $X^\phi$ as embedded in $(X')^{(n')}\times (X")^{(n")}$,
where $n'=|\un'|, n"=|\un"|$, the embedding coming from
by the various diagonal embeddings $X'\to (X')^{(n.)}$
which induce $(X')\spr{\mu.}.\to ((X')\spr {n'.}.)
\spr{\mu.}.$;
similarly for $X"$. Note the natural map $X^\phi\to X\spr n'+n".$.

If $X'=X"$, we take
$n"=\emptyset$.\par
Now, an obvious issue that comes up  is
to determine the degree of the symmetrization map
$X^{(I.)}\to X^{(|I.|)}$.
To this end,
let $I.$ be a full partition
on $[1,m]$ with length distribution $(n.)$. Let
$$in(I.)\lhd out(I.)(<\frak S_m)$$
denote the groups of permutations of $\{1,...,m\}$
 taking each block
of $I.$ to the same  (resp. to some) block. Then
$$aut(I.)=out(I.)/in(I.)$$
 is the 'automorphism group' of $(I.)$ (block permutations
 induced by elements of $\frak S_m$). Let
$$a(I.)=|aut(I.)|.$$
Then $a(I.)$ is easily computed: if $(n.^{\mu.})$ is the shape
of $(I.)$,
we have
\beql{aI1}{a(I.)=\prod(\mu_\l)!\ \ \
.}
Clearly $a(I.)$ depends only on the length distribution $\un=(|I.|)$, so we may (abusively) write $a(\un)$ for $a(I.)$.
Significant for our purposes, but easy to verify, is the following \begin{lem} If $I.$ is full and $\un=(|I.|)$ , then
the mapping degree
\beql{aI2}{\deg(X^{(I.)}\to X^{(|I.|)})
=a(\un).}\qed
\end{lem}
\subsubsection{Cohomology and base classes}
\label{base-classes}
Finally, we briefly discuss cohomology. If $X$ is a reasonable
space (topological, scheme, etc.) and $H^.$ is a reasonable
$\Q$-valued cohomology theory (singular, Chow, etc.), then
we have a homomorphism
\beql{}{\Sym^\mu(H^.(X))\to H^.(X\spr \mu.),
}
where $X\spr\mu.$ is the appropriate symmetric product,
whose image is called the ring of \emph{base classes}
on $X\spr\mu.$ and denoted $B(X\spr \mu.)$ (abusively so,
of course, since it depends on $X$).
Similarly, in the relative situation $X/B$, we  get
a ring homomorphism ('symmetrization')
\beql{}{\sigma:\Sym^\mu(H^.(X))\to H^.(X\spr \mu._B),}
whose image is called the ring of {base classes}
on $X\spr\mu._B$ and denoted $B(X\spr \mu._B)$.
The ring  $B(X\spr\mu.)$ is clearly generated
by images, called \emph{polyclasses} of classes of the form
\beql{}{X\spr \mu.[\alpha\spr\lambda.]:=
\sigma(\alpha^\lambda), \lambda\leq\mu.}
When $\lambda=1$, the superscript will be omitted.
We also write
\beql{}{X\spr \mu.[\alpha_1\cdots
\alpha_\lambda]:=
\sigma(\alpha_1\cdots\alpha_\lambda),}
this being simply
$$\frac{1}{\lambda !}\pi_*
(\sum\limits_{s\in\mathfrak S_\lambda}
\alpha_{s(1)}\bt \cdots\alpha_{s(\lambda)}\bt 1
\cdots \bt 1)$$
where $\pi:X^\mu\to X\spr\mu.$ is the natural
map and $\bt$ refers to exterior cup product (restricted
on the fibred product).
Thus, \[ X\spr \mu.[\alpha\spr\lambda.]=
X\spr \mu.[\alpha\spr\lambda.1\spr{\mu-\lambda}.].\]
 We will primarily use the natural
analogues of these constructions in the relative case
$X/B$. Note that the polyclasses are multiplied according to the rule
\beql{poly-mult}{
X\spr \mu.[\alpha_1\spr\lambda_1.]X\spr \mu.[\alpha_2\spr\lambda_2.]=\sum\limits_{\nu=0}
^{\min(\lambda_1, \lambda_2)} X\spr \mu.[\alpha_1\spr\lambda_1-\nu.\cdot\alpha_2\spr
\lambda_2-\nu.\cdot(\alpha_1\substack{\\ .\\ X}\alpha_2)\spr\nu.]
}
\par
Given spaces $X_1,...,X_r$ as above, we write
$$(X_1\times...\times X_r)[\alpha_1,...,
\alpha_r]=\alpha_1\bt\cdots\bt\alpha_r.$$
These are called base classes on the Cartesian product.
In particular, given a distribution $(n.^{\mu.})$, we obtain
in this way classes $X\spr {n.^{\mu.}}.
[\alpha_1\spr \lambda_1.,...,\alpha_r\spr \lambda_r.]$;
these are again called \emph{base classes} on the
symmetric diagonal locus $X\spr {n.^{\mu.}}.\simeq\prod\limits_{i=1}^rX\spr\mu(n_i).$.
\par
We will require some operations on these base classes.
Consider a base class $\gamma$ which we write in the form
\beql{}{\gamma=X\spr {n.^{\mu.}}.[\coprod \alpha(n)^{(\lambda(n))}].}
Then generally, given any (co)homology vector $\beta.=
(\beta_1,...,\beta_r)$, with each $\beta_i$ a class on
$X_i$, we can define a
new (co)homology
class by
\beql{}{\gamma\star_t[\beta.]=
\gamma\cup s_t(\beta)
}
where $s_t$ is the $t$-th elementary symmetric function
(in terms of the exterior $\bt$ product, where missing factors
are deemed $=1$). A particular case
we will use is
\beql{star1}{X\spr {n.^{\mu.}}.
[\alpha_1\spr \lambda_1.,...,\alpha_r\spr \lambda_r.]
\star_1[\omega_1,...,\omega_r]=
\sum\limits_{\l=1}^r
X\spr{n.^{\mu.}}.[\alpha.\spr{\lambda.}.]\cup
p_\l^*(\omega_\l)
}
For a single class $\omega_\l$ on $X$, set
\beql{}{{X\spr {n.^{\mu.}}.
[\alpha_1\spr \lambda_1.,...,\alpha_r\spr \lambda_r.]
\star_{1,\l}[\omega_\l]=
X\spr{n.^{\mu.}}.[\alpha.\spr{\lambda.}.]\cup
p_\l^*(\omega_\l)
}}\par
Similar notations will be used in the case of multipartition
data, e.g.
\beqa{}{X^{(n:n'.^{(\mu'.)}|n".^{(\mu".)})}
[(\alpha_1)^{(\lambda'_1)},...
;(\alpha"_1)^{(\lambda"_1)},...]=\\ \nonumber
(X')^{(n'.^{(\mu'.)})}
[(\alpha_1)^{(\lambda'_1)},...]\times
(X")^{(n".^{(\mu".)})}
[(\alpha"_1)^{(\lambda"_1)},...]}
Again, all these constructions also have natural analogues in the relative situation.
\par Also, there are analogues of these constructions with
$H^.(X)$ replaced by any $\Q$-subalgebra of itself.
For example, $X$ may be a surface fibred over a smooth curve $B$ and endowed
with a polarization $L$, in which case we will usually
consider the subalgebra
\beql{}{K^.(X/B)=H^.(B)[ L,\omega_{X/B}].}
\subsubsection{Canonical class and half-discriminant} Let $X/B$
be a family of smooth curves and \[ D^m=X\spr
2,1^{m-2}._B\subset X\spr m._B\]
the big diagonal or discriminant. This is a reduced Cartier divisor, defined locally by the discriminant function which is
a polynomial in the elementary symmetric functions of a local parameter of $X/B$. The associated line bundle $\O(D^m)$ is always divisible by 2 as line bundle. One way to see this
is to note that $D^m$ is the branch locus of a flat double cover
\eqspl{}{\epsilon:X^{\{m\}}_B\to X\spr m._B
}
where $X^{\{m\}}_B=X^m_B/\frak A_m$ is the 'orientation product', generically parametrizing an $m$-tuple together with
an orientation.
 An explicit 'half' of $\O(D^m)$ is given by
\eqspl{}{h=X\spr m._B[\omega_{X/B}]\otimes \omega_{X\spr m._B/B}\inv
} Indeed $\epsilon^*h$ is precisely the (reduced) ramification
divisor of $\epsilon$, which is half of $\epsilon^*D^m$. In particular, note that $\epsilon^*h$ is effective.
We also have \eqspl{}{\epsilon_*\O_{X^{\{m\}}_B}=\O_{X\spr m._B}\oplus h\inv.
}

\section{The cycle map as blowup}\newsubsection{Set-up} Let
\beq \pi : X\to B \eeq be a family of nodal (or smooth) curves
with $X, B$ smooth. Let $X^m_B,  X\spr m._B$, respectively, denote
the $m$th Cartesian and symmetric fibre products of $X$ relative
to $B$. Thus, there is a natural map
\beq \omega_m:X^m_B\to  X\spr m._B \eeq which realizes its target
as the quotient of its source under the permutation action of the
symmetric group $\frak S_n.$ Let \beq \Hilb_m(X/B)=X\sbr m._B \eeq
denote the relative Hilbert scheme paramerizing length-$m$
subschemes of fibres of $\pi$, and
\beq \frak c= \frak c_m :X\sbr m._B\to  X\spr m._B \eeq the natural
{\it{cycle map}} (cf.\cite{ang}). Let $D^m\subset X\spr m._B$ denote
the discriminant locus or 'big diagonal', consisting of cycles
supported on $<m$ points (endowed with the reduced scheme
structure). Clearly, $D^m$ is a prime Weil divisor on $ X\spr
m._B$, birational to $X\times_B\Sym^{m-2}(X/B)$, though it is less
clear what the defining equations of $D^m$ on $ X\spr m._B$ are
near singular points. The main purpose of Part 1 is to prove
\begin{thm}[Blowup Theorem]\label{blowup} The cycle map \beq \frak c_m:X\sbr m._B\to
 X\spr m._B \eeq is the blow-up of
$D^m\subset X\spr m._B$.\end{thm}
The proof will proceed through an explicit construction of the map
$\frak c_m$, source included, locally over $X\spr m._B$. This construction will play a crucial role in the entire paper and therefore will be studied in greater detail than is required solely to prove Theorem \ref{blowup}.

 \newsubsection{Preliminary
reductions} To begin with, we reduce the Theorem to a local
statement over a neighborhood of a 1-point cycle $mp\in X\spr m._B$
where $p\in X$ is a node of $\pi\inv(\pi(p))$. Set
\beq \Gamma\spr m.=\frak c_m\inv(D^m)\subset X\sbr m._B.\eeq
It was shown in \cite{R}, and will be reviewed below, that $\frak
c_m$ is a small birational map (with fibres of dimension $\leq
\min(m/2, \max\{|\text{sing}(X_b)|, b\in B\})$), all its fibres (aka
punctual Hilbert schemes or products thereof) are reduced, and that
$X\sbr m._B$ is smooth.\begin{lem}\label{purity} $\Gamma\spr m.$ has no
embedded components.\end{lem} We defer the proof to \S 1.5
below.\par Assuming the Lemma, $\Gamma\spr m.$ is an integral,
automatically Cartier, divisor, and therefore $\frak c=\frak c_m$
factors through a map $\frak c'$ to the blow-up $B_{D^m}( X\spr
m._B)$, and it would suffice to show that $\frak c'$ is an isomorphism,
which can be checked locally. This is the overall plan that we now set out to execute. First a few reductions.
\par
Let $U\subset X\spr m._B$ denote the open subset consisting
of cycles having multiplicity at most 1 at each fibre node. Then $U$ is smooth and the cycle map
$\frak c_m:\frak c_m\inv(U)\to U$ is
an isomorphism. Consequently,
it will suffice to show $\frak c_m$ is
equivalent to the blowing-up of $D^m$ locally near any cycle $Z\in
X\spr m._B$ having degree $>1$ at some point of the locus $X^\sigma\subset X$ of
singular points of $\pi$ (i.e. singular points of fibres).\par Our
next point, a standard one, is that it will suffice to analyze
$\frak c_m$ locally over a neighborhood of a 'maximally singular'
fibre, i.e. one of the form $mp$ where $p$ is a singular point of
$\pi$. For a cycle $Z\in X\spr m._B$, we let $(X\spr m._B)_{(Z)}$
denote its open neighborhood consisting of cycles $Z'$ which are 'no
worse' than $Z$, in the sense that their support $\supp(Z')$ has
cardinality at least equal to that of $\supp(Z)$. Similarly, for a
$k$-tuple $Z.=(Z_1,...,Z_k)\in\prod X\spr m_i._B$, we denote by
$(\prod X\spr m_i._B)_{(Z.)}$ its open neighborhood in the product
consisting of 'no worse' multicycles, i.e. $k$-tuples $Z'.\in \prod
X\spr m_i._B$ such that each $Z'_i$ is no worse than $Z_i$ and the
various $Z'_i$ are mutually pairwise disjoint. We also denote by
$(X\sbr m._B)_{(Z)}, (\prod X\sbr m_i._B)_{(Z.)}$ the respective
preimages of the no-worse neighborhoods of $Z$ and $Z.$ via $\frak
c_m$ and $\prod\frak c_{m_i}$. Now writing a general cycle
\beq Z=\sum \limits_{i=1}^km_ip_i \eeq with $m_i>0, p_i$ distinct, and
setting $Z_i=m_ip_i$, we have a cartesian (in each square) diagram
\beq\begin{matrix} (\prod\limits_{i=1}^k {_B}\ X\sbr m_i._B)_{(Z.)}&\stackrel
{\prod \frak c_{m_i}}{ \longrightarrow}& (\prod\limits_{i=1}^k {_B}\
X\spr m_i._B)_{(Z.)}\\e_1\uparrow&\square&\ \ \
\uparrow d_1\\
\ \ \ H &\to&S\\
e\ \downarrow &\square&\ \ \downarrow d\\
(X\sbr m._B)_{(Z)}&\stackrel{ \frak c_m}{\longrightarrow} & (X\spr
m._B)_{(Z)}
\end{matrix}
\eeq
Here $H$ is the restriction of the
natural inclusion correspondence on Hilbert schemes: \beq H=\{
(\zeta_1,...,\zeta_k, \zeta)\in (\prod\limits_{i=1}^k {_B}\ X\sbr
m_i._B)_{(Z.)}\times (X\sbr m._B)_{(Z)}:\zeta_i\subseteq\zeta,
i=1,...,k\}, \eeq and similarly for $S$. Note that the right vertical
arrows $d, d_1$ are \'etale and induce analytic isomorphisms between
some analytic neighborhoods $U$ of $Z$ and $U'$ of $Z.$ and the left
vertical arrows $e, e_1$ are also \'etale and induce isomorphisms
between $\frak c_m\inv(U)$ and $(\prod \frak c_{m_i})\inv(U')$.\par

Now by definition, the blow-up of $ X\spr m._B$ in $D^m$ is the Proj
of the graded algebra
\beq A(\I_{D^m})=\bigoplus\limits_{n=0}^\infty\I_{D^m}^{\ n}. \eeq Note that
\beq d\inv(D^m)=\sum p_i\inv(D^{m_i}) \eeq and moreover,
\beq d^*(\I_{D^m})=\bigotimes{_B}\ p_i^*(\I_{D^{m_i}}) \eeq where we use
$p_i$ generically to denote an $i$th coordinate projection.
Therefore,
\beq A(\I_{D^m})\simeq\bigotimes{_B}\ p_i^*A(\I_{D^{m_i}}) \eeq as graded
algebras,  compatibly with the isomorphism
\beq \O_{\prod\limits_{i=1}^k {_B}\ \Sym^{m_i}(X/B)}\simeq\bigotimes
\limits_{i=1}^k{_B}\  \O_{ \Sym^{m_i}(X/B)}. \eeq Now it is a general
fact that Proj is compatible with tensor product of graded
algebras, in the sense that \beq \Proj(\bigotimes {_B}\ A_i)\simeq
\prod {_B}\ \Proj(A_i). \eeq Consequently (1.2.2) induces another
cartesian diagram with unramified vertical arrows
\beq\begin{matrix} (\prod\limits_{i=1}^k {_B}\
X\sbr m_i._B)_{(Z.)}&\stackrel{ \prod c'_{m_i}}{ \longrightarrow}&
(\prod\limits_{i=1}^k {_B}\ B_{D^{m_i}}X\spr m_i._B)_{(Z.)}\\
\uparrow&\square&\uparrow\\
H'&\to&S'\\
\downarrow&\square&\downarrow \\
(X\sbr m._B)_{(Z)}&\stackrel{ c'_m}{\longrightarrow} &(B_{D^m} X\spr
m._B)_{(Z)}.
\end{matrix}
\eeq
Here ${B_uv}_{(Z)}$ in a blowup means, $\forall u,v$,  the inverse image of $v_{(Z)}$ via the blowing-up map. To prove $c'_m$ is an isomorphism,
it will suffice to prove that each of its fibres over a point
$\gamma$ lying over $Z$ is schematically a point. Given $\gamma$,
there is a unique point of $S'$ that maps to it and that on the
other side maps to a point, say $\delta\in (\prod\limits_{i=1}^k
{_B}\ B_{D^{m_i}}X\spr m_i._B)_{(Z.)}$, that lies over $(Z.)$. Then
$(\frak c'_m)\inv(\gamma)$ corresponds isomorphically to
$(\prod\frak c'_{m_i})\inv(\delta).$ Therefore, it suffices to prove
that $c'_{m_i}$ is an isomorphism for each $i$. 
The upshot of all this is that it suffices to prove $c=\frak c_m$ is
equivalent to the blow-up of $ X\spr m._B$ in $D^m$, locally over a
neighborhood of a cycle of the form $mp, p\in X$, and we may
obviously assume that $p$ is a singular point of $\pi.$\par
\newsubsection{A local model}
We now reach the heart of the matter: an explicit construction,
locally over the symmetric product, of the relative Hilbert scheme in terms of coordinates. This construction will stand us in good stead for the remainder of the paper, much beyond the proof of the Blowup Theorem. We begin with some preliminaries.\par
Fixing a fibre node $p$ as above, lying on a singular fibre
$X_0$,
an affine (if $p$ is a separating node) or analytic or formal (in any case) neighborhood $U$ of $p$ in $X$ so that
$\pi$ is given on $U$ by pulling back a universal deformation
\beql{node-deform}{ t=xy.}
Since both the relative Hilbert scheme and the blowing-up process
are compatible with pullback, we may as well  assume that $U/B$ is itself given by \eqref{node-deform}.
 Then the relative cartesian product $X^m_B$, as
subscheme of $X^m\times B$, is given locally  by
\beq x_1y_1=...=x_my_m=t.\eeq Let $\sigma_i^x, \sigma_i^y,
i=0,...,m$ denote the elementary symmetric functions in
$x_1,...,x_m$ and in $y_1,...,y_m$, respectively, where we set
$\sigma_0=1$.
We note that these functions satisfy the relations
\begin{eqnarray}\label{sigma-x-y-rel}
\sigma^y_m\sigma^x_j=t^j\sigma^y_{m-j},\ \  \sigma^x_m\sigma^y_j=t^j\sigma^x_{m-j}, \\
\label{sigma-x-y-rel2}
t^{m-i}\sigma^y_{m-j}=t^{m-i-j}\sigma^x_j\sigma^y_m,\ \
t^{m-i}\sigma^x_{m-j}=t^{m-i-j}\sigma^y_j\sigma^x_m,
\end{eqnarray}
Putting the sigma functions together with the projection to $B$, we get
a map
\beq \sigma:\Sym^m(U/B)\to \A^{2m}_B= \A^{2m}\times B\eeq
 \beq \sigma=((-1)^m\sigma_m^x,...,-\sigma_1^x,
(-1)^m\sigma_m^y,...,-\sigma_1^y, \pi^{(m)}) \eeq where $\pi^{(m)}:
X\spr m._B\to B$ is the structure map.

  \begin{lem}\label{sigma-emb} $\sigma$ is an
embedding locally near $mp$.\end{lem}\begin{proof} It suffices
to prove this formally, i.e. to show that $\sigma_i^x, \sigma_j^y,
i,j=1,...,m$ generate topologically the completion $\hat{\m}$ of
the maximal ideal of $mp$ in $ X\spr m._B.$ To this end it
suffices to show that any $\frak S_m$-invariant polynomial in the
$x_i, y_j$ is a polynomial in the $\sigma_i^x, \sigma_j^y$ and
$t$. Let us denote by $R$ the averaging or symmetrization operator
with respect to the permutation action of $\frak S_m$, i.e.
\beq R(f)=\frac{1}{m!}\sum\limits_{g\in\frak S_m}g^*(f). \eeq
 Then it suffices to show that the elements
$R(x^Iy^J)$, where $x^I$ (resp. $y^J$) range over all monomials in
$x_1,...,x_m$ (resp. $y_1,...,y_m$) are polynomials in the
$\sigma_i^x, \sigma_j^y$ and $t$. Now the relations
(\ref{sigma-x-y-rel}-\ref{sigma-x-y-rel2})
on the image of $X^m_B$ easily implies that
\beq R(x^Iy^J)-R(x^I)R(y^J)=tF \eeq where $F$ is an $\frak
S_m$-invariant polynomial in the $x_i, y_j$ of bidegree
$(|I|-1,|J|-1)$, hence a linear combination of elements of the
form $R(x^{I'}y^{J'}), |I'|=|I|-1, |J'|=|J|-1$. By induction, $F$
is a polynomial in the $\sigma_i^x, \sigma_j^y$ and clearly so is
$R(x^I)R(y^J).$ Hence so is $R(x^Iy^J)$ and we are
done.\end{proof}
\begin{rem} It will follow from Theorem 1 and its proof that
the equations (\ref{sigma-x-y-rel}-\ref{sigma-x-y-rel2}) actually
define the image of $\sigma$ scheme-theoretically
(see Cor. \ref{sym-eq} below); we won't need this,
however.\end{rem}

Now we present a construction of our local model $\tilde H$.
This is motivated by our study in
\cite{Hilb} of the relative Hilbert scheme of a node. Let $C_1,...,C_{m-1}$ be copies of $\P^1$, with homogenous
coordinates $u_i,v_i$ on the $i$-th copy. Let $$\tilde{C}\subset
C_1\times...\times C_{m-1}\times B$$ be the subscheme defined by
\beql{Ctilde}{ v_1u_2=tu_1v_2,..., v_{m-2}u_{m-1}=tu_{m-2}v_{m-1}.} Thus
$\tilde{C}$ is a reduced complete intersection of divisors of type
$(1,1,0,...,0), (0,1,1,0,...,0)$ ,..., $(0,...,0,1,1)$ and it is
easy to check that the fibre of $\tilde{C}$ over $0\in B$ is
\beql{Ctilde0}{ \tilde{C}_0=\bigcup\limits_{i=1}^{m-1}\tilde{C}_i, }
where \beq \tilde{C}_i= [1,0]\times...\times[1,0]\times
C_i\times[0,1]\times...\times[0,1] \eeq and that in a neighborhood of the special fibre
$\tilde{C}_0$, $\tilde{C}$ is smooth and $\tilde{C}_0$ is its
unique singular fibre over $B.$ We may embed $\tilde{C}$ in
$\P^{m-1}\times B,$ relatively over $B$ using the plurihomogenous
monomials \beql{Z_i}{ Z_i=u_1\cdots u_{i-1}v_{i}\cdots v_{m-1}, i=1,...,m. }
These satisfy the relations \beql{Z-rel-quadratic}
{ Z_iZ_j=t^{j-i-1}Z_{i+1}Z_{j-1}, i<j-1} so they embed $\tilde{C}$ as a family of rational normal
curves $\tilde{C}_t\subset\P^{m-1}, t\neq 0$ specializing to
$\tilde{C}_0$, which is embedded as a nondegenerate, connected
chain
 of $m-1$ lines.\par
 Next consider an affine
space $\A^{2m}$ with coordinates $a_0,...,a_{m-1},
d_0,...,d_{m-1}$ and let $\tilde{H}\subset\tilde{C}\times\A^{2m}$
be the subscheme defined by
\begin{eqnarray} \label{H-equations} a_0u_1=tv_1, d_0v_{m-1}=tu_{m-1} \\ \nonumber
 a_1u_1=d_{m-1}v_{1},...,a_{m-1}u_{m-1}=d_1v_{m-1}.\end{eqnarray}
Set $L_i=p_{C_i}^*\O(1).$ Then consider the subscheme of
$Y=\tilde{H}\times_{B}U$ defined by the equations
\begin{eqnarray}\label{Y-equations} F_0:=x^m+a_{m-1}x^{m-1}+...+a_1x+a_0\in
\Gamma(Y,\O_Y)
\\ F_1:=u_1x^{m-1}+u_1a_{m-1}x^{m-2}+...+u_1a_2x+u_1a_1+v_1y
\in\Gamma(Y,L_1) \\ \nonumber ...
\\ \nonumber F_i:=u_ix^{m-i}+u_ia_{m-1}x^{m-i-1}+...+u_ia_{i+1}x+u_ia_i+
v_id_{m-i+1}y+...+v_id_{m-1}y^{i-1}+ v_iy^i \\
\in\Gamma(Y,L_i)\\ \nonumber ...
\\ F_m:=d_0+d_1y_1+...+d_{m-1}y^{m-1}+y^m\in
\Gamma(Y,\O_Y). \end{eqnarray} The following statement summarizes results from
\cite{Hilb}.  \begin{thm}\label{Hilb-local}\begin{enumerate}
\item  $\tilde{H}$ is smooth and
irreducible.\item  The ideal sheaf $\I$ generated by
$F_0,...,F_m$ defines a subscheme of $\tilde{H}\times_BX$ that is
flat of length $m$ over $\tilde{H}$
\item The classifying map \beq \Phi=\Phi_\I:
\tilde{H}\to\Hilb_m(U/B) \eeq
is an isomorphism.
\item $\Phi$ induces an isomorphism
$$(\tilde C)_0=p_{\A^{2m}}\inv(0)\to \Hilb^0_m(X_0)=\bigcup\limits_{i=1}^m C^m_i$$
(cf. \cite{Hilb}) of the fibre of $\tilde{H}$ over $0\in\A^{2m}$ with
the punctual Hilbert scheme of the special fibre $X_0$,  in
such a way that the point $[u,v]\in\tilde{C}_i\sim C_i\sim\P^1$
corresponds to\begin{itemize}\item the subscheme $I^m_i(u/v)=(x^{m-i}+(u/v)y^i)\in C^m_i\subset\Hilb^0_m(X_0)$ if $uv\neq 0,$\item the subscheme
$(x^{m+1-i}, y^i)\in C^m_i$ if $[u,v]=[0,1]$,\item  the subscheme $(x^{m-i}, y^{i+1})\in C^m_i$ if
$[u,v]=[1,0].$\end{itemize}
\item  over $U_i,$ a co-basis for the universal ideal $\I$ (i.e. a
basis for $\O/\I$) is given by $$1,...,x^{m-i}, y,...,y^{i-1}.$$
\item $\Phi$ induces an isomorphism of the special fibre
$\td H_0$ of $H$ over $B$ with $\Hilb_m(X_0)$, and $\td H_0\subset\td H$ is a divisor with global normal crossings
$\bigcup\limits_{i=0}^m D^m_i$ where each $D^m_i$ is smooth,
birational to $(x-{\text{axis}})^{m-i}\times (y-\text{axis})^i$, and has special fibre
$C^m_i$ under the cycle map $p_{\A^m}$.
\end{enumerate}
\end{thm}\begin{proof} The smoothness of $\tilde{H}$ is clear
from the defining equations equations and also follows from
smoothness of $\Hilb_m(U/B)$ once (ii) and (iii) are proven. To
that end consider the point $q_i, i=1,...,m,$ on the special fibre
of $\tilde{H}$ over $\A^{2m}_B$ with coordinates \beq v_j=0,\ \forall
j< i; u_j=0,\ \forall j\geq i. \eeq Then $q_i$ has an affine
neighborhood $U_i$ in $\tilde{H}$ defined by \beq U_i=\{ u_j=1, \ \forall j< i;\ v_j=1, \ \forall j\geq i\},\eeq and these
$U_i, i=1,..., m$ cover a neighborhood of the special fibre of
$\tilde{H}.$ Now the generators $F_i$ admit the following
relations:
\beq u_{i-1}F_j=u_jx^{i-1-j}F_{i-1},\ 0\leq j<i-1;\ v_iF_j=v_jy^{j-i}F_i,\
m\geq j>i \eeq where we set $u_i=v_i=1$ for $i=0,m.$ Hence $\I$ is
generated there by $F_{i-1}, F_i$ and assertions (ii), (iii)
follow directly from Theorems 1,2 and 3 of \cite{Hilb} and (iv) is obvious.\par As for (v), it follows immediately from the definition of the $F_i$,  plus
the fact just noted that, over $U_i,$ the ideal $\I$ is generated by
$F_{i-1}, F_i$, and that on $U_i$, we have
$u_{i-1}=v_{i}=1.$ Finally (vi) is contained in \cite{Hilb}, Thm. 2.

\end{proof}
\subsection{Excursions about $H_m$}

 In light of Theorem \ref{Hilb-local}, we
identify a neighborhood $H_m$ of the special fibre in $\tilde{H}$
with a neighborhood of the punctual Hilbert scheme (i.e. $\frak
c_m\inv(mp)$) in $X\sbr m._B$, and note that the projection
$H_m\to \A^{2m}\times B$ coincides generically, hence everywhere,
with $\sigma\circ \frak c_m$. Hence $H_m$ may be viewed as the
subscheme of $\Sym^m(U/B)\times_B\tilde{C}$ defined by the
equations
\begin{eqnarray}\label{sigma-uv-rel}\nonumber \sigma_m^xu_1=tv_1, \\ \sigma_{m-1}^xu_1=\sigma_{1}^yv_{1},...,
\sigma^x_{1}u_{m-1}=\sigma^y_{m-1}v_{m-1},\\
\nonumber
 tu_{m-1}=\sigma^y_mv_{m-1} \end{eqnarray}
Alternatively, in terms of the $Z$ coordinates, $H_m$ may be defined as the subscheme of
$\Sym^m(U/B)\times \P^{m-1}_Z\times B$ defined by the relations
\eqref{Z-rel-quadratic}, which define $\tilde{C}$, together with
\eqspl{sigma-Z-rel}{\sigma_{i}^yZ_i=
\sigma_{m-i}^xZ_{i+1},\
 i=1,...,m-1}
Now having determined the structure of $c_m$ along its 'most
special' fibre $c_m\inv(m(0,0))$, we can easily, and usefully,
determine its structure along other fibres, as follows. For
simplicity we assume for the rest of this subsection that $B$ is a
smooth curve, with local coordinate $t$, and that the singular fibre
$X_0$
has a unique node $p$, with $U$ being a neighborhood of $p$ in
$X$.
  \par Let $X', X"$ denote the $x,y$ axes, respectively
in $U_0=X_0\cap U$, with their respective origins $0', 0"$. If the
special fibre $X_0$ is reducible, then $X', X"$ globalize to the two
components of the normalization (which will be denoted in the same
way if no undue confusion results). If $X_0$ is irreducible, then
both $X'$ and $X"$ globalize to the normalization.
 For any pair of natural numbers $(a,b), 0<a+b< m$,
set $$X^{(a,b)}=X^{'(a)}\times X^{"(b)}$$
(which globalizes to a component (the unique one,
if $X_0$ is irreducible) of the normalization of $X_0^{a+b}$.
 Then we have a
natural map $$X^{(a,b)}\to\Sym^m(U_0)\subset
\Sym^m(U/B)$$ given by
\beq (\sum m_ix_i,\sum n_jy_j)\mapsto \sum
m_i(x_i,0)+\sum n_j(0,y_j)+(m-a-b)(0,0). \eeq  This map
is clearly birational to its image, which we denote by $\bar X^{(a,b)}.$ Thus $X^{(a,b)}$  coincides with the normailzation of $\bar X^{(a,b)}$.
 It is clear that $\bar X^{(a,b)}$ is
defined by the equations \beq \sigma^x_m=...=\sigma^x_{a+1}=0,
\sigma^y_m=...=\sigma^y_{b+1}=0. \eeq
 A  point
$$c\in \bar X^{(a,b)}-(\bar X^{(a+1,b)}\cup \bar X^{(a,b+1)}),$$ i.e. a
cycle in which $(0,0)$ appears
with multiplicity exactly $n=m-a-b$, is said to be of \emph{type $(a,b)$}. Type yields a natural stratification of the symmetric product $X\spr m._0$.
Now let $H^{(a,b)}$ be the closure of the locus of schemes whose cycle if of type $(a,b)$. i.e.
\beql{Hab}{ H^{(a,b)}={\mathrm{closure}}(c_m\inv(\bar X^{(a,b)}-(\bar X^{(a+1,b)}\cup \bar X^{(a,b+1)})))\subset H_m}
Clearly the restriction of $c_m$ on $H^{(a,b)}$ factors through a map
\begin{eqnarray*}\tilde c_m: H^{(a,b)}\to X^{(a,b)},\\
\tilde c_m=((\sigma^x_1, ...,\sigma^x_a), (\sigma^y_1, ...,\sigma^y_b))\end{eqnarray*}
Approaching the 'origin cycle' $m(0,0)$ through
cycles of type $(a,b)$, i.e. approaching the point $(a0', b0")$
on $X^{(a,b)}$, means that $a$ (resp. $b$) points are approaching the origin $0'$ (resp. $0"$) along the $x$ (resp. $y$)-axis.  For a cycle $c$ of type $(a,b)$, we have, for all $j\leq b$, that
$\sigma^y_j\neq 0, \sigma_{m-j}^x=0$,  hence
by the equations \eqref{H-equations}
(setting each $a_i=\sigma^x_{m-i},
d_i=\sigma^y_{m-i}$), we conclude $v_j=0$; thus
\beql{v=0}{ v_1=...=v_{b}=0; } similarly, for all $j\leq a$,
we have $\sigma_{m-j}^y=0, \sigma_j^x\neq 0$, hence
again by the equations \eqref{H-equations} , we conclude $u_{m-j}=0$;
thus
\beql{u=0}{  u_{m-1}=...=u_{m-a}=0. }
Consequently, the fibre of $\frak c_m $ over this point is schematically
\beql{cm fibre}{c_m\inv(c)=\tilde c_m\inv(c)= \bigcup\limits_{i=b+1}^{m-a-1} C^m_i, } provided
$a+b\leq m-2$. If $a+b=m-1$, the fibre is the unique point given by
\beq v_1=...=v_b=u_{b+1}=...=u_{m-1}=0  \eeq (this is the point denoted
$Q^m_{b+1}$ in \cite{Hilb}, i.e. the subscheme with ideal $(x^{m-b}, y^{b+1})$). As $c$ approaches the 'origin' $(a0',b0")$ in $X^{(a,b)}$, the equations \refp{v=0}.,\refp{u=0}. persist, so we conclude
\beql{tilde cm fibre} {\tilde c_m\inv((a0',b0"))=\begin{cases}\bigcup\limits_{i=b+1}^{m-a-1} C^m_i, a+b\leq m-2,\\
Q^m_{b+1}, a+b=m-1.
\end{cases}
}
Thus, working  in $H^{(a,b)}$ over $X^{(a,b)}$, the special
fibre is the same as the general fibre.
Moreover as subscheme of $H_m\times_{X\spr m.}X^{(a,b)}$,
$H^{(a,b)}$ is defined by the equations \eqref{u=0} and
\eqref{v=0}.
 And in the special case $a+b=m-2$, we see that $H^{(a,b)}$ forms a $\P^1$-bundle with fibre $C^m_{b+1}$, locally near $(a0', b0")$. This is a
so-called \emph{node scroll}, to be discussed further below.\par
Incidentally, in case $a+b=m$, a similar but simpler analysis
shows that the fibre $\tilde c_m\inv((a0',b0"))$ coincides with
$C^m_b\simeq\P^1$ if $1\leq b\leq m-1$ and with the single point $Q^m_{b+1}$ if $b=0,m$. This, of course, is contained in part (vi) of Theorem \ref{Hilb-local} above.
\par
Summarizing this discussion, a bit more usefully, in terms of $Z$ coordinates,
we have
\begin{cor} For any $\l_1+\l_2\leq m, \l_1, \l_2\geq 0,$
we have, in any component of
 the locus $c_m\inv((X')\spr \l_1.\times (X")\spr \l_2.)
\subset H_m$ dominating $(X')\spr \l_1.\times (X")\spr \l_2.$, that
\eqspl{}{
Z_i=0,& \forall i\geq\max(m-\l_1+1, \l_2+2),\\
& \forall i\leq \min(\l_2, m-\l_1-1).
} In particular, if $\l_1+\l_2=m$ (resp. $\l_1+\l_2<m$), the
 only nonzero $Z_i$
are where $$i\in [\l_2,\l_2+1]\cap [1,m] \ \  (\textrm{resp.}
\ \  i\in[l_2+1, m-\l_1]\cap [1,m]).$$
\end{cor}
Taking to account the linear relations \eqref{sigma-Z-rel}, we
also conclude
\begin{cor}
\begin{enumerate}\item
For any component $(X')^\l\times(X")^{m-\l}, 0<\l<m$ of the
special fibre of $U\spr m._B$, the unique dominant component of
$c_m\inv((X')^\l\times(X")^{m-\l})$ coincides with
the graph of
rational map
\eqspl{}{
(X')^\l\times(X")^{m-\l}\dashrightarrow\P^1_{Z_\l, Z_{\l+1}}\subset\P^{m-1}_Z} defined by
\[
[\sigma_\l^x, \sigma_{m-\l}^y];
\]
\item ditto for $\l=m$ (resp. $\l=0$), with the constant
rational map
to $[1,0,...]$ (resp. $[...,0,1]$);
\item ditto over $(X')^\l\times(X")^{m-\l-1}, 0\leq\l\leq m-1$,
with the constant map to $[...,0,1_{m-\l},0...]$.
\end{enumerate}\end{cor}
\par
Now on the other hand, working near a cycle $c$ of type $(a,b)$
and fixing its off-node portion, say of length $k=m-n$,
 we also have an obvious identification
of the same (general) fibre of $H^{(a,b)}/X^{(a,b)}$ over $c$ as the special fibre
in a local model $H_n$ for the length-$n$ Hilbert scheme.
Namely,
if we let $c'=n(0,0)$ be the part of $c$ supported at the origin, then essentially the same fibre
$c_m\inv(c)$ can also be written as
\beql{cn fibre}{c_n\inv(c')=\bigcup\limits_{j=1}^{n} C^n_j}
and naturally $C^n_j$ corresponds to $C_{j+b}^m=C_{j+b}^{n+a+b}$.
Of course under the identification of Theorem \ref{Hilb-local}, $c_n\inv(c')$ corresponds to the punctual Hilbert scheme $Hilb^0_n(X_0)$. So we conclude that the $j$-th punctual
length-$n$ Hilbert scheme component at $c$ specializes to the $(j+b)$-th length-$m$ Hilbert scheme component at $m(0,0)$
as $c$ specializes to $m(0,0)$ over the normalization $X^{(a,b)}$.
Note that the analogous fact holds for any cycle of multiplicity $n$ at $(0,0)$ specializing to one of multiplicity $m$ at
$(0,0)$, even if its total degree is higher. Thus we have
\begin{lem} As a cycle $c$ on $X_0$, having multiplicity $n$ at the origin, approaches a cycle $d$ with multiplicity
$m>n$ at the origin, so that for some $a,b$ with
$a+b=m-n$, $a$ points approach along the $x$-axis and $b$ points along the $y$-axis, the punctual Hilbert scheme component $C^n_j$ over $c$ specializes smoothly to $C^m_{j+b}$ at $d$.
\end{lem}
  We also see, comparing \refp{cm fibre}. and \refp{tilde cm fibre}., that the fibre $\tilde c_m\inv(c)$ is 'constant', i.e. it doesn't depend on $c$ as it moves in $X^{(a,b)}$. Moreover, as $c$ moves in $X^{(a,b)}$,
the individual components of this fibre, which have to do with
branches of $X_0^{(m-a-b)}$ at $(m-a-b)(0,0)$, $X_0$ being the
entire singular fibre (or what is the same, branches of $X_0^m$
generically along $X^{(a,b)}$), remain well defined (i.e. not
interchanged by monodromy), and specialize smoothly to
similar components on lower-dimensional strata. Note that this is true even if $X_0$ is (globally) irreducible, in which case
the other $a+b$ branches of $X_0$, $a$ from $X'$ and $b$ from
$X"$ are globally interchangeable.
Therefore:
\begin{lem}\label{Hab-struct} Notation as in \refp{Hab}.
\emph{et seq.}, we have
\beql{bundle}{H^{(a,b)}=\bigcup\limits_{j=1}^{n-1} F_j^{(a|b)}}
where $F_j^{(a,b)}\subset H_m$ is the subscheme defined by
\eqspl{uv-eq-nodescroll}{
v_1=...=v_{j+b}=u_{j+b+1}=....=u_{m-1}=0
} and
 $$\tilde c_m:F_j^{(a|b)}\to{X}^{(a,b)}$$
is a $\P^1$ bundle with general fibre $C_j^{n}$.
Moreover, $\tilde c_m|_{F^{(j:a|b)}}$ admits two disjoint sections with respective general fibres the points corresponding to the
punctual schemes of type
$Q^n_{j}, Q^n_{j+1}$.\end{lem}
Fixing $a,b$ for now, the $F_j=F_j^{(a|b)}$ are special (but typical) cases of what are called \emph{node scrolls}.
It follows from the lemma that we can write
$$F_j=\P(L^n_{j}\oplus L^n_{j+1})$$
for certain line bundles $L^n_j$ on ${X}^{(a,b)}$,
corresponding to the disjoint sections $Q^n_j, Q^n_{j+1}$,
where the
difference $L_j^n-L^n_{j+1}$ is uniquely determined (we use additive notation for the tensor product of line bundles
and quotient convention for projective bundles). The identification of a natural choice for both these line bundles,
using methods to be developed later in this section,
 will be taken up in the next section and plays an important role in the enumerative
geometry of the Hilbert scheme. But the difference
$L_j^n-L^n_{j+1}$, and hence the intrinsic structure of
the node scroll $F_j$, may already be computed now,  as follows.\par Write
$$Q_j=\P(L_j), Q_{j+1}=\P(L_{j+1})$$ for the two special sections of type $Q^n_j, Q^n_{j+1}$ respectively. Let
$$D_{0'}, D_{0"}\subset X^{(a,b)}$$ be the divisors
comprised of cycles containing $0'$ (resp. $0"$).
In the local model, these are given locally by the
respective equations
$$D_{0'}=(\sigma^x_a), D_{0"}=(\sigma^y_b).$$
\begin{lem}\label{nodescroll0} We have, using the quotient convention for
projective bundles,
\beql{nodescroll-1}{F_j=\P_{X^{(a,b)}}(\O(-D_{0'})\oplus \O(-D_{0"})), j=1,...,n-1.}\end{lem}
\begin{proof} Our key tool is a $\C^*$- parametrized family
of sections
'interpolating' between  $Q_j$
and $Q_{j+1}$. Namely,
note that for any $s\in\C^*$, there is a well-defined
section $I_s$ of $F_j$ whose fibre over a general point
$z\in X^{(a,b)}$ is the scheme
$$I_s(z)=(sx^{n-j}+y^j)\coprod \sch(z),$$
where $\sch(z)$ is the unique subscheme of length $a+b$, disjoint from the nodes, corresponding to $z$.\par
\emph{Claim:} The fibre of $I_s$ over a point $z\in D_{0'}$
(resp. $z\in D_{0"}$) is
$Q^n_j$ (resp. $Q^n_{j+1}$).\par
\emph{Proof of claim}. Indeed set-theoretically the claim is clear from the
fact the this fibre corresponds to a length-$n$ punctual scheme
meeting the $x$-axis (resp. $y$-axis) with multiplicity at least $n-j+1$ (resp. $j+1$).\par
To see the same thing schematically, via equations in the local model $H_{n+1}$, we proceed as follows. Working near
a generic point $z_0\in D_{0'}$ we can, discarding distal
factors supported away from the nodes, write
the singleton scheme corresponding nearby cycle $z$ as  $\sch(z)=(x-c,y)$ where $c\to 0$ as $z\to z_0$, and then
$$I_s(z)=(sx^{n-j}+y^j)(x-c,y)=(sx^{n-j+1}-csx^{n-j}-cy^j, y^{j+1}).$$ Thus, in terms of the system of generators \refp{Y-equations}., $I_s(z)$ is defined locally by \beql{I_s-equation}{cu_j-sv_j=0} (with other $[u_k,v_k]$ coordinates either
$[1,0]$ for $k<j$ or $[0,1]$ for $k>j$.
The limit of this as $c\to 0$ is $[u_j,v_j]=[1,0]$, which is the point $Q_j$. \emph{QED Claim.}\par
 Clearly $I_s$ doesn't meet $Q_j$ or $Q_{j+1}$ away from $D_{0'}\cup D_{0"}$. Therefore, we have
\beql{}{ I_s\cap Q_j=Q_j.D_{0'},}
\beql{}{I_s\cap Q_{j+1}=Q_{j+1}.D_{0"};}
an easy calculation in the local model shows that the
intersection is transverse.
Because $Q_j\cap Q_{j+1}=\emptyset$, it follows that
\begin{eqnarray}I_a\sim Q_j+D_{0'}.F_j\\
\sim Q_{j+1}+D_{0"}.F_j.\end{eqnarray}
These relations also follow from the fact,
which comes simply from setting $s=0$ or
dividing by $s$ and setting $s=\infty$ in \refp{I_s-equation}., that
\beql{limits}{ \lim\limits_{s\to 0}I_s=Q_j+D_{0'}.F_j, \lim\limits_{s\to\infty} I_s= Q_{j+1}+D_{0"}.F_j}
It then  follows that
$$(Q_j)^2=Q_j.(I_s-D_{0'}.F_j)=Q_j.(Q_{j+1}+(D_{0"}-D_{0'}).F_j),$$
hence
\begin{eqnarray}(Q_j)^2=Q_j(D_{0"}-D_{0'}),\end{eqnarray}
therefore finally
\beql{nodescroll-0}{L_j^n-L^n_{j+1}=D_{0"}-D_{0'}.}
This proves the Lemma.\end{proof}

\newsubsection{Globalization}\label{globalization}
We now wish to extend the discussion of the last subsection,
in paticular the notion of node scrolls, to the
general case, with higher-dimensional base and fibres with
more than one node: in this case
a node scroll becomes a $\P^1$-bundle over a relative Hilbert
scheme associated to a 'boundary family' of $X/B$, i.e a
family obtained essentially as the partial normalization of the subfaily of $X/B$ lying over the normalization of a component of the locus of singular curves in $B$ (a.k.a. the boundary of $B$).\par
 To this end let $\pi:X\to B$ now denote an arbitrary flat
 family of
nodal curves of arithmetic genus $g$ over an irreducible base, with smooth
generic fibre. In order to specify the additional information
required to define a node scroll, we make
the following definition.
\begin{defn}
A \underline{boundary datum} for $X/B$ consists of\begin{enumerate}
\item an irreducible variety $T$ with a map  $\delta:T\to B$
unramified
to its image;
\item a lifting $\theta:T\to X$ of $\delta$ such that each
$\theta(t)$ is a node of $X_{\delta(t)}$;
\item a labelling, continuous in $t$, of the two branches
of $X_{\delta(t)}$ along $\theta(t)$ as $x$-axis and $y$-axis.

\end{enumerate}
Given such a datum, the \underline{associated boundary family
} $X^\theta_T$ is the
normalization (= blowup) of the base-changed family $X\times_BT$
along the section $\theta$, i.e.
\[ X^\theta_T=\Bl_\theta(X\times_BT),
\] viewed as a family of curves of genus
$g-1$ with two, everywhere distinct, individually defined
 marked points $\theta_x, \theta_y$. We denote by $\phi$ the
 natural map fitting in the diagram  $$
 \begin{matrix}
 X^\theta_T&&\\
 \downarrow&\stackrel{\phi}{\searrow}&\\
 X\times_BT&\to&X\\
 \downarrow&&\downarrow\\
 T&\stackrel{\delta}{\to}&B.
 \end{matrix}
 $$

\end{defn}
Note that any component $T_0$ of the boundary of $B$, i.e.
the (divisorial) locus of
singular fibres, gives rise to (finitely many) boundary data in this
sense: first consider a component $T_1$ of the
normalization of $T_0\times_B\mathrm{sing}(X/B)$, which already
admits a node-valued lifting $\theta_1$ to $X$, then further base-change
by the normal cone of $\theta_1(T_1)$ in $X$
(which is 2:1 unramified, possibly disconnected, over $T_1$),
to obtain a boundary datum as above. 'Typically', the curve corresponding to a general
point in $T_0$ will have a single node $\theta$ and then the degree of $\delta$ will be 1 or 2 depending on whether the branches along $\theta$ are distinguishable in $X$ or not (they always are distinguishable if $\theta$ is a separating node
and the separated subcurves have different genera). Proceeding
in this way and taking all components which arise, we obtain finitely many boundary data which 'cover', in an obvious sense,
the entire boundary of $B$. Such a collection, weighted so that
each boundary component $T_0$ has total weight $=1$ is called a
\emph{covering system of boundary data}.
\par
\begin{propdef}
Given a boundary datum $(T, \delta, \theta)$ for $X/B$ and
natural numbers $1\leq j<n$, there exists a $\P^1$-bundle $F^n_j(\theta)$,
called a \underline{node scroll} over
the Hilbert scheme $(X^\theta_T)\sbr m-n.$,
endowed with two disjoint sections $Q^n_{j,j}, Q^n_{j+1,j}$,
 together with a surjective
map generically of degree equal to $\deg(\delta)$
of
$$\bigcup\limits_{j=1}^{n-1} F^n_j(\theta):=
\coprod\limits_{j=1}^{n-1} F^n_j(\theta)/
\coprod\limits_{j=1}^{n-2} (Q^n_{j+1,j}\sim Q^n_{j+1, j+1})$$
onto
the closure in $X\sbr m._B$ of the locus of
schemes having length precisely $n$ at $\theta$, so that
a general fibre of $F^n_j(\theta)$ corresponds to the family $C^n_j$
of length-$n$ schemes at
$\theta$ generically of type $I^n_j(a)$, with the two nonprincipal schemes $Q^n_j, Q^n_{j+1}$ corresponding to
$Q^n_{j,j}, Q^n_{j,j+1}$ respectively.
We denote by $\delta^n_j$ the natural map of $F^n_j(\theta)$
to $X\sbr m._B$.
\end{propdef}\begin{proof}[Proof-construction]
The scroll $F_j^n(\theta)$ is defined as follows. Fixing
the boundary data, consider first the locus \[\bar{ F}_j^n \subset T\times_B X\sbr m._B\]  consisting of
compatible pairs $(t,z)$ such that $z$ is in the closure of
the set of schemes which are of type $I^n_j$ (i.e. $x^{n-j}+ay^j, a\in \C^*$) at $\theta(t)$, with
respect to the branch order $(\theta_x, \theta_y)$.
The discussion of the previous subsection shows that
the general fibre of $\bar F_j$ under the cycle map is a
$\P^1$, namely a copy of $C^n_j$; moreover the closure
of the locus of schemes having multiplicity $n$ at $\theta$
is the union $\bigcup\limits_{j-1}^{n-1}\bar F^n_j.$
In fact locally over a cycle having multiplicity precisely $n+e$ at $\theta$, $\bar F^n_j$
is a union of components $\bar F^{(n:a,b)}_j, a+b=e$, where $\bar F^{(n:a,b)}_j$ maps to $(X')^a\times (X")^b$ and is defined in the local model $H_{n+e}$ by
 is defined by the vanishing of
all $Z_i, i\neq j+b, j+b+1$ or alternatively, in terms of $u,v$ coordinates, by
\[ v_1=...=v_{j+b}=u_{j+b+1}=...=u_{n+e}=0\]\par
Then $F^n_j(\theta)$ is the locus \eqspl{}{
\{(w,t,z)\in (X^\theta_T)\sbr m-n.\times_T\bar{F}^n_j:
\phi_*(c_{m-n}(w))+n\theta=c_m(z),\}
}
where $\phi:X^\theta\to X$ is the natural map, clutching together $\theta_x$ and $\theta_y$, and $\phi_*$ is the induced push-forward map on cycles. Then the results of the previous section show that  $F^n_j(\theta)$  is
locally defined near a cycle
having multiplicity $b$ at $\theta_y$, e.g. by the vanishing
of the $Z_i, i\neq j+b, j+b+1$ on
$$\{(w,u,Z)\in (X^\theta_T)\sbr m-n.\times X\spr e._B\times\P^{n+e}:\phi_*(c_{m-n}(w))_\theta+n\theta=u$$
where $._\theta$ indicates the portion near $\theta$.
The latter locus certainly projects isomorphically to its image
in $(X^\theta_T)\sbr m-n.\times\P^{n+e}$, hence $F^n_j(\theta$ is a $\P^1$-bundle over $(X^\theta_T)\sbr m-n.$.
Since $F^n_j(\theta)$
admits the two sections $Q^n_{j,j}, Q^n_{j+1,j}$, it is the projectivization of a decomposable rank-2
vector bundle.
\end{proof}
In addition to the node scroll $F^n_j(\theta)$,
we will also consider its ordered version,
i.e. \eqspl{}
{OF^n_j(\theta)=F^n_j(\theta)\times_{(X^\theta_T)^{(m-n)}}
(X^\theta_T)^{m-n},
} and and similarly for $\bar{OF}^n_j(\theta)$. Also, for each $n$-tuple $I\subset [1,m]$, the corresponding
locus in $X\scl m._B$, i.e.
\eqspl{}{
OF^I_j=\{(w,t,z)\in (X^\theta_T)\scl m-n.\times_T\bar{OF}^n_j:
\phi_*(oc_{m-n}(w))+\sum\limits_{i\in I}
p_i^*(\theta)=c_m(z)\},
} this being the 'node scroll inserted over the $I$-indexed
coordinates.

\newsubsection{Reverse engineering}
 Our task now is effectively
to 'reverse-engineer' an ideal in the $\sigma$'s whose syzygies
are given by \eqref{sigma-Z-rel} and \eqref{Z-rel-quadratic} . To this end, it is convenient to
introduce order in the coordinates. Thus let
$OH_m=H_m\times_{\Sym^m(U/B)}U^m_B$, so we have a cartesian
diagram \beq \begin{matrix} OH_m&\stackrel {\varpi_m}{\longrightarrow}&
H_m\\ o\frak c_m \downarrow &\square&\downarrow \frak c_m\\
X^m_B&\stackrel{ \omega_m}{\longrightarrow}& X\spr
m._B\end{matrix}
 \eeq and its global analogue
\beq\begin{matrix} X^{\lceil m\rceil}_B&\stackrel {\varpi_m}{\longrightarrow}&
X\sbr m._B\\ o\frak c_m \downarrow &\square&\downarrow \frak c_m\\
X^m_B&\stackrel{ \omega_m}{\longrightarrow}& X\spr m._B\end{matrix}\eeq

Note that $ X\spr m._B$ is normal and Cohen-Macaulay: this follows from the fact that it is a quotient by
$\frak S_m$ of $X^m_B$, which is a locally complete intersection
with singular locus of codimension $\geq 2$ (in fact, $>2$, since
$X$ is smooth). Alternatively, normality of $ X\spr m._B$ follows
from the fact that $H_m$ is smooth and the fibres of \nl$\frak
c_m:H_m\to
 X\spr m._B$ are connected (being products of connected chains of
rational curves). Note that $\omega_m$ is simply ramified
generically over $D^m$ and we have \beq \omega_m^*(D^m)=2OD^m  \eeq where
\beq OD^m=\sum\limits_{i<j}D^m_{i,j} \eeq where
$D^m_{i,j}=p_{i,j}\inv(OD^2)$ is the locus of points whose $i$th and
$j$th components coincide.
 To prove $\frak c_m$ is equivalent to the
blowing-up of $D^m$ it will suffice to prove that $o\frak c_m$ is
equivalent to the blowing-up of $2OD^m=\omega_m^* (D^m)$ which in
turn is equivalent to the blowing-up of $OD^m.$ 
The advantage of working with $OD^m$ rather than its unordered
analogue is that at least some of its equations are easy to write
down: let \beq v^m_x=\prod_{1\leq i<j\leq m}(x_i-x_j), \eeq and likewise
for $v^m_y.$ As is well known, $v^m_x$ is the determinant of the Van
der Monde matrix
\beq V^m_x=\left [\begin{matrix} 1&\ldots&1\\x_1&\ldots&x_m\\\vdots&&\vdots\\
x_1^{m-1}&\ldots&x_m^{m-1}\end{matrix} \right ]. \eeq
 Also
set
\beq \tilde{U}_i=\varpi_m\inv(U_i), \eeq where $U_i$ is as in (1.3.7),
being a neighborhood of $q_i$ on $H_m$. Then in $U_1$, the
universal ideal $\I$ is defined by
\beq F_0, \ \ F_1=y+(\text{function of }x) \eeq and consequently the
length-$m$ scheme corresponding to $\I$ maps isomorphically to its
projection to the $x$-axis. Therefore over
$\tilde{U}_1=\varpi_m\inv(U_1),$ where $F_0$ splits as $\prod
(x-x_i),$ the equation of $OD^m$ is simply given by \beq G_1=v^m_x. \eeq
Similarly, the equation of $OD^m$ in $\tilde{U}_m$ is given by
\beq G_m=v^m_y. \eeq New let \beq \Xi:OH_m\to\P^{m-1} \eeq be the morphism
corresponding to $[Z_1,...,Z_m]$, and set $L=\Xi^*(\O(1)).$ Note
that $\tilde{U}_i$ coincides with the open set where $Z_i\neq 0$,
so $Z_i$ generates $L$ over $\tilde{U}_i.$ Let
\beq O\Gamma\spr m.=o\frak c_m\inv(OD^m). \eeq
This is a $1/2-$Cartier divisor because $2O\Gamma\spr
m.=\varpi_m\inv(\Gamma\spr m.)$ and $\Gamma\spr m.$ is Cartier,
$H_m$ being smooth. In any case, the ideal $\O(-O\Gamma\spr m.)$ is
a divisorial sheaf (reflexive of rank 1).
 Our aim now is to construct an isomorphism
\beql{gamma}{ \gamma:\O(-O\Gamma\spr m.)\to L.}
As we shall see, this isomorphism
will easily imply Theorem 1. To construct $\gamma,$ it suffices
to specify it on each $\tilde{U}_i.$
\newsubsection{ Mixed Van der
Mondes and conclusion of proof}  A clue as to how the
latter might be
done comes from the relations (1.4.2-1.4.3). Thus, set
\beql{Gdef1}{ G_i=\frac{(\sigma_m^y)^{i-1}}{t^{(i-1)(m-i/2)}}v^m_x=
\frac{(\sigma_m^y)^{i-1}}{t^{(i-1)(m-i/2)}}G_1,\ \ i=2,...,m.} Thus,  \beql{Gdef2}{ G_2=\frac{\sigma^y_m}{t^{m-1}}G_1,
G_3=\frac{\sigma^y_m}{t^{m-2}}G_2,...,
G_{i+1}=\frac{\sigma^y_m}{t^{m-i}}G_i, i=1,...,m-1.}

 An elementary calculation shows that if we denote by
$V^m_i$ the 'mixed Van der Monde' matrix
\beq V^m_i=\left[\begin{matrix} 1&\ldots&1\\x_1&\ldots&x_m\\\vdots&&\vdots\\
x_1^{m-i}&\ldots&x_m^{m-i}\\y_1&\ldots&y_m\\\vdots&&\vdots\\
y_1^{i-1}&\ldots&y_m^{i-1}\end{matrix}\right]\eeq then we
have
\beql{vandermonde}{ G_i=\pm\det (V^m_i), i=1,...,m.}
\begin{minipage}[0pt,50pt]{450pt}
$\ulcorner$ Indeed for $i=1$ this is standard; for general $i$, it suffices
to prove the analogue of \refp{Gdef2}. for the mixed Van der Monde determinants. For this, it suffices to multiply each $j$th column of $V^m_i$ by $y_j$,
and factor a $t=x_jy_j$ out of each of rows  $2,...,m-i+1$,
 which yields
\beql{vandermonde1}{\sigma^y_m\det(V^m_i)=(-1)^{m-i+1}t^{m-i}V^m_{i+1}.}
$\llcorner$\end{minipage}\par
    From \refp{vandermonde}. it follows, e.g., that $G_m$ as given in \refp{Gdef1}. coincides with
$v^m_y.$ Next we claim \begin{lem}
\label{Gi-generates-on-Ui} $G_i$ generates $\O(-O\Gamma\spr m.)$ over
$\tilde{U}_i.$\end{lem}\begin{proof}[Proof of Lemma] This is clearly true where $t\neq 0$ and it remains
to check it along the special fibre $OH_{m,0}$ of $OH_m$ over $B$.
Note that $OH_{m,0}$ is a sum of components of the form
\beql{Theta} {\Theta_I=\text{Zeros}(x_i,i\not\in I, y_i, i\in I),
 I\subseteq\{1,...,m\}, } none of which is contained in the
 singular locus of $OH_m.$
  Set
 \beq \Theta_i=\bigcup\limits_{|I|=i}\Theta_I. \eeq Note that
 \beq \Tilde{C}_i\times 0\subset \Theta_i, i=1,...,m-1 \eeq and therefore
 \beq \tilde{U}_i\cap\Theta_j=\emptyset, j\neq i-1, i. \eeq
 Note that $y_i$ vanishes to order 1 (resp. 0) on $\Theta_I$
 whenever $i\in I$ (resp. $i\not\in I$). Similarly, $x_i-x_j$
 vanishes to order 1 (resp. 0) on $\Theta_I$ whenever both $i,j\in
 I^c$ (resp. not both $i,j\in I^c$). From this, an elementary
 calculation shows that the vanishing order of
 $G_j$ on every component $\Theta$ of
 $\Theta_k$ is \beql{ordG}{ \ord_{\Theta}(G_j)=(k-j)^2+(k-j).}
 We may unambiguously denote this number by $\ord_{\Theta_k}(G_j)$.
 Since this order is nonnegative for all $k,j,$
 it follows firstly that
 the rational function $G_j$ has no poles, hence is in fact regular
 on $X^m_B$ near
 $mp$ (recall that $X^m_B$ is normal); of course,
 regularity of $G_j$ is also immediate from \refp{vandermonde}..
  Secondly, since this order is zero for $k=j, j-1$, and
 $\Theta_j, \Theta_{j-1}$ contain all the components of $OH_{m,0}$
  meeting
 $\tilde{U}_j$, it follows that in $\tilde{U}_j,$ $G_j$ has no
 zeros besides $O\Gamma\spr m.\cap\tilde{U}_j,$ so $G_j$ is a generator
 of $\O(-O\Gamma\spr m.)$ over $\tilde{U}_j.$ QED Lemma.
  \end{proof}
  The Lemma yields a set of generators for the ideal of $OD^m$:
  \begin{cor}[of Lemma]\label{Gs-generate-bigdiag}
  The ideal of $OD^m$ is generated, locally
 near $p^m$, by $G_1,...,G_m.$  \end{cor}\begin{proof}
  If $Q$ denotes the cokernel of the map
  $m\O_{X^m}\to \O_{X^m}(-{OD^m})$ given by $G_1,...,G_m$,
  then $c_m^*(Q)=0$ by the Lemma,
   hence $Q=0$, so the $G$'s generate $\O_{X^m}(-{OD^m})$.
  \end{proof}
  Now we can construct the desired isomorphism $\gamma$ as in
  \eqref{gamma}, as
   follows. Since $Z_j$ is a
 generator of $L$ over $\tilde{U}_j,$ we can define our
 isomorphism $\gamma$ over $\tilde{U}_j$ simply by specifying that
 \beq \gamma (G_j)=Z_j\  {\text{on}}\ \tilde{U}_j. \eeq Now to check that
 these maps are compatible, it suffices to check that
 \beq G_j/G_k=Z_j/Z_k \eeq as rational functions (in fact, units over
 $\tilde{U}_j\cap\tilde{U}_k$). But the ratios $Z_j/Z_k$ are
 determined by the relations \refp{sigma-Z-rel}., while $G_j/G_k$
 can be computed
 from \refp{Gdef2}., and it is trivial to check that these agree. \par
 Now we can easily complete
 the proof of Theorem 1. The existence of $\gamma$, together with the universal property of blowing up, yields a morphism
 \[ Bc_m:OH_m\to B_{OD^m}X^m_B\]
 which is clearly proper and birational, hence surjective.
 On the other hand, the fact that the $G$'s generate the ideal
 of $OD^m$, and correspond to the $Z$ coordinates on $OH_m\subset X^m_B\times \P^{m-1}$, implies that $Bc_m$
 is a closed immersion. Therefore $Bc_m$ is an isomorphism.
 \qed
 \begin{cor}\label{sym-eq} The image of the
 relative symmetric product $X\spr m._B$ under the elementary
 symmetric functions embedding $\sigma$ (cf. Lemma \ref{sigma-emb}) is schematically defined by the equations (\ref{sigma-x-y-rel}-\ref{sigma-x-y-rel2}).\end{cor}
 \begin{proof}
 We have a diagram locally
 \beql{}{\begin{matrix}
 H^m&\subset& \P^{m-1}\times\A^{2m}\times B\\
 \downarrow&&\downarrow\\
 X\spr m._B&\stackrel{\sigma}{\hookrightarrow}&\A^{2m}\times B.
 \end{matrix}
}
We have seen that the image of the top inclusion is defined by the equations \eqref{Z-rel-quadratic}, \eqref{sigma-Z-rel}.
 The equations of the schematic image of $\sigma$ are obtained
 by eliminating the $Z$ coordinates from the latter equations,
 and this clearly yields the equations as claimed.
 \end{proof}
 Now as one byproduct of the proof of Theorem \ref{blowup}, we obtained generators of the ideal of the ordered half-discriminant $OD^m$.
As a further consequence, we can determine the ideal of the
 discriminant locus $D^m$ in the symmetric product $X\spr m._B$ itself: let $\delta_m^x$
 denote the discriminant of $F_0$, which, as is
 well known \cite{L}, is a polynomial in the $\sigma_i^x$
 such that \beql{}{ \delta_m^x=G_1^2.} Set
 \beql{eta}{\eta_{i,j}=\frac{(\sigma_m^y)^{i+j-2}}
 {t^{(i-1)(m-i)+(j-1)(m-j)}}
 \delta^m_x.} It is easy to see that this is a polynomial in the $\sigma^x_.$ and the $\sigma^y_.$, such that $\eta_{i,j}=G_iG_j.$
 \begin{cor} The ideal of $D^m$ is generated, locally
 near $mp$, by $\eta_{i,j}, i,j =1,...,m.$\end{cor}\begin{proof} This
 follows from the fact that $\varpi_m$ is flat and that
 \beq \varpi_m^*(\eta_{i,j})=G_iG_j, i,j=1,...,m \eeq generate the ideal of
 $2OD^m=\varpi_m^*(D^m).$\qed\end{proof}
 \emph{Proof of Lemma \ref{purity}}. Consider the function $\eta_{i,i}$, a
 priori a rational function on $X\spr m._B$. Because it pulls back
 to the regular function $G_i^2$ on $X^{\lceil m\rceil}_B$, it
 follows that $\eta_{i,i}$ is in fact regular near the 'origin'
 $mp$. Clearly $\eta_{i,i}$ vanishes on the discriminant $D^m$. Now
 the divisor of $\eta_{i,i}$ on $U_i$ pulls back via $\varpi_m$ to
 the divisor of $G_i^2$, which also coincides with
 $\varpi_m^*(\Gamma\spr m.)$. Because $\varpi$ is finite flat, it
 follows the the divisor of $\eta_{i,i}$ coincides over $U_i$ with
 $\Gamma\spr m.$ and in particular, the pullback of the ideal of
 $D^m$ has no embedded component in $U_i$. Since the $U_i,
 i=1,...,m$ cover a neighborhood of the exceptional locus in $X\sbr
 m._B$, this shows $\frak c_m\inv(D^m)$ has no embedded components, i.e.
 is Cartier, as claimed. The reader can check that the foregoing
 proof is logically independent  of any results that depend on
 the statement of Lemma \ref{purity}, so there is no vicious circle.\qed

 Note that the ideal of the Cartier divisor $\frak c_m^*(D^m)$ on $X\sbr m._B$, that is,
 $\O_{X\sbr m._B}(-\frak c_m^*(D^m))$,
 is isomorphic in terms of our local model $\tilde{H}$ to
 $\O(2)$ (i.e. the pullback of $\O(2)$ from $\P^{m-1}$).
 This suggests that $\O(-\frak c_m^*(D^m))$
  is divisible by 2 as line bundle on $X\sbr m._B$,
 as the following result indeed shows. First some notation.
For a prime divisor $A$ on $X$, denote by $[m]_*(A)$ the prime
divisor on $X\sbr m._B$ consisting of schemes whose support meets
$A$. This operation is easily seen to be additive, hence can be
extended to arbitrary, not necessarily effective, divisors and
thence to line bundles.
 \begin{cor} Set \beql{O(1)}{\O_{X\sbr
 m._B}(1)=\omega_{X\sbr m._B}\otimes[m]_*(\omega_X\inv).} Then

 \beql{discr1}{ \O_{X\sbr m._B}(-\frak c_m^*(D^m))\simeq\O_{X\sbr m._B}(2)}
 and
 \beql{discr2}{\O_{X^{\lceil m\rceil}_B}(-o\frak c_m^*(OD^m))
 \simeq\varpi_m^*\O_{X\sbr m._B}(1).}

 \end{cor}

 \begin{proof} The Riemann-Hurwitz formula shows that the isomorphism
 \refp{discr1}. is valid on
 the open subset of $X\sbr m._B$ consisting of schemes disjoint from
 the locus of fibre nodes of $\pi$. Since this open is big (has
 complement of codimension $>1$), the iso holds on all of $X\sbr
 m._B$. A similar argument establishes \ref{discr2}..\end{proof}
 In practice, it is convenient to view \refp{O(1)}.
 as a formula for $\omega_{X\sbr m._B},$
 with the understanding that $\O_{X\sbr m._B}(1)$ coincides in our
 local model with the $\O(1)$ from the $\P^{m-1}$ factor, and that
 it pulls back over $X^{\lceil m\rceil}_B=X\sbr m._B\times_{X\spr
 m._B}X^m_B$ to the $\O(1)$ associated to the blow up of the
 'half discriminant' $OD^m.$ We will also use the notation
 \beq\O(\Gamma\spr m.)=\O_{X\sbr m._B}(-1),
 \Gamma\scl m.=\varpi_m^*(\Gamma\spr m.)\eeq
 with the understanding
 that $\Gamma\spr m.$ is Cartier, not necessarily effective,
 but $2\Gamma\spr
 m.$ and $\Gamma\scl m.$ are effective.
 Indeed $\Gamma\spr m.$ is essentially never effective
 (compare Remark \ref{P1}).
 Nonetheless, $-\Gamma\spr m.$
  is relatively ample on the Hilbert
 scheme $X\sbr m._B$ over the symmetric product $X\spr m._B$,
 and will be referred to as the \emph{discriminant polarization.}
 \newsubsection{Globalization II}\label{globalization2}
 We will now take up the globalization (over the base $B$, of
 arbitrary dimension) of the
  results of the previous section, in
their '$Z$ coordinate' form, which is more closely related to the blowup structure compared to the $u,v$ coordinate form.
 We will do this for the ordered version of the Hilbert scheme, viz $X\scl m._B$ with its ordered cycle map to $X^m_B$.
 To this end, a key point is the
 globalization of the $G$ functions on $X^m_B$, or rather
 their divisors of zeros.
 We will see that these constitute a chain of $m$
 essentially canonical 'intermediate diagonal' divisors,
interpolating between the 'x-discriminant' and the
 'y-discriminant'. These intermediate diagonals are Cartier divisors
 consisting of the big diagonal $OD^m$ plus certain
 boundary divisors, and the common schematic intersection
 of all of them
is exactly  $OD^m$. Most of
 our results on the intermediate diagonals are contained
 in the following statement.
 \begin{prop}
 Let $X/B$ be a flat family of nodal curves with irreducible base and generic fibre. Let $\theta$ be a relative node of $X/B$. Then
 \begin{enumerate}\item
 there exists an analytic neighborhood $U$ of $\theta$
  in $X$ and a rank-$m$ vector bundle $G^m(\theta)$,
defined in  $U^m_B\subset X^m_B$, together with
 a surjection in $U^m_B$: \eqspl{}{G^m(\theta)\to \I_{OD^m}}
 giving rise to a
 natural polarized embedding
 \eqspl{}{
 U\scl m._B=\Bl_{OD^m}(U^m_B)\hookrightarrow \P(\I_{OD^m}|_{U^m_B})
 \to \P(G^m(\theta)).
 }\item
 For any relatively
 affine, \'etale open  $\tilde U\to U$ of $\theta$ in $X/B$ in which the
 2 branches along $\theta$ are distinguishable,
 $G^m(\theta)$ splits over $(\tilde U)^m$ as a direct sum of invertible ideals
 $G^m_j(\theta), j=1,...,m$;
 \item Moreover if $V=U\setminus \pi\inv\pi(\theta)$, i.e. the union of the smooth fibres in $U$, the restriction of each $G^m_j(\theta)$ on $V^m_B$ is isomorphic to $\I_{OD^m}$.
 \end{enumerate}
 \end{prop}\begin{proof}[Proof-construction]
 Fix a boundary datum $(T,\delta, \theta)$ corresponding to $\theta$
  as in \S\ref{globalization}. We first work locally in $X$, near a node in one singular fibre. Then we may assume the two branches
  along $\theta$ are distinguishable in $U$. We let $\beta_x, \beta_y$ denote the $x$ and $y$ branches, locally defined respectively by $y=0, x=0$. Consider
 the Weil divisor on $U^m_B$ defined by
 \eqspl{}{OD^m_x(\theta)=OD^m+\sum\limits_{i=2}^m
 \binom{i}{2}\sum\limits_{\begin{matrix}
 I\subset[1,m]\\|I|=i\end{matrix}} p_I^*(\beta_y)p_{I^c}^*(\beta_x).
 }
 \begin{claim}
 We have
 \eqspl{}{
 OD^m_x(\theta)=\mathrm{zeros}(G_1).
 }
  where $G_1$ denotes the locally-defined Van der Monde determinant with respect to a local coordinate system
  as above.
 \end{claim}
\begin{proof}[Proof of claim]
 Indeed each factor $x_a-x_b$ of $G_1$ vanishes on
 $p_I^*(\beta_y)p_{I^c}^*(\beta_x)$ precisely when $a,b\in I$;
 the rest is simple counting.
 \end{proof}
 Thus, the divisor of $G_1$
 is canonically defined, depending only on the choice of branch.
 Given this, it is natural in view of \eqref{Gdef1} to define
 the $j$-th \emph{intermediate diagonal} along $\theta$ as
 \eqspl{}{
 OD^m_{x,j}(\theta)=OD^m_x(\theta)+(j-1)\sum p_i^*(\beta_x)-
 (j-1)(m-j/2)\partial_\theta
 } where $\partial_\theta=\beta_x+\beta_y$ is the boundary divisor
 corresponding to the node $\theta$. Indeed \eqref{Gdef1} now shows
 \eqspl{}{
 OD^m_{x,j}(\theta)=\mathrm{zeros}(G_j).
 } In particular, it is an effective Cartier divisor on
 $U^m_B$. Though each individual intermediate diagonal
 depends on the choice of branch, the collection of
 them does not. Indeed
the elementary identity
 \eqspl{}{
 \sigma_m^yV^m_1=(-t)^{\binom{m}{2}}V^m_m.
 } shows
that flipping $x$ and $y$ branches takes
 $-OD^m_{x,j}(\theta)$ to $-OD^m_{y,m+1-j}(\theta)$.
 Now set
 \eqspl{}{
 G^m(\theta)=\bigoplus \limits_{j=1}^m
 G^m_j(\theta), \text{\ \ where\ \ }
G^m_j(\theta)=  \O(-OD^m_{x,j}(\theta))
 } This rank-$m$ vector bundle is independent of the choice of branch, as is the natural map $G^m(\theta)\to\O_{U^m_B}$.
 Therefore these data are defined globally over $B$ in a
 suitable analytic neighborhood of $\theta^m$.
  By
 Corollary \ref{Gs-generate-bigdiag},
 the image of this map is precisely the ideal of $OD^m$, i.e.
 there is a surjection
 \eqspl{}{
 G^m(\theta)\to \I_{OD^m}|_{U^m_B}\to 0.
 }
 Applying the $\P$ functor, we obtain a closed embedding
 \eqspl{}{
 \P(\I_{OD^m})\to \P(G^m(\theta))
 } Now the blow-up of the Weil divisor $OD^m$, which we have
 shown coincides with the Hilbert scheme $U\scl m._B$,
 is naturally a subscheme of $\P(\I_{OD^m})$, whence
 a natural embedding
 \eqspl{}{
 U\scl m._B\to\P(G^m(\theta))
 } Note that this is well-defined globally over $U^m_B$,
 which sits over a neighborhood of the boundary component
 in $B$ corresponding to $\theta$, i.e. $\partial_\theta$.
 \end{proof}
 It is important to record here for future reference a compatibility between the $G^m_j(\theta)$ for different $m$'s.
 To this end let $U_x, U_y\subset U$ denote the complement
 of the $y$ (resp. $x$) branch, i.e. the open sets given by
 $x\neq 0, y\neq 0$.
 \begin{lem}[Localization formula] We have for all
 $1\leq j \leq m-k_x-k_y$,
 \eqspl{}{
G^m_{j+k_y}(\theta)|_{U^{m-k_x-k_y}\times U_x^{k_x}\times
U_y^{k_y}}&= \\ p_{U^{m-k_x-k_y}}^*G^{m-k_x-k_y}_{j}(\theta)
\otimes&\prod\limits_{a<b> m-k_x-k_y}  p^*_{a,b}(d_{a,b})
 } where $d_{a,b}$ is an equation for the Cartier
 divisor which equals the diagonal in
  $a,b$ coordinates; in divisor terms, this means
\eqspl{}{
OD^m_{x,j+k_y}|_{U^{m-k_x-k_y}\times U_x^{k_x}\times
U_y^{k_y}}=p_{U^{m-k_x-k_y}}^*(OD^{m-k_x-k_y}_{j})+
\sum\limits_{a<b>m-k_x-k_y}p_{a,b}^*(OD^2)
}  where the last sum is Cartier and independent of $j$.

 \end{lem}
 \begin{proof}
 We begin with the observation that, for the universal deformation $X_B$ of a node $p$, given by $xy=t$, the Cartesian square $X^2_B$ is nonsingular
 away from $(p,p)$, hence the diagonal is Cartier away from $(p,p)$, defined locally by $x_1-x_2$ in the open set where $x_1\neq 0$ or $x_2\neq 0$ and likewise for $y$.
Because the question is local and locally any deformation is induced by the
universal one, a similar assertion holds for an arbitrary family.
Now returning to our situation, let us write $n=m-k_x-k_y$ and $N, K_x, K_y$ for the respective index ranges
 $[1,n], [n+1, n+k_x], [n+k_x+k_y+1,m]$, and
 $x^N, y^N$ etc. for the corresponding monomials. Then
 $x_a-x_b$ is a single defining equation for the Cartier $(a,b)$ diagonal whenever $a$ or $b$ is in $K_x$.  Then by \eqref{Gdef1}, \eqref{Gdef2}, we can write $G^m_{j+k_y}(\theta)$, up to a unit, i.e. a function vanishing nowhere in the open set in question,
 in the form
 \eqspl{Gy/t}{
 \frac{G^N_1(y^N)^{j-1}}{t^{(j-1)(n-j/2)}}
 \prod\limits_{\substack{a<b\\ a\ \mathrm{or}\ b \in K_x}}(x_a-x_b)
 \frac{(y^N)^{k_y}\prod\limits_{a\in N, b\in K_y}(x_a-x_b)}
 {t^{nk_y}}
 \frac{\prod\limits_{a<b\in K_y}(x_a-x_b)}{t^{\binom{k_y}{2}}}
 }
 where we have used the fact that $\frac{(y^{K_x})^{j+k_y-1}}{t^{k_x(j+k_y-1)}}$ is a unit.
 Now in \eqref{Gy/t}, the first factor is just $G^N_j$
 while the second is the equation of a Cartier partial  diagonal. The third factor is equal up to a unit to
 $\prod\limits_{a\in N, b\in K_y}(y_a-y_b)$, hence is also
 the equation of a Cartier partial diagonal.
 Finally, in the fourth factor, each subfactor
 $(x_a-x_b)/t=y_b\inv-y_a\inv$, so this too yields a Cartier
  partial diagonal.
 \end{proof} Now recall the notion of boundary datum $(T, \delta, \theta)$ introduced in \S\ref{globalization}. We are now in position to determine globally the pullback of the intermediate diagonals to the partial normalization
 $X_T^\theta$:

 \begin{cor}\begin{enumerate}\item
 The pullback of $G^m_{j+k_y}(\theta)$ on $U^{m-k_x-k_y}_B\otimes_BU\scl k_x+k_y._B$
  extends over $U^{m-k_x-k_y}_B\otimes_B(X_T^\theta)\scl k_x+k_y._T$ to
 \eqspl{Gj-extn}{
 p_{U^{m-k_x-k_y}}^*G^{m-k_x-k_y}_j(\theta)
\otimes p_{X\scl k_x+k_y._B}^*\O(-\Gamma\scl k_x+k_y.)
\otimes\bigotimes\limits_{a\leq m-k_x-k_y<b}  p^*(\O(-OD_{a,b}^m))
 }
where the last factor is invertible;
\item the closure in
$U^{m-k_x-k_y}_B\otimes_B(X_T^\theta)\scl k_x+k_y._T$  of the pullback of  $OD^m_{j+k_y}(\theta)$
to  $U^{m-k_x-k_y}_B\otimes_BU\scl k_x+k_y._B$
equals
 \eqspl{ODj-extn}{
 p_{U^{m-k_x-k_y}}^*OD^{m-k_x-k_y}_j(\theta)
+ p_{X\scl k_x+k_y._B}^*(\Gamma\scl k_x+k_y.)
+\sum\limits_{a\leq m-k_x-k_y<b}  p^*(OD_{a,b}^m)
 }
where each summand is Cartier.
\end{enumerate}
 \end{cor}
 \begin{proof}
 The first assertion is immediate from the Proposition. For the second, it suffices to note that the divisor in question has no components supported off $U^{m-k_x-k_y}_B\otimes_BU\scl k_x+k_y._B$.
 \end{proof}
As an important consequence of this result, we can now determine the restriction
of the $G$-bundles (i.e. the intermediate diagonals) on (essentially) the locus
of cycles containing a node $\theta$ with given multiplicity; it is these restricted
bundles that figure in the determination of the (polarized) node scrolls.
\begin{prop}\label{bundle-for-node-scroll}
Let $(T, \delta, \theta)$ be a boundary datum, $1\leq j,n\leq m$ be integers, and consider the
map
\eqsp{\mu^n:(X^\theta)_T\scl k.\to X^m_B\\
\mu^n(z)=c_k(z)+n\theta.
}
Then with $j_0=\min(j,n)$, we have
\eqspl{ODj-for-scroll}{
(\mu^n)^*(OD^m_j(\theta))\sim -\binom{n-j_0+1}{2}\psi_x-\binom{j_0}{2}
\psi_y+(n-j_0+1)\theta_x\scl k.+(j_0-1)\theta_y\scl k.+\Gamma\scl k.
} where $\psi_x=\omega_{X^\theta_T/T}\otimes\O_{\theta_x}$ is the cotangent
(psi) class at $\theta_x$ (which is a class from $T$, pulled back to
$(X^\theta)_T\scl k.$),
$\theta_x\scl k.=\sum\limits_{i=1}^k p_i^*(\theta_x)$
(which is a class from $(X^\theta)_T^k$, pulled back to $(X^\theta)_T\scl k.$),
 and likewise for $y$.
\end{prop}
\begin{proof}
We factor $\mu^n$ through the map
\eqsp{\mu^n_{j_0}:(X^\theta)_T\scl k.\to(X^\theta)_T^m\\
\mu_{j_0}^n(z)=((\theta_x)^{n-j_0+1}, (\theta_y)^{j_0-1},c_k(z)).
}
We may write $G^m_j$ as $G^n_{j_0}$ times a partial diagonal
equation as above, and
the inequalities on $j_0$ ensure that $G^n_{j_0}$ does not vanish
identically on $\beta_x^{n-j_0+1}\times\beta_y^{j_0-1}$,
where $\beta_x, \beta_y$ are the branch neighborhoods
of $\theta_x, \theta_y$ in $X^\theta_T$.
Then it is straightforward that the last two summands in
\eqref{ODj-extn} correspond to the last three summands in
\eqref{ODj-for-scroll},
e.g. a diagonal $D^m_{a,b}$ with $a\leq n-j_0+1$ coincides with
$p_b^*\theta_x$.
So it's just a matter of evaluating the
pullback of $OD^n_j(\theta)$. For the latter, we use a Laplace (block) expansion of $G_{j_0}$ on the first $n-j_0+1$ rows. In
this expansion, the leading term is the first, i.e. the product of the two corner blocks. There writing $x_a-x_b=dx$, a generator of $\psi_x$, and likewise for $\psi_y$, we get the asserted form as in \eqref{ODj-for-scroll}.
\end{proof}
 \newsection{The tautological module}
 In this section we will compute arbitrary powers
 of the discriminant polarization $\Gamma\spr m.$ on the
 Hilbert scheme $X\sbr m._B$. The computation will be
 a by-product of a stronger result determining the
 (additive)
 \emph{tautological module} on $X\sbr m._B$,
 to be described informally in this introduction,
and defined formally in the body of the chapter
(see Definition \ref{taut-mod-def}).
 \par
 The tautological module
 $$T^m=T^m(X/B)\subset A^.(X\sbr m._B)_\Q$$
 is to be defined as the $\Q$-vector space generated by certain
basic \emph{tautological classes} (as described below).
 On the other hand, let
 $$\Q[\Gamma\spr m.]\subset A^.(X\sbr m._B)_\Q$$
 be the subring of the Chow ring
  generated by the discriminant polarization.
 Then the main result of this chapter is
 \begin{thm}[Module Theorem]\label{taut-module}
 Under intersection product, $T^m$ is a
 $\Q[\Gamma\spr m.]$-module; moreover, multiplication by
 $\Gamma\spr m.$ can be desribed explicitly.
 \end{thm}
 Because $1\in T^m$ by definition, this statement includes
 the nonobvious assertion that
 $$\Q[\Gamma\spr m.]\subset T^m;$$
 in other words, any polynomial in $\Gamma\spr m.$ is (explicitly) tautological.
 In this sense, the Theorem includes an 'explicit' (in the
 recursive sense, at least) computation of all the powers
 of $\Gamma\spr m.$.\par
 Now the aforementioned basic tautological classes come in two
 main flavors (plus some subflavors).\begin{enumerate}
 \item The (classes of)
 \emph{(relative) diagonal loci} $\Gamma\spr m.\subp{n_1,n_2,...}$: this
 locus  is
 essentially the closure of the set of schemes
 of the form $n_1p_1+n_2p_2+...$ where $p_1,p_2...$
 are distinct smooth points of the same (arbitrary) fibre.\par
 More generally, we will consider certain 'twists' of
 these, denoted \nl $\Gamma\subp{n_1,n_2,...}[\alpha_1,\alpha_2...]$,
 where the $\alpha.$ are 'base classes', i.e. cohomology classes on $X$.
 \item The \emph{node classes}. First, the
  \emph{node scrolls} $F_j^n(\theta)$:
 these are, essentially, $\P^1$-bundles over an
analogous diagonal locus $\Gamma\spr m-n._{(n.)}$ associated to
a boundary family $X^\theta_T$ of $X_B$,
whose general
 fibre can be naturally identified with the punctual
 Hilbert scheme component $C^{n}_j$ along the node $\theta$.\par
Additionally, there are the \emph{node sections}: these are simply
 the classes $-\Gamma\spr m..F$ where $F$ is a node scroll
 as above (the terminology comes from the fact that
 $\Gamma\spr m.$ restricts to $\O(1)$ on each fibre of a
 node scroll).\par
 All these classes again admit {twisted} versions, essentially
 obtained by multiplying by bases classes from $X$.

 \end{enumerate}
 \par
 Effectively, the task of proving Theorem \ref{taut-module}
 has two parts.\begin{enumerate}
 \item Express a
 product $\Gamma\spr m..\Gamma\subp{n.}$ in terms
 of other diagonal loci and node scrolls, see Proporsition
 \ref{disc.diag}.
 \item For
 each node $\theta$ and associated ($\theta$-normalized) boundary
 family $X_T^\theta$, determine a series of explicit line bundles
 $E^n_j(\theta), j=1,...,n$ on the relative Hilbert scheme
  $(X_T^\theta)\sbr m-n._T$
together with an identification
 $$F_j^n(\theta)\simeq \P(E_j^n(\theta)\oplus E_{j+1}^n(\theta)),$$ such that the restriction of the discriminant
 polarization $-\Gamma\spr m.$ on $F^n_j(\theta)$ becomes the standard $\O(1)$
 polarization on the projectivized vector bundle (see Proposition \ref{pol-scroll-unordered});
 in fact, $E_j^n(\theta)$ is just the sum of the polarization $\Gamma\sbr m-n.$ and a suitable base divisor, see \eqref{Ebundles}, \eqref{Ebundles-nonsep}.
It then transpires that the
 restriction of an arbitrary power $(\Gamma\spr m.)^k$ on
 $F$ can be easily and explicitly expressed in terms of
 other node classes coming from the Chern classes of $E$
 (see Corollary \ref{disc^k.scroll}).
 \end{enumerate}

 \newsubsection{The small diagonal}\label{smalldiag}
We begin our study of diagonal-type loci and their
intersection product with the discriminant polarization with the smallest such
locus,
 i.e. the small diagonal.  In a sense this is
 actually the heart
 of the matter, which is hardly surprising, considering as the small
 diagonal is in the 'most special' position vis-a-vis the
 discriminant.
 The next result is in essence a corollary to the Blowup Theorem \ref{blowup}.\par
  Let $\Gamma_{(m)}\subset
 X\sbr m._B$ be the small diagonal, which parametrizes schemes with
 1-point support, and which is the pullback of the small diagonal
  \beq D_{(m)}\simeq X\subset X\spr m._B. \eeq The
 restriction of the cycle map yields a birational morphism
 \beq \frak c_m:\Gamma_{(m)}\to X \eeq which is an isomorphism except
 over the  nodes of $X/B$.
 Fix a covering system of boundary data $\{(T., \delta., \theta.)\}$
and  focus on its typical node $\theta$.
 Let \beq J_m^{\theta.}=\bigcap\limits_i J_m^{\theta_i}\subset
 \O_X \eeq be the ideal sheaf whose stalk at each fibre node
  $\theta_i$ is locally of
 type $J_m$ as in \S 0. Note that $J^{\theta.}_m$ is
 well-defined independent of the choice of local parameters and independent as well of the ordering of the branches at each node, hence makes sense and is globally defined on $X$.
  \begin{prop}Via $\frak c_m$, $\Gamma_{(m)}$ is
  equivalent to the blow-up of $J_m^{\theta.}.$
  If $\O_{\Gamma_{(m)}}(1)_J$ denotes the canonical blowup
 polarization, we have
 \beql{Gamma-smalldiag}{\O_{\Gamma_{(m)}}(-\Gamma\spr m.)=\omega_{X/B}^{\otimes
 \binom{m}{2}}\otimes\O_{\Gamma_{(m)}}(1)_J.}
 Furthermore, if $X$ is smooth at a node $\theta$, then $\Gamma\subp{m}$ has
 multiplicity $\min(i,m-i)$ along the corresponding divisor
 $C^m_i-\{Q^m_i, Q^m_{i+1}\}$ for $ i=1,...,m-1$.
 In particular, $\Gamma\subp{m}$ is smooth along
 $(C^m_1-Q^m_2)\cup(C^m_{m-1}-Q^m_{m-1})$.

 \end{prop}
 \begin{proof} We may work with the ordered versions
 of these objects, then pass to $\frak G_m$-invariants.
We fist work locally over a
 neighborhood of a point $p^m\in X^m_B$ where $p$ is a fibre node.
 We may then assume $X$ is a smooth surface and $X/B$ is given by $xy=t$, as the general case is derived from this by base-change.
 Then the ideal of $OD^m$ is generated by $G_1,...,G_m$ and $G_1$
 has the Van der Monde form $v^m_x$, while the other $G_i$ are given
 by \refp{Gdef1}..
  We try to restrict the ideal of $OD^m$ on the small
 diagonal $OD_{(m)}.$ To this end, note to begin
 with the natural map
 $$\I_{OD^m}\to \omega^{\binom{m}{2}}, \omega:=\omega_{X/B}.$$
 Indeed this map is clearly defined off the singular locus of $X^m_B$, hence by reflexivity of $\I_{OD^m}$ extends everywhere, hence moreover factors through a map
 $$\I_{OD^m}.OD\subp{m}=\I_{OD^m}\otimes\O_{OD\subp m}/{\mathrm{(torsion)}}\to\omega^{\binom{m}{2}}.
 $$ To identify the image, note
 that
 \beq (x_i-x_j)|_{OD_{(m)}}=dx=x\frac{dx}{x} \eeq
 and $\eta=\frac{dx}{x}=-\frac{dy}{y}$ is a local generator of $\omega$ along $\theta$.
 Therefore
 \beq G_1|_{OD_{(m)}}=x^{\binom{m}{2}}\eta^{\binom{m}{2}}. \eeq
 From \refp{Gdef2}. we then deduce
 \beql{G-smalldiag}{ G_i|_{\Gamma_{(m)}}=x^{\binom{m-i+1}{2}}y^{\binom{i}{2}}
 \eta^{\binom{m}{2}}, i=1,...,m.}
 Since $G_1,...,G_m$ generate the ideal  $I_{OD^m}$
 along $\theta$, it follows
 that over a neighborhood of $\theta$, we have
 \beq I_{OD^m}.{OD_{(m)}}\simeq
 J_m^\theta\otimes\omega^{\binom{m}{2}}. \eeq
This being true for each node, it is also true globally.
 Consequently, passing to the $\frak S_m$-quotient, we also
 have \beq I_{D^m}.{D_{(m)}}\simeq
 J_m^\theta\otimes\omega^{\binom{m}{2}}. \eeq Then pulling back to
 $X\sbr m._B$ we get \refp{Gamma-smalldiag}..\par
 Finally, it follows from the above, plus the explicit
 description of the model $H_m$, that, along the
 'finite' part $C^m_i-Q^m_{i+1}$, $\Gamma\subp{m}$
 has equation $x^{m-i}-uy^i$ where $u$ is an affine
 coordinate on $C^m_i-Q^m_{i+1}$, from which our last
 assertion follows easily.
  \end{proof}
Now  it follows from the Proposition that, given a node $\theta$
of $X/B$, the pullback ideal of $J_m^\theta$ on $\Gamma=D\subp m$ is an invertible ideal supported on the inverse image of $\theta$, i.e. $\bigcup\limits_{i=1}^{m-1}C^m_i(\theta)$;
we denote this ideal by $\O_\Gamma(1)_{J^\theta_m}$
or $\O_\Gamma(-e_m^\theta)$. It must not be confused with the
pullback of the reduced ideal of $\theta$.
 \begin{prop}\label{compute-excdic}
 We have \beql{}{e_m^\theta=\sum\limits_{i=1}^{m-1}\beta_{m,i}C^m_i(\theta)
 } where the $\beta_{m,i}$ are as in \S0.
 \end{prop}
 \begin{proof} We may fix $\theta$ and
 work locally with the universal family
 $xy-t$.
 Clearly the support of $e_m$ is
 $C^m=\bigcup\limits_{i=1}^{m-1}C^m_i,$ so we can write
 \beq e_m=\sum\limits_{i-1}^{m-1}b_{m,i}C^m_i \eeq and we have
 \beq -e_m^2=\deg(\O(1).e_m)=\sum\limits_{i=1}^{m-1} b_{m,i}=:b_m. \eeq
 Now the general point on $C^m_i$ corresponds to an ideal
 $(x^{m-i}+ay^i), a\in\C^*$ and the rational function
 $x^{m-i}/y^i$ restricts to a coordinate on $C^m_i.$
 It follows that if
 $A_i\subset X$ is the curve with equation $f_i=x^{m-i}-ay^i$
 for some constant $ a\in\C^*,$
 then its
 proper transform $\tilde{A_i}$ meets $C^m$ transversely
 in the unique point
 $q\in C^m_i$ with coordinate $a$, so that
 \beq \tilde{A_i}.e_m=b_{m,i}. \eeq Thus, setting $J_{m,i}=J_m+(f_i)$
 we get following
 characterization of $b_{m,i}$:
 \beq b_{m,i}=\l(\O_X/J_{m,i}).\eeq
 To compute this, we start by noting that a cobasis $B_m$ for $J_m,$
 i.e. a basis for $\O_X/J_m$ is given by the monomials $x^ay^b$
 where $(a,b)$ is an interior point of the polygon $S_m$ as in \S 0;
 equivalently, the square with bottom left corner $(a,b)$ lies in
 $R_m.$ Then a cobasis $B_{m,i}$ for $J_{m,i}$ can be obtained by
 starting with $B_m$ and eliminating\par - all monomials $x^ay^b
 $ with $b\geq i$;\par - for any $j$ with $\binom{j}{2}\geq i,$
 all monomials that are multiples of
 $x^{\binom{m+1-j}{2}+m-i}y^{\binom{j}{2}-i}$;\nl the latter of
 course comes from the relation
 \beq x^{\binom{m+1-j}{2}}y^{\binom{j}{2}}\equiv 0\mod J_m. \eeq
 Graphically, this cobasis corresponds exactly to the polygon
 $S_{m,i}$ in \S 0, hence \beq b_{m,i}=\beta_{m,i},
 b_m=\beta_m;\eeq
 \end{proof}
\begin{cor} Suppose $B$ is 1-dimensional. With the above notations, we
 have \beql{}{ e_m^2=-\sigma\beta_m,}
 where $\sigma$ is the number of nodes of $X/B$;
 \item
 \beql{disc.smalldiag}{\Gamma\spr m..\Gamma\subp{m}=
 \sum \beta_{m,i}C^m_i-\binom{m}{2}\omega_{X/B};}\item
 \beql{}{\int\limits_{\Gamma_{(m)}}(\Gamma\spr
 m.)^2=-\sigma\beta_m+\binom{m}{2}^2\omega_{X/B}^2.}.
\end{cor}
\begin{rem} The components $C_{i}^m(\theta), i=1,...,m-1$ of $e_m$
 are $\P^1$-bundles over $\theta$ and are special cases of the {{node scrolls}}, encountered in the previous section,
  and which will be further discussed in
 \S2.3 below. The coefficients $\beta_{m,i}$ play an essential role our
 intersection calculus.\end{rem}
For the remainder of the paper, we set \beq\omega=\omega_{X/B}.\eeq We will view this
interchangeably as line bundle or divisor class.\par

\newsubsection{Monoblock and polyblock  digaonals: ordered case}
\label{ordered-diags}
Returning to our family $X/B$ of nodal curves, we now begin
extending the results of \S\ref{smalldiag} to
the more general diagonal  loci as defined above, first for
those that live over all of $B$, and subsequently for loci associated to the boundary.
We call these \emph{monoblock} and \emph{polyblock  diagonals},
depending on whether they correspond to a single block or to a
partition. These loci come in both ordered and unordered versions, the ordered version being more convenient,
the unordered one the more 'correct' or natural one.
 We begin with the ordered monoblock diagonal,
 defined as follows. For an index-set
$I\subset [1,m]$, set \beql{}{OD_I^m=OD_I=p_I\inv(OD_{|I|})\subset
X^m_B,} where $p_I:X^m_B\to X^{|I|}_B$ is the projection and
$OD_{|I|}\subset X^{|I|}_B$ is the small diagonal; thus, $OD^m_I$ is
$X^{(I.)}_B$ where $(I.)$ is any partition on $[1,m]$ equivalent to
the single block $I$. Note that \beql{}{OD^m=\sum\limits_{1\leq
a<b\leq m}OD^m_{a,b}.}
Also, we have an isomorphism $OD^m_I\simeq X^{m-|I|+1}_B$.
Let
\beql{}{\Gamma_I=\Gamma_I\scl m. :=
oc\inv(OD_I)\subset X\scl m._B
}   These are called (ordered) monoblock
diagonal loci
(by comparision, the unordered monoblock diagonals, to be
studied below, will be associated to a block \emph{size} rather
than a block).
. Note that $OD_I$, hence $\Gamma_I$, are defined
locally near a node by equations \beql{}{x_i-x_j=0=y_i-y_j, \
\forall i,j\in I.}
\par
Similarly, for any partition
$$(I.)=(I_1,...,I_r)\subset [1,m],$$
we define an analogous locus ( ordered \emph{polyblock  diagonal }) \beql{}{
\Gamma_{I_1|...|I_r}=\Gamma _{I_1|...|I_r}\scl m.\subset X\scl m._B}
and note that \beql{}{
\Gamma_{I_1|...|I_r}=\Gamma_{I_1}\cap...\cap\Gamma_{I_r}}
(transverse intersection). Also
\beql{}{
\Gamma\subp{I.}=oc\inv(OD\subp{I.})
}
where $OD\subp{I.}\subset X^m_B$ is the analogous polyblock diagonal.
Note that when $(I.)$ is full of length $r$, we have
\beql{}{OD\subp{I.}\simeq X^r_B
}

\par

Now to analyze a monoblock diagonal locus $OD_I$, a key
technical  question is
to determine the part of $OD_I$ over the boundary of $B$;
or equivalently, fixing a boundary datum $(\theta, T, \delta)$,
with the associated map $\phi:X^\theta_T\to X$, to determine
$(\phi^m)^*(OD_I)$.
 Fixing such a boundary datum, the answer is as follows,
 where we denote the $x$ and $y$ axes in $X^\theta_T$
 by $X', X"$ respectively.
 \begin{lem}\label{mono-spec} The irreducible components of
the boundary of $\Gamma_I$ over $T$ are as follows:
\begin{enumerate}
\item for each index-set $K$, $[1,m]\supset K\supset I$, a locus
 $\tilde\Theta _{K/I}$, mapping
birationally to its image $\Theta_{K/I}\subset OD_I$;
\item for each $K\subset I^c=[1,m]\setminus I$, ditto;
\item for each $K$ straddling $I$ and $I^c$,
and each $j=1,...,|I|-1$, a component
$OF_j^{I:K-I|K^c-I}(\theta)\subset OF^I_j(\theta)$
projecting
as $\P^1$-bundle to its image in
$(X^\theta_T)\scl m-|I|.$, which lies over
 $(X')^{K-I}\times_T(X")^{K^c-I}=:(X^\theta_T)^{K-I|K^c-I}\subset
 (X^\theta_T)^{m-|I|}$.
\end{enumerate}\end{lem}
\begin{proof} We may fix a node $\theta$ and work locally over
a neighborhood of $\theta$ in $X$. From the definition and basic properties of node scrolls (see \S\ref{globalization}), the main point is to determine the boundary of $OD_I$. But this is easily
determined:
referring to
\refp{Theta}., the latter boundary is given locally by
$$\bigcup\limits_{K\subset [1,m]} OD_I\cap \Theta_K.$$
Set $\Theta_{K/I}=OD_I\cap\Theta_K$. To describe these,
there are 3 cases
depending on $K$:\begin{enumerate}\item
if $I\subset K$, then
$$\Theta_{K/I}=(X')^{K/I}\times (X")^{K^c};$$
\item if $I\subset K^c$, then
$$\Theta_{K/I}=(X')^K\times (X")^{K^c/I};$$
\item otherwise, i.e. if $I$ straddles $K$ and $K^c$, then
 $$\Theta_{K/I}=\{y_i=0,\forall i\in K\cup I,
x_i=0, \forall i\in K^c\cup I\}$$
$$=(X')^{K-I}\times (X")^{K^c-I}\times 0^I=:X^{K-I|K^c-I}$$
(to specify the special value $s\in B$, a subscript $s$
may subsequently be added in the above).
\end{enumerate}
\par
Now is an elementary check that
the loci of type (i) and (ii) are precisely the irreducible
components of the special fibre of $OD_I$, while
 the union of the loci
$\Theta_{K/I}$ of type (iii) coincides with the intersection of $OD_I$ with the fundamental locus (=image
of exceptional locus) of the ordered cycle map
$oc_m$,
i.e. the locus of cycles containing the node with multiplicity
$>1$. Also, each $\Theta_{K/I}$ of type (iii) is of codimension 2 in $OD_I$. On the other hand, each such $\Theta_{K/I}=X^{K-I|K^c-I}$ is just a component of the invese
image in $X^m_B$ of the locus denoted 
$X^{(a,b)}$  in Lemma \ref{Hab},
where $a=|K-I|, b=|K^c-I|$, and therefore by
that Lemma, the ordered cycle map over it is a union of $\P^1$ bundles, viz
\beql{}{oc_m\inv(X^{K-I|K^c-I})=\bigcup\limits_{j=1}^{|I|-1}
OF_j^{I:K-I|K^c-I}}
where $OF_j^{I:K-I|K^c-I}$ is the pullback of $F_j^{(m-a-b:a|b)}$ over $X^{K-I|K^c-I},$ which is a $\P^1$
bundle with fibre $C^{|I|}_j$. This concludes the proof.
\end{proof}
Notice that, given disjoint index-sets $K_1, K_2$
with $K_1\coprod K_2= I^c$, the
number of straddler sets $K$ such that $K-I=K_1, K^c-I=K_2$ is
precisely $2^n-2$ (i.e. the number of proper nonempty subsets
of $I$). Thus, a given $OF_j^{I:K_1|K_2}$ will lie on
this many components of $\tilde{\Theta}.$ This however
is a completely separate issue from the multiplicity of
$OF_j^{I:K_1|K_2}$ in the intersection cycle
$\Gamma\sbr m..\Gamma_I$, which has to do with the blowup structure
and will be determined below.
\par
From the foregoing
 analysis, we can easily compute the intersection of
a\nl monoblock diagonal cycle with the discriminant polarization, as
follows. We will fix a covering system of boundary data
$(T_s, \delta_s, \theta_s)$,and recall that each datum
must be weighted by $\frac{1}{\deg(\delta_s)}$.
\begin{prop}\label{cut-monom} We have an equality of divisor classes on
$\Gamma_I$:
\begin{eqnarray}\label{monom-diag}
\Gamma\scl m..\Gamma_I=\sum\limits_{i<j\notin I}
\Gamma_{I|\{i,j\}}+|I|\sum\limits_{i\notin I} \Gamma_{I\cup \{i\}}\\ \nonumber
-\binom{|I|}{2}p_{\min (I)}^*\omega+
\sum\limits_s\frac{1}{\deg(\delta_s)}
\sum\limits_{j=1}^{|I|-1}\beta_{|I|,j}
\delta^I_{s,j*}
OF_{j}^{I}(\theta_s),\end{eqnarray}
where ${I|\{i,j\}}$ and $I\cup \{i\}$ denote the evident diblock partition and uniblock, respectively,  the 4th term denotes
the class of the image of the node scroll on $\Gamma_I$,
$OF_{j}^{I}(\theta)=
\sum\limits_{K_1\coprod K_2=I^c}OF_j^{I:K_1|K_2}(\theta)$,
and $\delta^I_{s,j}$ is the natural map of the latter to
$\Gamma_I\subset X\sbr m._B$ ;
precisely put, the line bundle on $\Gamma_I$ given by
$\O_{\Gamma_I}(\Gamma\scl m.)\otimes p_{\min(I)}^*(\omega^{|I|})$ is represented by
an effective divisor comprising the 1st, 2nd and 4th terms
of the RHS of \refp{monom-diag}..
\end{prop}
\begin{proof}
To begin with, the asserted equality trivially holds
away from the exceptional locus of $oc_m$, where the 1st, second and third summands come from components
$\Gamma_{i,j}$ of $\Gamma\scl m.$ having $|I\cap\{i,j\}|=0,1,2$,
respectively.\par
Next, both sides being divisors on $\Gamma_I$, it will suffice to check equality away from codimension 2, e.g.
over a generic point of each (boundary) locus
$(X^\theta_T)^{K-I|K^c-I}$.
But there, our cycle map $oc_m$ is locally just $oc_r\times
{\mathrm{iso}}$, $r=|I|$, with
$$\Gamma\scl m.\sim \Gamma\scl r.+\sum\limits_{\{i,j\}
\not\subset I}\Gamma_{i,j}.$$
We are then reduced to the case of the small diagonal, discussed in the previous subsection.

\end{proof}
The extension of this result from the monoblock to the polyblock
case- still in the ordered setting- is in principle straightforward, but a bit complicated to
describe. Again, a key issue is to describe the
boundary of a
polyblock diagonal locus $OD\subp{I.}$ in terms of the decomposition
\refp{Theta}.. Fix a boundary datum $(T,\delta, \theta)$.
To simplify notations, we will assume, losing no
generality, that the partition $I.$ is full. Now consider an
index-set $K\subset[1,m]$. As before, $K$ is said to be a
\emph{straddler} with respect to a block $I_\l$ of $(I.)$, and
$I_\l$ is a \emph{straddler block} for $K$,   if $I_\l$ meets both
$K$ and $K^c$. The \emph{straddler number} $\strad\subp{I.}(K)$ of
$K$ w.r.t. $(I.)$ is the number of straddler blocks $I_\l$. The
\emph{straddler portion} of $(I.)$ relative to $K$ is by definition
the union of all straddler blocks, i.e.
\beql{}{s_K(I.)=\bigcup\limits_{I_\l\cap K\neq\emptyset\neq I_\l\cap
K^c}I_\l.
 }
 The $x$- (resp. $y$-)\emph{-portion} of $(I.)$
(relative to $K$, of course)
  are by definition
 the partitions
\beql{}{x_K(I.)=\{I_\l:I_\l\subset K\}, y_K(I.)=\{I_\l:I_\l
\subset K^c\}.}
Finally the \emph{multipartition data} associated to $(I.)$
w.r.t. $K$ are
\beql{}{\Phi_K(I.)=(s_K(I.):x_K(I.)|y_K(I.)).}
In reality, this is a partition broken up into
3 parts: the \emph{nodebound} part $s_K(I.)$, a single block,
plus 2 \emph{at large} parts, an $x$ part and a $y$ part.
As before, we set
\beql{}{X^{\Phi_K(I.)}=(X')^{x_K(I.)}\times (X")^{y_K(I.)}}
and equip it as before with the map to $X_s^m$
obtained by inserting the node $\theta$ at the
$s_K(I.)$ positions.Now the analogue of Lemma
\ref{mono-spec} is the following
\begin{lem}For any partition $(I.)$ and boundary datum
$(T,\delta,\theta)$,, the corresponding boundary
portion of $\Gamma\subp{I.}$ is
\beql{}{\bigcup \limits_{\strad\subp{I.}(K)=0}
\tilde\Theta_{K,(I.)}\cup\bigcup\limits_{\l}
\bigcup\limits_{I'.\coprod I".=I.\setminus I_\l}
\bigcup\limits_{j=1}^{|I_\l|-1}OF_j^{(I_\l:I'.|I".)}
(\theta)}
\end{lem}
\begin{proof}
Now, one can easily verify
\beql{}{OD\subp{I.}\cap \Phi_K=X^{\Phi_K(I.)}=:\Theta_{K,(I.)}}
so that
\beql{}{OD\subp{I.}\cap X_0^m=\bigcup\limits_{K\subset[1,m]}
\Theta_{K,(I.)}.}
Now, an elementary observation is in order. Clearly, the codimension of $OD\subp{I.}$ in $X^m_B$ is $\sum\limits_\l (|I_\l|-1)$, and this also equals the codimension of
$OD\subp{I.}\cap X_0^m$ in $X_0^m$.
On the other hand,
we have
\begin{eqnarray}\dim(\Theta_{K,(I.)})=m-\left(\sum\limits_{I_\l \textrm{\ nonstraddler rel} K}(|I_\l|-1)+\sum\limits_{I_\l \textrm{\ straddler rel} K}|I_\l|\right)\\ \nonumber
= m-\sum_\l(|I_\l|-1)-\strad\subp{I.}(K).
\end{eqnarray}
It follows that\begin{itemize}
\item the index-sets $K$ such that $\Theta_{K,(I.)}$ is a
component of the boundary $OD\subp{I.}\cap (X^\theta_T)^m$
are precisely the nonstraddlers;
\item those $K$ such that $\Theta_{K,(I.)}$ is of codimension
1 in the special fibre are precisely those of straddle number 1 (unistraddlers).
\end{itemize}
\par
Next, what are the preimages of these loci
 upstairs in the ordered Hilbert scheme $X\scl m._B$?
 They can be analyzed as in the monoblock case:
\begin{itemize}\item
 if $K$ is a nonstraddler, a general cycle parametrized by
$\Theta_{K,(I.)}$ is disjoint from the node, so there will be
a unique component $\tilde\Theta_{K,(I.)}\subset oc_m\inv(\Theta_{K,(I.)}$ dominating $\Theta_{K,(I.)}$;
\item
if $K$ is a unistraddler (straddle number $=1$), the dominant
components of\nl
$oc_m\inv(\Theta_{K,(I.)})$ will be the $\P^1$-bundles
$F_j^{\Phi_K(I.)}, j=1,,,s_K(I.)-1$; note that if $I_\l$
the unique block making $K$ a straddler, then $\Phi_K(I.)=(I_\l:
x_K(I.)|y_K(I.))$; moreover as $K$ runs through all unistraddlers, $\Phi_K(I.)$ runs through the date consisting of a choice of block $I_\l$ plus a partition of the set of remaining blocks in two ('$x$- and $y$-blocks');
\item because all fibres of $oc_m$ are at most 1-dimensional,
while every component of the boundary is of codimension 1 in
$\Gamma\subp{I.}$,
no index-set $K$ with straddle number $\strad\subp{I.}(K)>1$
(i.e. multistraddler) can contribute a component to
that special fibre.
\end{itemize}
This completes the proof.
 \end{proof}
What the Lemma means is that the analysis leading
to Proposition \ref{cut-monom} extends with no essential changes to the polyblock case, and therefore the natural analogue
of that Proposition holds. This is the subject of the next
Corollary which for convenience will be stated in slightly greater
generality to allow for twisting.
Let us identify
\beql{}{\Gamma\subp{I.}\simeq{\prod\limits_{j=1}^r}\ _B
X
}
where $I_1,...,I_r$ are the blocks.
For a collection $\alpha_1,...,\alpha_r$ of cohomology classes on $X$, recall from Section \ref{partitions} the
notation
\beql{}{\Gamma\subp{I.}\star_t[\alpha_1,...,\alpha_r]=
s_t
(p_1^*(\alpha_1),...,
p_r^*(\alpha_r)
}
where
$s_t$ are the elementary symmetric functions.
Similarly, for any subvariety (or homology
class) $Y$ on $\Gamma\subp{I.}$, we have
\beql{Y-alpha}{
Y\star_t[\alpha_1,...,\alpha_r]=
\Gamma\subp{I.}\star_t[\alpha.]\cup [Y].}
When $t=r$, we write $Y\star_r[\alpha.]$ simply
as $Y[\alpha.]$.
We will use this in particular when $Y=\nolinebreak OF_j^{I_\l}(\theta)$ and note
that, because $OF_j^{I_\l}$ projects to a section
(viz. $\theta$ in each of the
$I_\l$ coordinates, we have
\beql{}{\deg(\alpha_\l)\geq\dim(B)\Rightarrow OF_j^{I_\l}[\alpha.]=0.}
To state our result compactly, it will be convenient to
introduce the following operations on partitions :
\beql{}{U_{k,\l}(I.)=(..., I_k\cup I_\l,...,\what{I_\l},...)}
(i.e. uniting the $k$th and $\l$th blocks),
\beql{}{V_\l(I.,i)=(...,I_\l\cup i,...).}

\begin{cor}
\label{disc.diag-ordered}
\begin{enumerate}\item
 For any partition $I.=I_1|...|I_r$ on $[1,m]$, we have an equality of divisor classes on
$\Gamma_{I.}$:
\begin{eqnarray}\label{polynom-diag}
\Gamma\scl m..\Gamma\subp{I.}=
 \sum\limits_{i<j\notin \bigcup I.}
\Gamma\subp{I.|\{i,j\}}+ \sum\limits_\l |I_\l|\sum\limits_{i\notin \bigcup I.} \Gamma\subp{V_\l(I.,i)}\\ \nonumber
+\sum\limits_{j<\l}|I_j||I_\l|\Gamma\subp{U_{k,\l}(I.)}
-\Gamma\subp{I.}\star_1[\binom{|I_1|}{2}\omega,...
\binom{|I_s|}{2}\omega]+
\\ \nonumber
\sum\limits_s\frac{1}{\deg(\delta_s)}
\sum\limits_\l \sum\limits_{j=1}^{|I_\l|-1}
\beta_{|I_\l|,j}\delta^{I_\l}_{s,j*}
[OF_j^{I_\l/I.}(\theta_s)]\end{eqnarray}
where
\beql{}{[OF_j^{I_\l/I.}(\theta)]=\sum\limits_{I'.\coprod I".=I.\setminus I_\l}
\sum\limits_{j=1}^{|I_\l|-1}
\delta^{I_\l}_{s,j*}[OF_j^{(I_\l:I'.|I".)}(\theta)].}
\item if $I.$ is full, we have
\eqspl{polynom-diag-full}{
&\Gamma\scl m..\Gamma\subp{I.}=
 \sum\limits_{j<\l}|I_j||I_\l|\Gamma\subp{U_{j,\l}(I.)}\\
&-\Gamma\subp{I.}\star_1[\binom{|I_1|}{2}\omega,...
\binom{|I_s|}{2}\omega]+
\sum\limits_s\frac{1}{\deg(\delta_s)}\sum\limits_\l \sum\limits_{j=1}^{|I_\l|-1}
\beta_{|I_\l|,j}\delta^{I_\l}_{s,j*}
[OF_j^{I_\l/I.}(\theta_s)]}
\item  if $I.$ is full, we have more generally
\begin{eqnarray}\label{polynom-diag-alpha}
\Gamma\scl m..\Gamma\subp{I.}[\alpha.]=
 \sum\limits_{j<\l}|I_j||I_\l|\Gamma\subp{U_{j,\l}(I.)}
 [...,\alpha_j._X\alpha_\l,...\what
 {\alpha_\l}...]\\ \nonumber
-\Gamma\subp{I.}[\alpha.]\star_1[\binom{|I_1|}{2}\omega,...
\binom{|I_s|}{2}\omega]
%
\\ \nonumber +
\sum\limits_s\frac{1}{\deg(\delta_s)}\sum\limits_\l \sum\limits_{j=1}^{|I_\l|-1}
\beta_{|I_\l|,j}\delta^{I_\l}_{s,j*}
OF_j^{I_\l/I.}(\theta_s)[\alpha.]
\end{eqnarray}

\end{enumerate}
\qed


\end{cor}

\newsubsection{Monoblock and polyblock  diagonals: unordered case}
We need the analogues of the formulae of the latter section in
the (unordered) Hilbert scheme. These are essentially
straightforward, and may be obtained from the ordered versions using
push-forward by the symmetrization map $\varpi_m$. We begin with the
monoblock case.
Recall first the the monoblock (unordered) diagonal $\Gamma\subp{n}$ is defined as a set by
\[\Gamma\subp{n}=\varpi_m(\Gamma\subp{I})\]
for any block $I$ of cardinality $n$. More generally, we may similarly define
$\Gamma\subp{n.}$ for any distribution $(n.)$: this will be
considered in detail below.
Similarly, if $\alpha$ is any cohomology class on $X$, there
is an associated class $X^{(n|1.)}_B[\alpha]$ on the symmetric product
$X\spr m._B$, and we define
\[\Gamma\subp n[\alpha]=c^*(X^{(n|1.)}_B[\alpha]).\]
alternatively, this could also be defined as
\[\Gamma\spr n.[\alpha]=\frac{1}{(m-n)!}\varpi_{m*}(\Gamma\subp{I})[\alpha].\]
 Note the following elementary
facts:\begin{enumerate}
\item \beql{}{\varpi_{m*}(\Gamma\scl m..\Gamma_I)=
\Gamma\spr m..\varpi_{m*}\Gamma_I} (projection formula,
because $\varpi_m^*(\Gamma\spr m.)=\Gamma\scl m.$; NB
$\varpi$ is ramified over the support of $\Gamma\spr m.$, still
no factor of 2 in $\varpi_m^*(\Gamma\spr m.)$,
by our definition of  $\Gamma\spr m.$ as $1/2$ its support);
\item \beql{}{\varpi_{m*}(\Gamma_I[\alpha])=
(m-n)!\Gamma_{(n)}[\alpha]\ \ , n=|I|>1;}
\item\begin{equation} \varpi_{m*} (\Gamma_{I|\{i,j\}})
=\begin{cases}
(m-n-2)!\Gamma_{(n|2)},\ \  n\neq 2;\\
 2(m-n-2)!\Gamma_{(2|2)},\ \  n=2,\\
(1+\delta_{2,n})(m-n-2)!\Gamma_{(n|2)}, \ \ \forall n
\end{cases}
\end{equation}
($\delta$= Kronecker delta);
\item \beql{}{\varpi_{m*}(\Gamma_{I\coprod \{i\}})=(m-n-1)!\Gamma_{(n+1)};}
\item \beql{}{\varpi_{m*}(OF_j^{I:K-I|K^c-I}(\theta))=
a!b!F_j^{(n:a|b)}(\theta),\ \
a=|K-I|, b=|K^c-I|=m-n-a;} moreover the number of distinct subsets $K-I$
with $a=|K-I|$,
 for fixed $I$ and $a$, is $\binom{m-n}{a}$.

\end{enumerate}
Putting these together, we conclude
\begin{prop}\label{disc.monomial}
 For any monoblock diagonal $\Gamma_{(n)},\: n>1$, we have an equivalence of codimension-1
cycles in $\Gamma_{(n)}$:
\beql{}{
\Gamma\spr m..\Gamma_{(n)}\sim\frac{1+\delta_{2,n}}{2} \Gamma_{(n|2)}
+n\Gamma_{(n+1)}-\binom{n}{2}\Gamma_{(n)}[\omega]+
\sum\limits_s\frac{1}{\deg(\delta_s)}
\sum\limits_{a=0}^{m-n}\sum\limits_{j=1}^{n-1}\ \beta_{n,j}\delta^{n}_{s,j*}F_j^{(n:a|m-n-a)}(\theta_s)
}\qed
\end{prop}
When $n=2$, $\Gamma_{(n)}$ is just $2\Gamma\spr m.$, hence
\begin{cor}\label{disc-square}
\beql{}{(\Gamma\spr m.)^2\sim \frac{1}{2}\Gamma_{(2|2)}+\Gamma_{(3)}
-\Gamma\spr m.[\omega]+\sum\limits_s
\frac{1}{\deg(\delta_s)}\half\sum\limits_{a=0}^{m-2
}\delta^{2}_{s,j*}F_1^{(2:1^a|1^{m-2-a})}(\theta_s).
}\qed\end{cor}
\begin{cor}
We have
\begin{enumerate}\item
\beql{}{\Gamma\spr m..\Gamma\subp{2}[\omega]=
\Gamma_{(2|2)}[\omega]+2\Gamma_{(3)}[\omega]
-\Gamma\subp{2}[\omega^2]
}
\item
\beql{}{\Gamma\spr m..\Gamma\subp{3}=
\half \Gamma\subp{3|2}+3\Gamma\subp{4}-
3\Gamma\subp{3}[\omega]+
\sum\limits_s
\frac{3}{\deg(\delta_s)}\sum\limits_{a=0}^{m-3}
\delta^{3}_{s,1*}(F_1^{(3:1^a|1^{m-3-a})}(\theta_s)+
\delta^{3}_{s,2*}F_2^{(3:1^a|1^{m-3-a})}(\theta_s)).
}
\end{enumerate}
\end{cor}
\begin{cor}\label{cube} We have
\begin{enumerate}
\item
\eqspl{Gamma2^k}{(\Gamma\spr 2.)^k&=\Gamma\spr 2.[(-\omega)^{k-1}]+
\sum\limits_s
\frac{1}{\deg(\delta_s)}\half\delta^{2}_{s,1*}
(\Gamma\spr 2.)^{k-2}.F^{(2:0|0)}_1(\theta_s)), k\geq 3;\\
\nonumber{\mathrm{if}} \dim(B)=1, \int\limits_{X\sbr 2._B}\Gamma\spr 2.)^3&=
\half\omega^2-\half\sigma,\qquad \sigma=|\{\text{singular values}\}|
;}
\item  \beql{}{
\begin{split}
 (\Gamma\spr 3.)^3&=-4\Gamma\subp{3}[\omega] +
\Gamma\spr 3.[\omega^2]\\
\nonumber &+\sum\limits_s
\frac{1}{\deg(\delta_s)}(3(\delta^{3}_{s,1*}F^{3:0|0}_1(\theta_s)
+\delta^{3}_{s,2*}F^{3:0|0}_2(\theta_s))+\half
\delta^{2}_{s,1*}\Gamma\spr
3.(F_1^{(2:1|0)}(\theta_s)+F^{(2:0|1)}_1(\theta_s)))
\end{split}
 }
 \end{enumerate}
\end{cor}
[In Part (ii) we have used the elementary fact that $\omega.\theta_s=0$, hence $\omega^i.F^{2:*}_1(\theta_s)=0,
\forall i>0$,
because this node scroll maps to $\theta_s$, more
precisely to $2[\theta_s]\subset X\spr 2._B$.]
\begin{cor}
If $m\leq 3$ (resp. $m>3$), then the class
$$-\Gamma\spr m.[\omega]+
\sum\limits_s
\frac{1}{\deg(\delta_s)}
\sum\limits_{a=0}^{m-2}\delta^{2}_{s,1*}F_1^{(2:1^a|1^{m-2-a})}(\theta_s)$$
(resp, $$\Gamma_{(2|2)}
-\Gamma\spr m.[\omega]+
\sum\limits_s
\frac{1}{\deg(\delta_s)}
\sum\limits_{a=0}^{m-2}\delta^{2}_{s,1*}F_1^{(2:1^a|1^{m-2-a})}(\theta_s)$$
is divisible
by 2 in the integral Chow group of $X\sbr m._B.$
\qed

\end{cor}
To simplify notation we shall henceforth denote
$\frac{1}{\deg(\delta_s)}\sum\limits_sF^\bullet_\bullet(\theta_s)$
simply as $F^\bullet_\bullet$.
\begin{example}{\rm\label{P1}This is presented here mainly as a check
on some of the coefficients in the formulas above. For $X=\P^1$, \ \  $X\spr m.=\P(H^0(\O_X(m)))=\P^m$,
and the degree of $\Gamma\spr m._{(n)}$ is $n(m-n+1)$. Indeed this degree may be computed as the degree of the degeneracy locus of a generic map $n\O_X\to P^{n-1}_X(\O_X(m))$ where $P^k_X$ denotes the $k$-th principal parts or jet sheaf.
It is not hard to show that $P^{n-1}_X(\O_X(m))\simeq n\O_X(m-n+1)$.\par
For example, $\Gamma\spr 3._{(2)}$ is a quartic scroll
equal to  the tangent developable of its
cuspidal edge, i.e. the twisted cubic $\Gamma\spr 3._{(3)}$.
The rulings are the lines $L_p=\{2p+q: q\in X\}$,
tangent to the $\Gamma\spr 3._{(3)}$,
each of which has class $-\half\Gamma\spr 3.[\omega]$. Therefore by Corollary \ref{disc-square}, the self-intersection of $\Gamma\spr 3.$  in $\P^3$ (or half the intersection
of $\Gamma\spr 3.$ with $\Gamma\spr 3._2$, as a class on
$\Gamma\spr 3._2$) is represented by $\Gamma\spr 3._{(3)}$ plus one ruling $L_p$.\par
If $m=4$ then $\Gamma\spr 4.$ is formally a cubic (half a
sextic hypersurface) in $\P^4$, whose self-intersection,
as given by Corollary \ref{disc-square}, is
half the Veronese $\Gamma\subp{2|2}$ plus the (sextic)
tangent developable $\Gamma\subp{3}$, plus one osculating
plane to the twisted quartic $\Gamma\subp{4}$,
representing $-\Gamma\spr 4.[\omega]$ .
\qed
}
\end{example}
Next we extend Proposition \ref{disc.monomial}
 to the polyblock  case.
Consider a distribution $\un$  or equivalently a shape $(n.)=(n.^{\mu.}), n_1>...>n_r$, and let
$(I.)$ be any partition having this shape.
Let $D^m\subp{n_.^{\mu.}}$ be the image of
$\Gamma\subp{n_.^{\mu.}}$ in $X\spr m._B$. We identify
\beql{}{D^m\subp{n_.^{\mu.}}\simeq{\prod\limits_{j=1}^r}\ _B
(X\spr{\mu_j})._B)={\prod\limits_{n=\infty}^1}\ _B(X\spr\mu(n)._B).
}
Generally, for any cohomology class $\alpha$ on $D^m\subp{n_.^{\mu.}}$, we will denote its pullback
via the cycle map to
$\Gamma\subp{n_.^{\mu.}}$
 by $\Gamma\subp{n_.^{\mu.}}[\alpha]$. Some special cases
 of this are:
for a collection $\alpha_1,...,\alpha_r$ of cohomology classes on $X$, and $\forall\ \lambda_j\leq\mu_j \forall j$,
 we will use the notation $\Gamma\subp{n.^{\mu.}}[\alpha_1\spr {\lambda_1}.,...,
\alpha_r\spr{\lambda_r}.]$ as in Section \ref{partitions};
and more generally
$\beta\star_t[\alpha.\spr{\lambda.}.], t\leq r,$ for any
(co)homology class $\beta$ on $\Gamma\subp{n_.^{\mu.}}$.

We need to set up analogous notations for the case
of a node scroll and its base, which is a special diagonal
locus associated with a boundary datum $(T,\delta, \theta)$.
Thus fix such a boundary datum, and let $X', X"$ be the
components of its normalization along $\theta$ (where by
convention $X'=X"$ in case $\theta$ is
nonseparating).
Define \emph{multidistribution data} $\phi$ of total
length $n$ as
$$\phi=(n_\l:\un'|\un")$$ where $ (\un')\coprod (\un")$ is a distribution of total length $n-n_\l$. In other words, if we set $$\un=\un'\coprod \un"\coprod n_\l,$$ then $\un$ is a distribution of length $n$. We will usually assume $(\un)$ is full of length $n=m$. We view $\phi$ as obtained from $\un$ by removing a single block of size $n_\l$  and
declaring each remaining block as either $x$ type or $y$ type).
Thus $\phi$ is the natural unordered analogue of a multipartition $\Phi$ and of course to each
$\Phi=(J:I'.|I".)$ there is an associated mutlidistribution
$\phi=(|J|:|I'.|\; |\; |I".)$.\par The
special diagonal locus corresponding to this multidistribution (and
to the boundary datum) is of course
\beql{}{X_\theta^\phi=\Gamma_{(n'.)}(X')\times_T
\Gamma_{(n".)}(X")\subset (X')\sbr n'._T\times (X")\sbr n"._T
}
where $\un'=(n'.^{(\mu.')}),
\un"=(n".^{(\mu.")})$. It maps to $X\spr m._B$ by adding $n_\l \theta$.
If $\theta$ is nonseparating, so that $X'=X"$, we
take $\un"=\emptyset$ as usual. The (unordered) \emph{node scroll}
$F^\phi_j(\theta)$ is the appropriate ($j$-th) component of the inverse
image of $X^\phi_\theta$ in the Hilbert scheme. It is a $\P^1$-bundle over
$X^\phi_\theta,$ with bundle projection $c=$ restriction of cycle map.\par
Now generally, for any cohomology class $\alpha$ on $X^\phi_\theta$, we may
define a class on $F_j^\phi(\theta)$ by
$$F^\phi_j(\theta)[\alpha]=c^*(\alpha).$$
More particularly, we will need the following type of class.
First, set
\beql{}{a(n_\l:\un'|\un")=a(\un')a(\un").} Now
for a
collection of cohomology classes $\alpha'., \alpha".$ on $X', X".$
respectively, we define as in Section \ref{partitions}, an
associated \emph{twisted node scroll class} by
\beql{}{F_j^\phi (\theta)[\alpha'.\spr{\lambda.'}.,
\alpha".\spr{\lambda."}.]=
\frac{1}{a(\phi)}\varpi_{m*}(OF_j^{\Phi}(\theta) [\alpha.\spr{\lambda.}.])
} where $|\Phi|=\phi$ (cf. \eqref{Y-alpha}). The numerical coefficient is just the reciprocal of the
degree of the symmetrization map $\varpi:X^{\Phi}_\theta\to X_\theta^{\phi}$.
Hence the twisted node scroll classes
 are just flat pullbacks of the analogous classes defined on
the base $X^{\phi}$ of the node scroll
$F_j^{\phi}(\theta)$, i.e.
$$F_j^{\phi}(\theta)
[\alpha'.\spr{\lambda.'}., \alpha".\spr{\lambda."}.]=
c^*X^{\phi}_\theta [\alpha'.\spr{\lambda.'}.,
\alpha".\spr{\lambda."}.].
$$
Of course, we also set
\beql{}{F_j^{n_\l/\un}(\theta)[\alpha.]=\sum\limits_{(\un')\coprod (\un")=(\un)\setminus n_\l} F_j^{(n_\l:\un'|\un")}(\theta)[\alpha.]
}

Also note that
if $\un'\coprod \un"=\un\setminus\{n_\l\}$, then
we have
\beql{}{\frac{a(\un')a(\un")}{a(\un)}=\frac{1}
{\mu\subp{\un}(n_\l)}\frac{1}
{\binom{\mu\subp{\un}(n_\l)-1}{\mu\subp{\un'}(n_\l)}}
} and also
\beql{}{\frac{a(\un\setminus\{n_\l\})}{a(\un)}=\frac{1}{\mu(n(\l))}
}
These are the respective ratios $$\frac{\deg(OF_j^{(I_\l:I'.|I".)}
\to F_j^{(|I_\l|:n'.|n".)})}{\deg(\Gamma\subp{I.}\to \Gamma\subp{\un})}$$ when the special fibre is reducible or
irreducible.\par
Similarly to the case of partitions, we will use the notation
$u_{j,\l}(n.^{\mu.})$
to denote the new distribution (of the same total
length) obtained from $(n.^{\mu.})$ by uniting
a block of size $n_j$ with one of size $n_\l$, i.e.
the distribution
whose frequency function $\mu_u$ coincides with
$\mu$, except for the values
\begin{eqnarray*}\mu_u(n_j+n_\l)&=&\mu(n_j+n_\l)+1,\\
\mu_u(n_j)&=&\mu(n_j)-1,\\
\mu_u(n_\l)&=&\mu(n_\l)-1.
\end{eqnarray*}
There is an analogous operation
on a cohomological vector $ (\alpha(n)^{\lambda(n)})$
defined by
\begin{eqnarray}
u_{j,\l}(\alpha.^{\lambda.})=\hskip 3in\\ \nonumber
(
...,\alpha(n_j+n_\l)^{\lambda(n_j+n_\l)}(\alpha(n_j)\cdot_X
\alpha(n_\l)),...,\alpha(n_j)^{\lambda(n_j)-1},...,\alpha
(n_\l)^{\lambda(n_\l)-1},...,\alpha(1)^{\lambda(1)})
\end{eqnarray}
Now set
\begin{eqnarray}
\nu\subp{\un}(a,b)=\frac{\mu\subp{\un}(a+b)+1}{\mu\subp{\un}(a)
(\mu\subp{\un}(b)-1)}&,& a=b\\ \nonumber
=\frac{\mu\subp{n.}(a+b)+1}{\mu\subp{\un}(a) \mu\subp{\un}(b)}&,&
a\neq b.\end{eqnarray} Note that \beql{}{
\nu\subp{\un}(n_j,n_\l)=\frac{\deg(\Gamma_{U_{j,\l}(I.)}\to
\Gamma_{u_{j,\l}(|I.|)})}{\deg(\Gamma\subp{I.}\to \Gamma\subp
{\un})} }
Now the following result follows directly from Corollary
\ref{disc.diag-ordered} by adjusting for the degrees of the
various symmetrization maps.
\begin{prop}\label{disc.diag} Let $(n.)=(n_1^{\mu_1}|...|n_r^{\mu_r})$
be a full distribution on $[1,m]$, $n_1>...>n_r$, and $\alpha_1,...,\alpha_r$
cohomology classes on $X$.
Let $F^\bullet_\bullet=\sum\limits_sF^\bullet_\bullet
(\theta_s)$ denote various weighted node scrolls,
with reference to a fixed covering system of boundary data.
 Then we have, :
 \begin{equation}\begin{split}
\Gamma\spr m..\Gamma_{(n.)}[\alpha.^{\lambda.}.] \sim
\sum\limits_{j<\l}
\nu\subp{n.}(n_j,n_\l)n_jn_\l
\Gamma_{(u_{j,\l}(n.))}
[u_{j,\l}(\alpha.^{\lambda.})]
\\ \nonumber -
\sum\limits_\l\Gamma_{(n)}[\alpha.]
\star_1
[\binom{n_1}{2}\omega,...,\binom{n_r}{2}\omega]
\\ \nonumber +
\sum\limits_{\theta_s {\mathrm{\ separating}}}
\sum\limits_\l
\sum\limits_{n'.\coprod n".=n.\setminus\{n_\l\}}
\frac{1}
{\mu(n_\l)}\frac{1}
{\binom{\mu(n_\l)-1}{\mu\subp{n'.}(n_\l)}}
\sum\limits_{j=1}^{n_\l-1}\
\beta_{n_\l,j}F_j^{(n_\l:n'.|n".)(\theta_s)}
[\alpha.]
\\ \nonumber +
\sum\limits_{\theta_s {\mathrm{\ nonseparating}}}
\sum\limits_\l
\frac{1}{\mu(n_\l)}
\sum\limits_{j=1}^{n_\l-1}\
\beta_{n_\l,j}F_j^{(n_\l:n'.|n".)}(\theta)
[\alpha.]
\end{split}
\end{equation}
\qed
\end{prop}

\begin{example}{\rm
We have
\begin{eqnarray}\Gamma\spr m..\Gamma\subp{2|2}\sim\hskip 3in\\
\nonumber
\frac{3}{2}\Gamma\subp{2|2|2}+2\Gamma\subp{4}
+2\Gamma\subp{3|2}-\Gamma\subp{2|2}\star_1[\omega]+
\sum\limits_{a=0}^{m-4}\frac{1}{2\binom{m-4}{a}}(
F_1^{2:2,1^a|1^{m-4-a}}+F_1^{2:1^a|2,1^{m-4-a}})
\end{eqnarray}
}\end{example}
\begin{cor}
We have
\begin{eqnarray}
(\Gamma\spr m.)^3\sim\hskip 3in\\
\nonumber
\frac{3}{4}\Gamma\subp{2|2|2}+4\Gamma\subp{4}
+\frac{3}{2}\Gamma\subp{3|2}-\Gamma\subp{2|2}\star_1[\omega]
-4\Gamma\subp{3}[\omega]+
\Gamma\spr m.[\omega^2]
\\ \nonumber +
\half\sum\limits_{a=0}^{m-4}\frac{1}{2\binom{m-4}{a}}(
F_1^{2:2,1^a|1^{m-4-a}}+F_1^{2:1^a|2,1^{m-4-a}})\\
\nonumber
+3\sum\limits_{a=0}^{m-3}
(F_1^{(3:1^a|1^{m-3-a})}+F_2^{(3:1^a|1^{m-3-a})})\\
\nonumber
+\half\sum\limits_{a=0}^{m-2}\Gamma\spr m..F_1^{(2:1^a|
1^{m-2-a})}
\end{eqnarray}
\end{cor}

Proposition \ref{disc.diag} completes one major step in the proof of Theorem \ref{taut-module}. The remaining step will
be taken in the next subsection.
\newsubsection{Polarized node scrolls}
What now remains to be done to complete the proof of Theorem
\ref{taut-module} is to work out the intersection product of the
discriminant polarization  $-\Gamma\spr m.$, and all its powers,
with a node scroll $F$. As discussed in the introduction to this
chapter, this will be accomplished via an analysis of the polarized
structure of a node scroll, refining the one given in Lemma
\ref{nodescroll0}. There we found one (decomposable)
 vector bundle $E_0$ (over a suitable  special
 diagonal locus),
such that $F=\P(E_0)$. The point  now is effectively to find a line bundle
$L$ such that the twisted vector bundle $E=E_0\otimes L$ -- which
also has $F=\P(E)$) -- is the 'right one' in the sense
that the associated polarization
$\O_{\P(E)}(1)$ coincides with the restriction of the
discriminant polarization,
i.e. $\O_F(-\Gamma\spr m.)$.
We will focus first  on 'maximal' node scrolls (ones with no
diagonal conditions); the formula thus obtained is of course
still valid in the nonmaximal case but there is can and will be usefully explicated. Finally, we will consider everything first
in the ordered case and descend to the unordered one at the end of the section.
\subsubsection{Pencils}
We begin by
studying the case where $X/B$ is a 1-parameter family of curves with a finite
number of singular members, all 1-nodal (the '1-parameter, 1-nodal' case). While in this case specifying a singular fibre is equivalent to specifying a node, in the general, higher-dimensional base case, it is the latter that carries through.
As in our study of diagonal classes, we will first consider the ordered case.
To begin with, we review and amplify
 some notations relating
to (multi)partitions and associated diagonal loci for a
\emph{special} fibre $X_s$, analogous to those established
previously in the relative case over $B$.
Thus, fix
a singular fibre $X_s$ with unique node $\theta=\theta_s$,  and let $(X',\theta_x),(X",\theta_y)$
 be the components of
its normalization with the distinct node preimages marked;
if $X_s$ is irreducible (or equivalently, $\theta$ is nonseparating),
then $X'=X"$
as global varieties, but they differ in the marking;
if necessary to specify the singular value $s$, the same
will be denoted $(X'_s,  \theta_{x,s}), (X"_s, \theta_{y,s})$ .
\par
In this setting, we recall that a \emph{ (full) set of
multipartition data}
$$\Phi=(J:I'.|I".)$$ consists of a 'nodebound' block $J=0(\Phi)$,
 plus a pair of '$x$ and $y$' partitions $I'=x(\Phi), I"=y(\Phi)$,
such that $(I.)=J\coprod (I'.)\coprod(I".)$ is a full partition on
$[1,m]$. Set
$$n=|J|,
 I'=\bigcup\limits_\l I'_\l, I"=\bigcup\limits_\l I"_\l, n'=|I'|, n"=|I"|,$$
 and let
 $$\un'=(n'.^{\mu'.}):=(|I'.|),\un"=(n".^{\mu".}):=(|I".|) $$
 be the associated distributions and shapes (with $n'_1>n'_2>...$ and ditto $n".$).
The \emph{multidistribution} associated to $\Phi$ is by
definition
$$\phi=|\Phi|=(n:\un'.|\un".).$$
Here again $n, \un', \un"$  are referred to as the \emph{nodebound},
$x$-, and $y$- portions of $\phi$, and  denoted $\theta(\phi), x(\phi),
y(\phi)$ respectively.
\par
We will say that $\Phi$ is \emph{maximal} if each $I'_\l$ and $I"_\l$ is a singleton (and thus represents a vacuous
condition). In general, we will say that
$\Phi_1\prec\Phi_2$, where $\Phi_1, \Phi_2$
are full multipartition data, if they have the same $J$ block,
and if
$(I'_{1.})\prec(I'_{2.}), (I"_{1.})\prec(I"_{2.})$ in the sense
defined earlier (i.e. if each non-singleton block of $\Phi_2$
is contained in a block of $\Phi_1$).
 \par
 As before,
 we define
 $$X_s^\Phi=X^\Phi=(X')^{(I'.)}\times (X")^{(I".)}.$$
 The notation means that the coordinates in a given block
are set equal to each other.
Therefore $$X^\Phi\simeq X^{r'}\times X^{r"},$$ where
$r',r"$ are the respective numbers of blocks in
$(I'.), (I"_.)$.
If
$X_s$ is irreducible,
$X^\Phi_s$ depends only on ${(I'.)\coprod (I".)}$,
therefore in a global context we may, and will
in this case, always  take $I"=\emptyset$;
however in a local context, specifying $(I'.), (I".)$ specifies a sheet of $X^\Phi_s$ over the origin.
We have
\beql{}{\Phi_1\prec\Phi_2\Rightarrow X^{\Phi_1}
\subset X^{\Phi_2},}
an embedding of smooth varieties.
 As before, $X^\Phi_s$ maps to the Cartesian product $X_s^m$ by putting $0$ in the $J$ coordinates. The image is defined by the following equations:
\begin{itemize}\item horizontal
\beql{}{y_i=0, \forall i\in J\cup I',\; I':=\bigcup_\l I_\l\;;}
\item vertical
\beql{}{x_i=0, \forall i\in J\cup I",\; I":=\bigcup_\l I"_\l \;;}
\item diagonal
\beql{}{x_{i_1}-x_{i_2}=0,\forall i_1,i_2\in I'_\l, \forall \l;\\
y_{i_1}-y_{i_2}=0, \forall i_1,i_2\in I"_\l, \forall \l.}
\end{itemize}
Notice that the 'origins' $\theta_x, \theta_y$
(i.e. node preimages) induce a stratification of $X^\Phi$, where the stratum $\mathcal S_{k',k"}$ of codimension $k=k'+k"$ is the locally
closed locus of points having exactly $k'$ (resp. $k"$)
 of their components
indexed by $I'$ (resp. $ I"$) equal to the origin $\theta_x$
(resp. $\theta_y$).\par

The node scroll $OF^\Phi_j=OF_j^\Phi(\theta)$ is a component of the
inverse image $oc_m\inv(X^\Phi)$.
  It is a $\P^1$ bundle over $X^\Phi$ whose fibre
  may be identified with $C^{n}_j$
over the 0-stratum $\mathcal S_{0,0}$.
There, in terms of the local model $H_n$, the fibre $C^n_j
\subset\P^{n-1}$ is defined by the homogeneous equations
$$Z_1=...=Z_{j-1}=Z_{j+2}=...=Z_n=0.$$ In other words, homogeneous
coordinates
on $C^n_j\sim\P^1$ are given by $Z_j, Z_{j+1}.$ Now recall that
under our identification of the local model $H_n$ with the blowup of
the discriminant, the $Z_i$ correspond to the generators $G_{i,J}$
given by the mixed Van der Monde determinants \refp{vandermonde}. in
the $J$-indexed variables.\par Similarly, in terms of the local
model $H_m$ over a neighborhood of the 'origin'
$(\theta_x)^{I'}\times(\theta_y)^{I"}$, i.e. the smallest stratum $\mathcal
S_{n'+n"}$, the fibre $C^m_{j+n"}$ is coordinatized by $Z_{j+n"},
Z_{j+n"+1}$, with the other $Z$ coordinates vanishing, and these
correspond to the generators
 $G_{j+n"}, G_{j+n"+1}$ (in all the
variables).\par Likewise, over a neighborhood of a point $p$ in the
stratum $\mathcal S_{k',k"}\subset X^\Phi$, we have a local model
$H_{n+k'+k"}$ for the Hilbert scheme, and there the fibre becomes
$C^{n+k'+k"}_{j+k"}$ and is coordinatized by $Z_{j+k"}, Z_{j+k"+1}$,
which correspond to the generators $G_{j+k"}, G_{j+k"+1}$ in the
appropriate variables (where the components of $p$ are
equal to $\theta_x$ or
$\theta_y$).
\par
 We need to analyze the mixed Van der Monde determinants restricted on
 $X^\Phi$. To this end, assume to begin that
 $\Phi$ is maximal in that $(I'.), (I".)$ have \emph{singleton} blocks.
We also assume for now that the singular fibre $X_s$
 in question
is reducible. We work in the local model $H_n$ over
a neighborhood of the origin.
 Consider a Laplace expansion, along the $J$-indexed columns,
 of the mixed Van der Monde determinant that yields $G_{j+n"}$.
 This expansion has one $n\times n$ subdeterminant that is equal to $G_{j,J}$,
and in particular is \emph{constant} along $X^\Phi$; the
complementary submatrix to this,
restricted on the locus $X^\Phi$, itself splits in two blocks, of size
$n'\times n', n"\times n"$, which are themselves 'shifts' of
ordinary Van der Mondes, in the $x_{I'}, y_{I"}$ variables
respectivley, where the exponents are shifted up by $n-j+1$ (resp.
$j-1$). The determinant of this complementary matrix on $X^\Phi$
equals, using block expansion, \beql{cofactor}{
\gamma_j^{(n:I'|I")}=
(x^{I'})^{n-j+1}(y^{I"})^{j-1}\prod\limits_{a<b\in
I'}(x_a-x_b)\prod\limits_{a<b\in I"}(y_a-y_b)} (where
$x^{I'}=\prod\limits_{i\in I'}x_i$ etc.). Note that $x^{I'}$ is a
defining equation for the '$x$-boundary'
\beql{}{\del_xX^{\Phi}=\bigcup\limits_{i\in I'}X^{\Phi\setminus i}
=:X^{\del_x\Phi}} where $\Phi\setminus i$ means remove $i$ from $I'$
and add it to $J$ and map the appropriate locus to $X^\Phi$ as a
Cartier divisor by putting $\theta_x$ at the $i$th coordinate; similarly
$y^{I"}$. The other factors of $\gamma_j^{(n:I'|I")}$ define
respectively the big $x$ diagonal $D_x^\Phi$ and big $y$ diagonal
$D^\Phi_y$ on $X^\Phi$, i.e. \beql{}{
D^\Phi_x=p_{(X')^{I'}}^*(D_{X'}^{I'}),
D^\Phi_y=p_{(X")^{I"}}^*(D_{X"}^{I"}) } where $D_{X'}. D_{X"}$ are
the usual big diagonals.
 Note also that the $G_{j,J}$ are globally defined along $X^\Phi$.
Now define a line bundle on $X^\Phi$ as follows:
\beql{OEPhi}{OE^\Phi_{s,j}=\O_{X^\Phi}
(-(n-j+1)\del_x(X^\Phi)-(j-1)\del_y(X^\Phi)- D_x^\Phi-D_y^\Phi),}
($s$ will be omitted when understood).\par
In the irreducible case there is no distinction globally
between $I'$ and $I"$, so we may as well assume $I"=\emptyset$ and take $D^\Phi_y=0$; the $\del_x$ and $\del_y$ are
still defined, and different,
based on setting the appropriate coordinates equal to $\theta_x$ or $\theta_y$. With this understood, we still define $OE^\Phi_{s,j}$ as in \refp{OEPhi}.. Then the
foregoing calculations have the following conclusion
(NB we are using the quotient convention for projective bundles).

\begin{prop}\label{pol-scroll} Let $X_s$ be a singular fibre and
 $\Phi=(J:I'.|I".)$ a set of multipartition data where $I".=\emptyset$
if $X_s$ is irreducible. Define line bundles $OE^\Phi_{s,j}$ on
$X_s^\Phi$ by

\beql{OEPhi-gen}{
OE^\Phi_{s,j}=OE^{\Phi_{\max}}_{s,j}\otimes\O_{X^\Phi} } where
$\Phi_{\max}$ is the unique maximal partition dominating $\Phi$ and
$OE^{\Phi_{\max}}_{s,j}$ is defined by \refp{OEPhi}.. Then we have
an isomorphism of $\P^1$-bundles over $X^\Phi$
\beql{iso-scroll}{OF^\Phi_{s,j}\simeq\P(OE^\Phi_{s,j}\oplus
OE^\Phi_{s,j+1})} that induces an isomorphism
\beql{iso-polar}{\O(-\Gamma\scl m.)\otimes\O_{OF_s^\Phi}\simeq
\O_{\P(OE_{s,j}^\Phi\oplus OE^\Phi_{s,j+1})}(1);} these isomorphisms
are uniquely determined by the condition that over a neighborhood of
a point in $\mathcal S_{k',k"}\subset X_s^\Phi$,  they take the
generators $G_{j+k"}, G_{j+k"+1}$ to  generators
$\gamma_j^{(n:I'|I")},\gamma_j^{(n:I'|I")}$ of $OE^\Phi_{s,j},
OE^\Phi_{s, j+1}$, respectively.

\end{prop}
\begin{proof}
To begin with, note
that the requirement of compatibility of \refp{iso-scroll}. with \refp
{iso-polar}. determines it uniquely over any open set
of $X^\Phi_s$.\par
Next, note that
 every multipartition set $\Phi$ is dominated by
a maximal one and the appropriate $\P^1$ bundles and polarizations
restrict in the natural way. In fact, quite generally, whenever
$\Phi_1\prec\Phi_2$ are multipartitions, we have a Cartesian diagram
of polarized $\P^1$-bundles: \beql{}{
\begin{matrix}
OF_j^{\Phi_1}&\to&OF_j^{\Phi_2}\\
\downarrow&\square&\downarrow\\
X^{\Phi_1}&\to&X^{\Phi_2}.
\end{matrix}
}
 Therefore is suffices
to prove the assertions in case $\Phi$ is maximal. In that case the
foregoing discussion yields the claimed isomorphisms in a
neighborhood of the 0-stratum $\mathcal S_{0,0}$. A similar argument
applies in a neighborhood of a point in any other stratum. The
compatibility of all these isomorphisms with \refp{iso-polar}.
ensures that the local isomorphisms glue together to a global one.

\end{proof}
Note that \refp{iso-scroll}. reproves Lemma \ref{nodescroll0},
though the latter, of course, does not yield the 'correct'
polarization and is therefore of little use enumeratively.
\begin{example}\label{deg-pol-ord}
We have
\beql{}{OF_{s,1}^{(12:3|\emptyset)}=
\P_{X'}(\O(-2\theta_x)\oplus\O(-\theta_x))
}
Consequently
\beql{}{(-\Gamma\spr 3.)^2.OF_{s,1}^{(12:3|\emptyset)}=-3.
}
Of course, in this example ordering is irrelevant.

\end{example}
\subsubsection{General families: maximal multipartition}
We now take up the extension of Proposition \ref{pol-scroll} to the setting
of an arbitrary nodal family $X/B$, where a maximal node scroll
$F_j^n(\theta)$ is associated to a relative node $\theta$,
or more precisely to a boundary datum
$(T, \delta, \theta)$, as in \S\ref{globalization} and
\S\ref{globalization2}.
In this setting the scrolls $F^n_j(\theta)$  are (polarized,
via the discriminant) $\P^1$-bundles over $(X^\theta_T)\sbr m-n.$
defined for
all $1\leq j<n\leq m$, and we aim to identify them,
or rather more conveniently, their pullbacks over the ordered Hilbert scheme
$(X^\theta_T)\scl m-n.$. As in the foregoing discussion
in the 1-parameter case, the polarized $\P^1$-bundle
$F^n_j(\theta)$ is
just the projectivization of the rank-2 bundle that is
the direct sum of the invertible ideals of the  intermediate diagonals $OD^m_j(\theta), OD^m_{j+1}(\theta)$ via the map
'add $n\theta$'. Now these ideals were determined in
Proposition \ref{bundle-for-node-scroll}. Therefore we conclude

\begin{prop}\label{pol-scroll-global} Let $X/B$
be a family of nodal curves and $ (T, \delta, \theta)$ be a
boundary datum as in \S\ref{globalization}.
Let $OE^n_j(\theta)$ be the line bundle over $(X^\theta_T)\scl m-n.$ defined (in divisor notation)
by
\eqspl{OEtheta}{
OE_j^n(\theta)=& -(n-j+1)(\theta_x)\scl m-n.
-(j-1)(\theta_y)\scl m-n.-\Gamma\scl m-n.\\& +((\pi^\theta)\scl m-n.)^*(\binom{n-j+1}{2}\psi_x+
\binom{j}{2}\psi_y).
}
Then the pullback of the node scroll $F^n_j(\theta)$ on
$(X^\theta_T)\scl m-n.$ is polarized-isomorphic to
$\P(OE^n_j(\theta)\oplus OE_{j+1}^n(\theta)).$
\end{prop}
\begin{rem}
Note that interchanging the $x$ and $y$ branches along
$\theta$ interchanges
$OE^n_j(\theta)$ and $OE^n_{n-j+1}(\theta)$, hence also
the node scrolls $F^n_j$ and $F^n_{n-j}(\theta)$.
\end{rem}
\subsubsection{General families: nonmaximal multipartition}
Next we extend the basic formula \eqref{OEtheta} for the line bundles
$OE_j^\Phi(\theta)$ making up the node scroll to the case where
$\Phi$ is a general, not necessarily maximal, multi-partition.
This is essentially a matter of computing the restrictions
of the
various 'constituents' of $OE_j^\Phi(\theta)$ on a
general diagonal locus, and is  readily done based on our earlier results, notably Proposition \ref{disc.diag-ordered}.\par
Thus let $\Phi=
(J:I'.|I".)$ be a full multipartition on $[1,m]$ where
as usual we take $I".=\emptyset$ if $\theta$ is a nonseparating node. We set $n=|J|, k=m-n$ where we may assume $k\leq m-2$, and let
$$OD^\theta_\Phi\subset (X^\theta_T)^k,
\Gamma^\theta_\Phi\subset(X^\theta_T)\scl k.$$
be the associated diagonal loci. Note that these are ordinary
diagonal loci associated to the family $X^\theta_T$ (which is disconnected when $\theta$ is separating); in the nonseparating case, $\Gamma^\theta_\Phi=\Gamma_{(I'.)}$
(on $X^\theta_T$). Therefore, to start with, the restriction
of $\Gamma\scl k.$ on $\Gamma^\theta_\Phi$ is computed by
Proposition \ref{disc.diag-ordered}.
\par
 As for $\theta_x, \theta_y$, the restriction is quite elementary (and comes from the corresponding transversal
intersection on the cartesian product); namely, in the reducible case,
\eqspl{}{
\theta_x\scl k..\Gamma_\Phi^\theta=\sum\limits_\l|I'_\l|p_{\min(I'_\l)}
(\theta_x)\\
\theta_y\scl k..\Gamma_\Phi^\theta=
\begin{cases}\sum\limits_\l|I"_\l|p_{\min(I"_\l)}
(\theta_y), \quad \theta~ {\mathrm{ separating}}\\
\sum\limits_\l|I'_\l|p_{\min(I'_\l)}
(\theta_y), \quad \theta ~{\mathrm{ nonseparating}}
\end{cases}
}

Summarizing, we have
\begin{cor}In the situation above
we fix an arbitrary full multipartition $\Phi$ and
boundary datum $(T,\delta, \theta)$ and identify a divisor
class with the corresponding line bundle. Then we have on
$\Gamma_\Phi=\Gamma_\Phi^\theta$, suppressing the
node $\theta$ for brevity:\nl if $\theta$ is separating, \eqspl{}{
&OE^\Phi_j\sim -\Gamma\scl k..\Gamma_\Phi
  +\sum\limits_\l p_{\min(I'_\l)}^*(-|I'_\l|(n-j+1)\theta_x))\\ &+\sum\limits_\l
p_{\min(I"_\l)}^*(-|I"_\l|(j-1)\theta_y)
+((\pi^\theta)\scl m-n.)^*(\binom{n-j+1}{2}\psi_x+
\binom{j}{2}\psi_y).;}
if $\theta$ is nonseparating,
\eqspl{}{
OE^\Phi_j\sim& -\Gamma\scl k..\Gamma_\Phi
  +\sum\limits_\l p_{\min(I'_\l)}^*(-|I'_\l|((n-j+1)\theta_x +(j-1)\theta_y)\\
&+((\pi^\theta)\scl m-n.)^*(\binom{n-j+1}{2}\psi_x+
\binom{j}{2}\psi_y).}
\end{cor}
\subsubsection{Unordered cases}
Finally we carry our results over to the unordered case,
i.e. node scrolls over diagonal loci in the Hilbert
scheme itself. This is straightforward, as both the scrolls and related line bundles descend. To state the result, we use the following notation:
let
\beql{}{\phi=(n:\un'|\un")} be a full multidistribution,
i.e. a natural number plus 2 distributions  such that
the total length
$$n+\sum\limits_\l \un'(\l)+\sum\limits_\l \un"(\l)=m.$$
The shape of $\phi$, $(n:({n'.}^{\mu'.})|({n".}^{\mu".}))$
is defined as before.
To a multipartition $\Phi$ as above we associate the multidistribution
$$\phi=|\Phi|=(|J|:|I'.|||I".|).$$
The locus $\Gamma^\theta_\phi\subset (X^\theta_T)\sbr k.$,
$k\leq m-2$,
and over it, the scrolls $F_{j}^\phi(\theta)$ are defined as before.
 As before, we let
$$\varpi_{\Phi}:\Gamma_\Phi^\theta\to
\Gamma_\phi^\theta
$$
be the natural symmetrization map, of degree $a(n'.)a(n".)$.
We also define in the separating case, for collections $$\alpha'_1,...,\alpha'_{r'},
\alpha"_1,...,\alpha"_{r"}$$ of cohomology classes on $X$
('twisting classes'),
\beql{}{\Gamma_\phi^\theta\star_k[\alpha'.;\alpha".]
=s_k((X')\spr \mu'_1.[\alpha'_1],...,
(X")\spr \mu"_1.[\alpha"_1],...)
}
where $s_k$ is the $k$th
elementary symmetric function (in all the $r'+r"$
indicated variables). There is an analogous notion, with a single collection of twisting classes, in the nonseparating case.

Then the appropriate line bundles on $\Gamma_\phi$ are here
given up to numerical equivalence by:\par
$\theta$ \underline{separating}:
\eqspl{Ebundles}{
E_{j}^\phi(\theta)  \sim_{\mathrm{num}}
\Gamma\spr k..\Gamma_\phi^\theta
&+\Gamma_\phi^\theta\star_1
[-n'.(n-j+1)\theta_x\ ;
-n".(j-1)\theta_y]
\\
&+((\pi^\theta)\spr m-n.)^*(\binom{n-j+1}{2}\psi_x+
\binom{j}{2}\psi_y);}\par
$\theta$ \underline{nonseparating}:
\eqspl{Ebundles-nonsep}{
E_{j}^\phi(\theta)  &\sim_{\mathrm{num}}
\Gamma\spr k..\Gamma_\phi^\theta
+\Gamma_\phi^\theta\star_1
[-n'.(n-j+1)\theta_x\ ]\\
&+((\pi^\theta)\spr m-n.)^*(\binom{n-j+1}{2}\psi_x+
\binom{j}{2}\psi_y);}

These bundles have the property, easy to check, that they pull back
to $OE_{s,j}$ whenever $\phi$ is the multidistribution associated to
$\Phi$. This suffices to ensure they are the 'correct' bundles at
least up to torsion. Thus,
\begin{prop}\label{pol-scroll-unordered} For each boundary
datum $(T,\delta,\theta)$ and multidistribution $\phi$,
we have a polarized isomorphism of
$\P^1$-bundles over $(X^\theta_T)\sbr k.$
\beql{iso-scroll-unordered}{F^\phi_{j}(\theta)
\simeq\P(E^\phi_{j}(\theta)\oplus
E^\phi_{j+1}(\theta))}
where polarized means it induces an isomorphism
\beql{iso-polar}{\O(-\Gamma\spr m.).{F_{j}^\phi\theta)}
\simeq \O_{\P(E_{j}^\phi(\theta)\oplus E^\phi_{j+1}(\theta))}(1).}
\end{prop}
\begin{proof}
To begin with, in case $\phi$ is maximal the result follows from the
ordered case by (faithful) flatness of the symmetrization map.\par
Given this, one is reduced to checking that the pullback
$E^{\phi_{\max}}_{s,j}$ to $\Gamma_\phi$ is as in \refp{Ebundles}.
This can be done as in the case of relative diagonal loci (see
Proposition \ref{disc.diag})
\end{proof}
To state the next result compactly, we introduce the following
formal polynomial
\beql{}{s_k(a,b)=a^k+a^{k-1}b+...+b^k=(a^{k+1}-b^{k+1})/(a-b).} Thus
$s_0=1, s_1=a+b$ etc. Also, to avoid confusion, we recall that the
polarization on a node scroll is given by $-\Gamma\spr m.$ rather
than $+\Gamma\spr m.$.
\begin{cor}\label{disc^k.scroll} For a node scroll $F=F^\phi_{j}(\theta)$
and  all $\l\geq 2$, we have, setting $e_j=E^\phi_{j}(\theta)$:
\beql{}{(-\Gamma\spr m.)^\l|_F=s_{\l-1}(e_j,e_{j+1})
(-\Gamma\spr m.)-e_je_{j+1}s_{\l-2}(e_j,e_{j+1}).
}
In particular, for any twist $\alpha.$,
 $(-\Gamma\spr m.)^\l .F[\alpha.]$ is a linear combination of twisted node scrolls and twisted node sections.
\end{cor}
\begin{proof}
It follows from a standard, analogous formula for the self-intersection of the polarization on the projectivization
of an arbitrary decomposable bundle, which can be easily proved
by induction, starting from Grothendieck's formula for the
case $k=2$.
\end{proof}
\begin{rem}
Note that in our case, $e_j-e_{j+1}=\theta_x-\theta_y$, so the latter formula
can be rewritten as \beql{}{(-\Gamma\spr m.)^\l|_F=
((e_j^\l-e_{j+1}^\l)(-\Gamma\spr m.)-e_je_{j+1}
(e_j^{\l-1}-e_{j+1}^{\l-1}))/([\theta_x-\theta_y]).} The dividing by $[\theta_x-\theta_y]$
must be done with care, i.e. one must first expand the numerator in
some ring where $[\theta_x-\theta_y]$ is not a zero divisor, then do the
dividing, and only then impose the remaining relations defining the
Chow ring of $F$.
\end{rem}
\begin{rem}
It is worth noting that intersection products involving
the sections
$\theta_x, \theta_y$ are 'elementary' in view of the fact
that \eqspl{}{
\theta_x^j=(-\omega)^{j-1}.\theta_x, j\geq 1.
} Also, at least on the cartesian product $(X^\theta_T)^k$,
$p_i^*(\theta_x)$ is geometrically a copy of
$(X^\theta_T)^{k-1}$ embedded via inserting $\theta_x$ at the
$i$-th coordinate, and similarly for products $p_i^*(\theta_x)p_j^*(\theta_x), i\neq j$ etc. Also, $\theta_x\theta_y=0$, as the sections are disjoint.
On the other hand intersections on $T$ involving the
$\psi$ classes ultimately reduce to pure psi products,
which are the subject of the Witten conjecture, as proven
by Kontsevich \cite{Kon}.\par
Note that in the 'extreme' case $m=n$, the $E^\phi_j(\theta)$
and the node scroll $F^\phi_j(\theta)$ live on the base itself
$T$ of the boundary datum and we have
\eqspl{E-extreme}{
E^\phi_j(\theta)=\binom{m-j+1}{2}\psi_x+
\binom{j}{2}\psi_y.
}
\end{rem}
\begin{example}
For $m=n=2, F=F^2_1(\theta)$, we have
\eqspl{Gamma^k-on-F}{
(-\Gamma\spr 2.)^k|_F=(\psi_x^{k-1}+\psi_x^{k-2}\psi_y+
...+\psi_y^{k-1})(-\Gamma\spr 2.)-\psi_x\psi_y
(\psi_x^{k-2}+\psi_x^{k-3}\psi_y+
...+\psi_y^{k-2})
.} In particular, for $k=\dim(B)=\dim(F)=1+\dim(T)$, we have
\eqspl{}{
(-\Gamma\spr 2.)^k.F=\int\limits_T(\psi_x^{k-1}+\psi_x^{k-2}\psi_y+
...+\psi_y^{k-1}).
} Note that if $B=\overline{\mathcal M}_g$ and
$T=\overline{\mathcal M}_{i,1}\times\overline{\mathcal M}_{g-i,1}$,
$1\leq i\leq g/2$ (the usual $i$-th boundary component),
the latter integral reduces to
\[\int\limits_{\overline{\mathcal M}_i}\psi_x^{3i-3}
\int\limits_{\overline{\mathcal M}_{g-i}}\psi_y^{3(g-i)-3}
\]
\end{example}
 Note that \eqref{Gamma^k-on-F} and \eqref{Gamma2^k} together imply
\begin{cor}\label{polpowers,m=2}(i) The powers of the polarization on $X\sbr 2._B$ are
\eqspl{Gamma2-kpower}{
(-\Gamma\spr 2.)^k=-\Gamma[\omega^{k-1}]&+\\
\half\sum\limits_s\delta_{s*}( (\psi_x^{k-3}+\psi_x^{k-4}\psi_y+
...+\psi_y^{k-3})(-\Gamma\spr 2.)&-\psi_x\psi_y
(\psi_x^{k-4}+\psi_x^{k-5}\psi_y+
...+\psi_y^{k-4}))
} (ii) The image of the latter class on the symmetric product $X\spr 2._B$
equals
\eqspl{Gamma2-kpower-on-sym}{
-\Gamma[\omega^{k-1}]+
\half\sum\limits_s\delta_{s*}( (\psi_x^{k-3}+\psi_x^{k-4}\psi_y+
...+\psi_y^{k-3})
}
\end{cor}
\begin{proof}
\eqref{Gamma2-kpower} has been proved above; \eqref{Gamma2-kpower-on-sym}
follows because in the last summation in \eqref{Gamma2-kpower},  the terms
without $\Gamma\spr 2.$, i.e. the twisted node scroll,
collapses under the cycle map to $X\spr 2._B$.
\end{proof}

\begin{example}\label{deg-pol}$m=3, n=2, \dim(B)=1$:
\beql{}{(-\Gamma\spr 3.)^2.F_{1}^{(2:1|0)}(\delta)=-3
} (see Example \ref{deg-pol-ord}).
Consequently, in view of Corollary \ref{cube}, we conclude
\eqspl{fourth}{ \int\limits_{X\sbr 4._B}(\Gamma\spr 3.)^4=13\omega^2-9\sigma
} (recall that each $F^{(3,0|0)}_i, i=1,2$ is a line
with respect to the discriminant polarization $-\Gamma\spr 3.$.
\end{example}
\subsection{Tautological module}
We are now in position to give the formal definition of the tautological module $T^m$ and the proof of Theorem \ref{taut-module}.
\begin{defn}\label{taut-mod-def}
Given a cohomology theory $A^.$ and a $\Q$-subalgebra
$R\subset A^.(X)_\Q$
containing the canonical class $\omega$,
the tautological module $T^m_R$ is the $R$-submodule
of $A^.(X\sbr m._B)$ generated by the twisted diagonal classes $\Gamma\subp{n.^{\mu.}}[\alpha.]$ and the direct
images on $X\sbr m._B$ the twisted node scroll classes $F_{j}^\phi(\theta)[\alpha.]$ and the twisted node scroll sections $-\Gamma\spr m..F_{j}^\phi(\theta)[\alpha]$
as $(T,\delta, \theta)$ ranges over a fixed covering system
of boundary data for the family $X/B$. For the default
choice $R=\Q[\omega]$, we denote $T^m_R$ by $T^m$.
\end{defn}
\begin{proof}[Proof of Theorem \ref{taut-module}].
We wish to compute the product of a tautological class $c$ by
$\Gamma\spr m.$. If $c$ is a (twisted) diagonal class
$\Gamma\subp{n.^{\mu.}}[\alpha.]$, this is clear from
Proposition \ref{disc.diag}. If $c$ is a twisted
node scroll class
$F_{s,j}^\phi$, it is obvious. Finally if $c$ is
a node scroll section  $-\Gamma\spr m..F_{s,j}^\phi$
it is clear from the case $k=2$ of Corollary \ref{disc^k.scroll}.
\end{proof}
\begin{rem}
In the important special case of computing a power
$(\Gamma\spr m.)^k$ it is probably more efficient not to
proceed by simple recursion, but rather to apply
just Proposition \ref{disc.diag} repeatedly to
express $(\Gamma\spr m.)^k$ in terms of twisted diagonals
plus classes $(\Gamma\spr m.)^t.F$ for various $t$'s and
various $F$'s; then each of the latter classes can be computed
at once using Corollary \ref{disc^k.scroll}.
\end{rem}
\newsection{Tautological transfer and Chern numbers}
In this chapter we will complete the development of our intersection calculus. First we study the \emph{transfer} operation $\tau_m$, taking cycles on $X\sbr m-1._B$ to cycles
(of dimension 1 larger) on $X\sbr m._B$, via the flaglet Hilbert
scheme $X\sbr\mm._B$. In the Transfer Theorem \ref{taut-tfr} we will show
in fact that for any tautological class $u$ on $X\sbr m-1._B$,
the image $\tau_m(u)$ is a simple linear combination of
basic tautological classes on $X\sbr m._B$. We then review a
splitting principle established in \cite{R}, which expresses the Chern classes of the tautological bundle $\Lambda_m(E)$,
pulled back on $X\sbr \mm._B$, in terms of
those of $\Lambda_{m-1}(E)$, the discriminant polarization
$\Gamma\spr m.$, and base classes. Putting this result together with the Module Theorem and the Transfer Theorem yields the calculus for arbitrary polynomials in the Chern classes of $\Lambda_m(E).$
\newsubsection{Flaglet geometry and the transfer theorem}
In this section we study the $(m,m-1)$ flag (or 'flaglet') Hilbert
scheme, which we view as a correspondence between the Hilbert
schemes for lengths $m$ and $m-1$ providing a way of
transporting cycles, especially tautological ones, between these Hilbert
schemes. We will make strong use of the results of \cite{Hilb}.\par
Thus let
$$X\sbr m,m-1._B\subset X\sbr m.\times _B X\sbr m-1.$$
denote the flag Hilbert scheme, parametrizing
pairs of schemes $(z_1, z_2)$ satisfying
$z_1\supset z_2$. This comes equipped with a (flag) cycle
map
$$c_{m,m-1}:X\sbr m,m-1._B\to X\spr m,m-1._B,$$
where $X\spr m,m-1._B\subset X\spr m._B\times_B X\spr m-1._B$ is the
subvariety parametrizing cycle pairs $(c_{m}\geq c_{m-1}).$ Note
that the \emph{normalization} of $X\spr \mm._B$ may be identified with
$X\spr m-1._B\times _B X$; however the normalization map, though bijective,
is not an isomorphism. Note also that we also have an ordered
version $X\scl m,m-1._B$, with its own cycle map
$$oc_{m,m-1}:X\scl m,m-1._B\to X^m_B.$$
In addition to the obvious projections
\beql{}{\begin{matrix}&&X\sbr\mm._B&&\\
&p_m\swarrow&&\searrow p_{m-1}\\
X\sbr m._B&&&&X\sbr m-1._B
\end{matrix}
}
with respective generic fibres $m$ distinct points (corresponding to removing a point from a
given $m$-tuple)
and a generic fibre of $X/B$ (corresponding to adding a point
to a given $m-1$-tuple), $X\sbr \mm._B$ admits
a natural map
\beql{}{a:X\sbr\mm._B\to X,}
$$(z_1\supset z_2)\mapsto \mathrm{ann}(z_1/z_2)$$
(identifying $X$ with the Hilbert scheme of colength-1 ideals).
Therefore $X\sbr \mm._B$ admits a 'refined cycle map'
(factoring the flag cycle map) \beql{}{c:X\sbr \mm._B\to X\times_B X\spr m-1._B
}
$$c=a\times (c_{m-1}\circ p_{m-1}).$$
\par
Now in \cite{Hilb} (Theorem 5 et seq., especially
Construction 5.4 p.442) we worked out a
complete model for $X\sbr m,m-1._B$, locally over
$X\spr\mm._B$. Let
\begin{eqnarray}H_m&\subset& X\spr m._B\times\tilde C^m
_{[u.,v.]}
\subset X\spr m._B\times\P^{m-1}_{Z.},\\
H_{m-1}&\subset& X\spr m-1._B\times \tilde C^{m-1}
_{[u'.,v'.]}
\subset X\spr m-1._B\times\P^{m-2}_{Z'.}
\end{eqnarray}
be respective local models for $X\sbr m._B, X\sbr m-1._B$
as constructed in \S1 above, with coordinates
as indicated. Consider  the subscheme
\beql{}{H_{\mm}\subset H_m\times_BH_{m-1}
\times_{X\spr m._B\times X\spr m-1._B}X\spr \mm._B
}
defined by the equations
\beql{}{u_i'v_i=(\sigma^x_1-\sigma^{'x}_1)u_iv'_i,\ \
v'_iu_{i+1}=(\sigma^y_1-\sigma^{'y}_1)v_{i+1}u'_i,\ \
 1\leq i\leq m-2
}
or alternatively, in terms of the $Z$ coordinates,
\begin{eqnarray}\label{ZZ'} Z_iZ'_j&=&(\sigma^x_1-\sigma^{'x}_1)
Z_{i+1}Z'_{j-1},\ \  i+1\leq j\leq m-1\\ \nonumber
   &=& (\sigma^y_1-\sigma^{'y}_1)Z_{i-1}Z'_{j+1},l
\ \ 1\leq j\leq i-2.
\end{eqnarray}
To 'explain' these relations in part, note that
in the ordered model over $X^m_B$, we have
$$\sigma^x_1-\sigma^{'x}_1=x_m, \sigma^y_1-\sigma^{'y}_1
=y_m$$
and then the analogue of \refp{ZZ'}. for the $G$ functions
is immediate from \refp{Gdef2}..  Then the result of
\cite{Hilb}, Thm. 5, is that $H_{\mm}$, with its map to
$X\spr \mm._B$ is isomorphic to a neighborhood of the
special fibre over $(mp, (m-1)p)$ of the flag Hilbert
scheme $X\sbr \mm._B$. In fact the result of \cite{Hilb}
is even more precise and identifies $H_{\mm}$ with a subscheme
of $H_m\times_B H_{m-1}$ and even of $H_{m-1}\times_B
\tilde C^m\times_B X$, where the map to $X$ is the annihilator
map $a$ above.\par
As noted in \cite{Hilb}, Thm 5, the special fibre of
the flag cycle map on $H_{\mm}$,
aka the punctual flag Hilbert scheme, is a normal-crossing
chain of $\P^1$'s:
\beql{C-mm-1}{C^{\mm}=\td C^m_1\cup \td C^{m-1}_1\cup \td C^m_2\cup...
\cup \td C^{m-1}_{m-2}\cup \td C^m_{m-1}\subset
 C^m\times C^{m-1}.
}
where the embedding is via
$$\td C^m_i\to C^m_i\times \{Q^{m-1}_i\},\ \
 \td C^{m-1}_i\to \{Q^m_{i+1}\}
\times C^{m-1}_i$$
and in particular,
\beql{}{\td C^m_i\cap \td C^{m-1}_i=\{(Q^m_{i+1}, Q^{m-1}_i)\},
\td C^{m-1}_i\cap \td C^m_{i+1}=\{(Q^m_{i+1}, Q^{m-1}_{i+1})\}
}
where $Q^m_i=(x^{m-i+1}, y^i)$ as usual.
\begin{thm}\label{flag-blowup}
The cycle map $c_{\mm}$ exhibits the flag Hilbert scheme
$X\sbr\mm._B$ as the blow-up of the sheaf of ideals
$\I_{D^{\mm}}:=\I_{D^{m-1}}.\I_{D^m}$ on $X\spr\mm._B$.
\end{thm}
We shall not really need this result, just the explicit
constructions above, so we just sketch the proof, which
is analogous to that of Theorem \ref{blowup}. To begin
with, it is again sufficient to prove the
ordered analogue of this result, for the 'ordered flag
cycle map'
$$X\scl\mm._B\to X^m_B.$$
Here $X\scl\mm._B$ is embedded as a subscheme of
$X\scl m._B\times_{X^m_B}(X\scl m-1._B\times _B X)$,
and we have already observed that as such, it satisfies
the equations \refp{ZZ'}..\par
Now we will use the following construction. Let $\I_1, \I_2$
be ideals on a scheme $Y$. Then the surjection of
graded algebras
$$(\bigoplus\limits_n \I_1^n)\otimes
(\bigoplus\limits_n \I_2^n)\to
\bigoplus\limits_n (\I_1\I_2)^n$$ yields a closed immersion
\beql{}{\Bl_{\I_1\I_2}Y\hookrightarrow
\Bl_{\I_1}Y\times_Y \Bl_{\I_2}Y;
}
the latter is in turn a subscheme of the Segre subscheme
\beql{}{\P(\I_1)\times_Y \P(\I_2)\subset\P(\I_1\otimes\I_2)
.} In our case, Theorem \ref{blowup} allows us to
identify
$$OH_m\simeq \Bl_{\I_{OD^m}}X^m_B,\ OH_{m-1}\times_B X
\simeq\Bl_{\I_{OD^{m-1}.X^m_B}}X^m_B$$
(where $OH_m=H_m\times_{X\spr m._B}X^m_B$ etc.),
whence an embedding
\beql{}{\Bl_{\I_{OD^{\mm}}}X^m_B
\to OH_m\times_{X^m_B} (OH_{m-1}\times_BX)
}
As observed above, the generators $G_i\cdot G_j'$ satisfy the
analogues of the relations \refp{ZZ'}., so
the image is actually contained in $OH_{\mm}$,
so we have an embedding
\beql{}{\Bl_{\I_{OD^{\mm}}}X^m_B
\to OH_{\mm}.
}
We are claiming that this is an isomorphism. This
can be verified locally, as in the proof of Theorem
\ref{blowup}.
\qed\par
One consequence of the explicit local model for $X\sbr\mm._B$
is the following
\begin{cor}\label{q-flat}
\begin{enumerate}
\item The projection $q_{m-1}$ is flat, with 1-dimensional fibres;\item
Let $z\in X\sbr m-1._B$ be a subscheme of a fibre $X_s$, and let $z_0$ be the part of $z$ supported on nodes of $X_s$, if any. Then
 if $z_0$ is principal (i.e. Cartier)  on $X_s$, the
fibre $q_{m-1}\inv(z)$ is birational to $X_s$ and its general
members are equal to $z_0$ locally at the nodes.
\end{enumerate}
\end{cor}
\begin{proof}
(i) is proven in \cite{R}, and also follows easily from our
explicit model $H_{\mm}$. As for (ii), we may suppose, in the
notation of \cite{Hilb}, that $z_0$ is of type $I^n_i(a)$. Now if $z'\in q_{m-1}\inv(z)$, then the part $z'_0$ of $z'$ supported
on nodes must have length $n$ or $n+1$. In the former case
$z'_0=z_0$, while in the latter case
$z'_0$ must equal $Q^{n+1}_{i+1}$ by
 \cite{Hilb}, Thm. 5 p. 438, in which case $z'$ is unique, hence
 not general.
\end{proof}
Next we define the fundamental {transfer} operation. Essentially, this takes cycles from $X\sbr m-1._B$ to $X\sbr m._B$, but we  also allow the additional flexibility of
 twisting by base classes
via the $m$-th factor. Thus the \emph{twisted transfer map}
$\tau_m$ is defined by
\beql{}{\tau_m:A.(X\sbr m-1._B)\otimes A.(X)\to A_{.}(X\sbr m._B)_\Q,
}
$$\tau_m=q_{m*}(q_{m-1}^*\otimes a^*).$$
Note that this operation raises dimension by 1 and preserves
codimension. Suggestively, and a little abusively, we will
write a typical decomposable element of the source of $\tau_m$
as $\gamma\beta\subp{m}$ where $\gamma\in A.(X\sbr m-1._B), \beta\in A.(X)$.
The following result which computes $\tau_m$ is a key to our inductive computation of Chern numbers.
\begin{thm}
\label{taut-tfr} ({\bf{Tautological transfer}}) $\tau_m$ takes
tautological classes on $X\sbr m-1._B$ to tautological classes on
$X\sbr m._B$. More specifically we have, for any class $\beta\in A.(X)$:
\begin{enumerate}\item
for any twisted diagonal class $\gamma=\Gamma\subp{n.}[\alpha.\spr {\lambda.}.]$,
\beql{}{
\tau_m(\gamma\beta\subp{m})=\Gamma\subp{(n.)\coprod (1)}
[\alpha.\spr{\lambda.}.\coprod\beta]
;}
\item for any twisted node scroll $F[\alpha]=F_j^{(n:n'.|n".)}(\theta)[\alpha',\alpha"]$,

\eqspl{tfr-scroll}{\tau_m(F[\alpha]\beta\subp{m})&=\\
&F_j^{(n:(n'.)\coprod (1)|n".))}[\alpha'\coprod\beta,\alpha"]
 + F_j^{(n:n'.|(n")\coprod (1))}[\alpha',\alpha"\coprod\beta]
}
\item for any twisted node section $\Gamma\spr m-1..F[\alpha]$ with
$F[\alpha]$ as above and
\[ n'.=(n_1^{(\mu'_1)},...), n_1>n_2>...,\]
and ditto $(n".)$,
 we have
\eqspl{tfr-section}{\tau_m(\Gamma\spr m-1..F[\alpha]\beta\subp{m})&=\Gamma\spr m..\tau_m(F[\alpha]\beta\subp{m})\\  -
\sum\limits_\l \mu'_\l n'_\l F_j^{(n:(n'.)^{-\l}|n".)}(\theta)
[\delta_{(n'.),\l}^*(\alpha'\times\beta),\alpha"]&-
\sum\limits_\l \mu"_\l n"_\l F_j^{(n:(n'.)|n".^{-\l})}(\theta)
[\alpha',
\delta_{(n".),\l}^*(\alpha"\times\beta)]
\\  &- n\beta F_j^{(n+1:n'.|n".)}(\theta)
[\alpha]-n\beta F_j^{(n+1:n'.|n".)}(\theta)[\alpha]
}
where $\delta_{(n.),\l}$ is as in \eqref{delta-n-ell}
and where $(n'.)^{+\l}$ is the distribution obtained from $(n'.)$
by replacing one block of size $n'_\l$ with one of size $n'_\l+1$.
\end{enumerate}
\end{thm}
\begin{proof}
Part (i) is obvious. As for Part (ii), the flatness of $q_{m-1}$ allows us to work over a general $z\in F$ and then Corollary \ref{q-flat}, (ii) allows us to assume that the added point is
a general point on the fibre $X_s$, which leads to \refp{tfr-scroll}..\par
As for (iii), note that on $X\sbr\mm._B$, we can write
\beql{}{q_m^*\Gamma\spr m.=q_{m-1}^*\Gamma\spr m-1.+\Delta\spr m.}
where $\Delta\spr m.$ is the pullback from $X\sbr m-1._B\times_BX$ of the locus of pairs $(z,w)$ where $w$ is a point subscheme of $z$. Restricted on $q_{m-1}^*F$, $\Delta\spr m.$
remains an effective Cartier divisor which splits in two
parts, depending on whether the point $w$ added to a scheme $z\in F$
is in the off-node or nodebound portion of $z$. It is easy
to see that the first part gives rise to the 2nd and 3rd terms
in the RHS of \refp{tfr-section}.. \par
The analysis of the second part, which leads
to the coefficient $n$ in the last two terms of \eqref{tfr-section},
 is a bit more involved.
To begin with, it is easy to see that we may assume $m=n+1$,
in which case $F$ is just a $\P^1$, namely $C^{m-1}_j$.
Now as in \refp{C-mm-1}.,, the special fibre of the cycle map on
$q_{m-1}\inv(C^{m-1}_j)$, as a set, is given by
$\td C^m_j\cup \td C^{m-1}_{j}\cup\td C^m_{j+1}$
and this coincides as a set with $\Delta\spr m..q_{m-1}\inv(C^{m-1}_j)$. As $\td C^{m-1}_j$
 collapses under $q_m$, the proof will be complete if
  we can show that the multiplicity of $\td C^m_j$ and $\td C^m_{j+1}$ on $\Delta\spr m..q_{m-1}\inv(C^{m-1}_j)$
are both equal to $n=m-1$. We will do this for $\td C^m_{j+1}$
 as the case of  $\td C^m_j$ is similar and only notationally more cumbersome.\par
In that case, our assertion
will be an elementary
consequence of the equations on p. 440, l. 9-14 of
\cite{Hilb},
describing the local model $H_{\mm}$, as well as those on
p. 433, describing the analogous local model $H_m$, to which
equations we will be referring constantly in the
remainder of the present proof.
Note that $c_{m-i}$ (resp. $b'_{i-1}$) plays the role of the affine
coordinate $u_i/v_i$ (resp. $v'_{i-1}/u'_{i-1}$). Also our
$j+1$ is the $i$ there. We work on $q_{m-1}\inv(C^{m-1}_j)$.\par
\emph{Claim 1} :  Over a neighborhood of $Q^{m-1}_{j+1}$,
$q_{m-1}\inv(Q^{m-1}_{j+1})$ contains $\td C^m_{j+1}$
with multiplicity 1.\par
 To see this note that the defining equations
of $C^{m-1}_j$ on $X\sbr m-1._B$ are given by setting all $a'_k$ and $d'_k$, as well as $c'_{m-i-1}$ to zero . By loc. cit. p.433
l.9, this implies that we have $b'_1=...=b'_{i-2}=0$
on $q_{m-1}\inv(C^{m-1}_j)$ as well. At a general
point of $C^m_{j+1}$,  $c_{m-i}$ is nonzero. Therefore we may
consider $c_{m-i}$ as a unit. By loc. cit. p.440, eq. (15), we conclude $a_{m-i}=0$. From this we see easily
that all $a_k=d_k=0$ except $d_{i-1}$, which is a local equation for $\td C^m_{j+1}$, while $b'_{i-1}$ is a coordinate along $C^{m-1}_j$ having $Q^{m-1}_{j+1}$ as its unique zero. Now by p.440 l. 14, $b'_{i-1}$ and $d_{i-1}$ differ by the multiplicative unit $-c_{m-i}$, therefore $b'_{i-1}$ generically cut out exactly $\td C^m_{j+1}$, which proves Claim 1.\par
\emph{Claim 2 } : The restriction of $\Delta\spr m.$ on $\td C^{m-1}_j$
is  the subscheme $((b'_{i-1})^n)=(d_{i-1}^n)$.
\par
To prove Claim 2 we can pull back to the ordered version where
the pullback of $\td C^{m-1}_j$ is totally ramified, hence can
be written as $(m-1)U$, where $U$ maps isomorphically to $C^{m-1}_j$.
On the other hand, on the ordered version of $q_{m-1}\inv(C^{m-1}_j)$, we have
$$x_1=...=x_m=y_1=...=y_{m-1}=0$$
(this by the vanishing of all the $a'_k,d'_k$,
which are the elementary symmetric functions in $x_1,...x_{m-1}$ and $y_1,...,y_{m-1}$, respectively,
and the vanishing of $a_{m-i}=x_1+...+x_m$). Therefore, the pullback of $\Delta\spr m.$ is defined by the single
nonzero generator, that is
$$\prod\limits_{k=1}^{m-1}(y_m-y_k)=y_m^{m-1}\sim d_{i-1}^{m-1},$$ $d_{i-1}$ being a coordinate on $U\simto C^{m-1}_j$. From this Claim 2 is obvious.\par
The conjunction of Claims 1 and 2 completes the proof of the Proposition.
\end{proof}

\newsubsection{Full-flag transfer and Chern numbers}
We are now ready to tackle the computation of Chern numbers,
and in fact all polynomials in the Chern classes
of the tautological bundle on the relative Hilbert scheme
$X\sbr m._B$. The computation is based on passing from
$X\sbr m._B$ to the corresponding full-flag Hilbert scheme
$W=W^m(X/B)$ studied in \cite{R} and a \emph{
diagonalization theorem} for the total Chern class of
(the pullback of) a tautological bundle on $W$, expressing it
either as a simple
(factorable) polynomial in diagonal classes
induced from the various $X\sbr n._B,\; n\leq m$, plus base classes, or, more conveniently, as the product of the Chern class of a smaller tautological bundle and a diagonal class.
Given this, we can compute Chern numbers essentially by repeatedly applying
the transfer calculus of the last section.\par
We start by reviewing some results from \cite{R}.
Let
 \beq W^m=W^m(X/B)\stackrel{\pi\spr m.}{\longrightarrow} B \eeq denote the relative
 flag-Hilbert scheme of $X/B$, parametrizing flags
 of subschemes \beq z.=(z_1<...<z_m) \eeq
 where $z_i$ has length $i$ and $z_m$ is contained in some fibre
 of $X/B$. Let \beq w^m:W^m\to X\sbr m._B, w^{ m,i}.:W^m\to X\sbr i._B \eeq be the canonical
 (forgetful) maps. Let \beq a_i:W^m\to X \eeq be  the canonical map
 sending a flag $z.$ to the 1-point support of $z_i/z_{i-1}$
 and \beq a^m=\prod a_i:W^m\to X^m_B \eeq their (fibred) product,
 which might be called the 'ordered cycle map'.
 Let \beq \I_m<\O_{X\sbr m._B\times_BX}\eeq
 be the universal ideal of colength $m$.
 For any coherent sheaf on $X$, set
 \beq \lambda_{m}(E)=p_{{X\sbr m._B*}}(p_{X}^*(E)\otimes
 (\O_{X\sbr m._B\times_BX}/\I_m))\eeq
  These are called the
 \emph{tautological sheaves} associated to $E$; they are locally free if $E$ is. Abusing notation,
 we will also denote by $\lambda_m(E)$ the pullback of the
  tautological sheaf to appropriate flag Hilbert schemes
  mapping naturally to $X\sbr m._B$, such as $W^m$ or $X\sbr
  \mm._B$. With a similar convention, set
  \beql{}{\Delta\spr m.=\Gamma\spr m.-\Gamma\spr m-1.
  .}
 The various tautological sheaves form
 a flag of quotients on $W^m$:
 \beql{fu--flag}{...\lambda_{m.i}(E)\twoheadrightarrow \lambda_{m,i-1}(E)
 \twoheadrightarrow...}
 This flag makes possible a simple formula for the total Chern class of the tautological bundles, namely the following
 \emph{diagonalization theorem} (\cite{R}, Cor. 3.2):
 \begin{thm}\label{tautbun-W}
 The total Chern class of the tautological bundle $\lambda
 _m(E)$ is given by
 \beql{}{c(\lambda_m(E))=\prod_{i=1}^m c(a_i^*(E)(-\Delta\spr i.))
 }
 \end{thm}
 An analogue of this, more useful for our purposed, holds already on the flaglet Hilbert scheme. It can be proved in the same way, or as an easy consequence of Thm \ref{tautbun-W}
 \begin{cor}\label{tautbun}
 We have an identity in $A.(X\sbr \mm._B)_\Q$:
 \beql{}{c(\lambda_m(E))=c(\lambda_{m-1}(E))c(a_m^*(E)
 (-\Delta\spr m.)).
 }

 \end{cor}
 \begin{proof}
 By Theorem \ref{tautbun-W}, the RHS and LHS pull back to the same class in $W^m$. As the projection $W^m\to X\sbr \mm._B$ is generically finite, they agree mod torsion.
 \end{proof}

 Motivated by this result we make the following definition.
\begin{defn} Let $R$ be a $\Q$-subalgebra of $A(X)$ containing
the canonical class $\omega$. The Chern tautological ring on $X\sbr m._B$,
denoted
$$TC^m_R=TC^m_R(X/B),$$ is the $R$-subalgebra of $A(X\sbr m._B)_\Q$ generated by
  the Chern classes of $\lambda_m(E)$ and the discriminant class $\Gamma\spr m.$.
   \end{defn}
   \begin{rem}
   If $E$ is a line bundle, then it is easy to see from Theorem
   \ref{tautbun-W} that
   $$c_1(\lambda_m(E))=mc_1(E)-\Gamma\spr m..\qed$$
   \end{rem}
   The following is the main result of this paper.
  \begin{thm}
  There is a computable inclusion
  \beql{}{TC^m_R\to T^m_R.}
  More explicitly, any polynomial in the Chern classes of $\lambda_m(E)$, in particular the Chern numbers, can be computably expressed as a linear combination of standard
  tautological classes: twisted diagonal classes, twisted node scrolls, and twisted node sections.
  \end{thm}
\begin{proof} For $m=1$ the statement is essentially vacuous.
For $m=2$ it is a consequence of the Module Theorem
\ref{taut-module}. For general $m$, we
assume inductively
the result is true for $m-1$. Given any polynomial $P$ in the
Chern classes of $\lambda_m(E)$, Corollary \ref{tautbun} implies
that we can write its pullback on $X\sbr\mm._B$ as a sum of terms of
the form $p_{X\sbr m-1._B}^*Q.(\Gamma\spr m.)^k.S$ where $Q\in
TC^{m-1}_R$. By induction, $Q\in T^{m-1}_R$, so by the Transfer Theorem
\ref{taut-tfr}, $\tau_m(Q)\in T^m_R$. By the projection formula and
the Module Theorem \ref{taut-module}, it follows that $P\in T^m_R$.
\end{proof}
\begin{rem}
This result suggests the natural question: is $T^m_R$ a ring? more
ambitiously, is the inclusion $TC^m_R\to T^m_R$ an equality?
\end{rem}
\newsubsection{Example: trisecants to one space curve curve}
If $X$ is a smooth curve of degree $d$ and genus $g$ in $\P^3$, the virtual degree of its
trisecant scroll, i.e. the virtual number of trisecant lines to $X$ meeting a generic line, is given by $c_3(\bigwedge\limits^2 (\Lambda_3(\O_X(1)))$, which can be easily computed to be
\eqspl{}{
\frac{1}{6}(2d^3-12d^2+16d-3d(2g-2)+6(2g-2))
}
\newsubsection{Example: trisecants in a pencil } With $X/B$ as above ($B$ a smooth curve),
 suppose \beq f:X\to\P^{2m-1}\eeq is a
 morphism. One, quite special, class of examples of this situation
 arises as what we call a {\it{generic rational pencil}}; that is,
 generally,
 the normalization of the family of rational curves of fixed degree $d$
  in $\P^r$ (so
 $r=2m-1$ here) that are incident to a generic collection $A_1,
 ...A_k$ of linear spaces, with
 \beq (r+1)d+r-4=\sum(\text{codim}(A_i)-1); \eeq see \cite{R3} and
 references therein, or \cite{RA} for an 'executive summary'. While
 our result seems new in this case, we note that it applies to curves of
 arbitrary genus.
 \par Returning to the general situation,
 one expects a finite number $N_m$ of curves $f(X_b)$ to
 admit an m-secant $(m-2)$-plane, and this number can be evaluated as
 follows. Let $G=G(m-1,2m)$ be the Grassmannian of $(m-2)$-planes
  in $\P^{2m-1}$,
 with rank-$(m+1)$ tautological subbundle $S$, and let $L=f^*\O(1).$
 Then \beq m!N_m=
 \int\limits_{W^m\times G}c_{m(m+1)}(S^*\boxtimes w^*\lambda_m(L)) \eeq
 \beq =\int\limits_{W^m\times G}c_{m+1}(S^*(L\subpr 1.))
 c_{m+1}(S^*(L\subpr 2.-\Delta\spr 2.))
 \cdots c_{m+1}(S^*(L\subpr m.-\Delta\spr m.)) \eeq
 \beq =\int\limits_{W^m\times G}
 \prod\limits_{i=1}^m(\sum\limits_{j=0}^{m+1}\binom{m+1}{j}
 c_{m+1-j}(S^*)(L\spr
 i.-\Delta\spr i.)^j) \eeq
 \beq =\sum\limits_{|(j.)|=m+1}\int\limits_Gc_{m+1-j_1,...,m+1-j_m}(S^*)
 \int\limits_{W^m}(L\subpr
 1.)^{j_1}(L\subpr 2.-\Delta \spr 2.)^{j_2}...(L\subpr m.-\Delta\spr
 m.)^{j_m} \eeq
where $c_{u,v,w...}=c_uc_vc_w\cdots$ and, applied to $S^*$, represents
the condition that an $(m-2)$-plane in $\P^{2m-1}$ meet a generic
collection of planes of respective dimensions $u,v,w,..$. Note that only terms with
$j_1+...+j_k\leq k+1, \forall k$, can contribute. By the intersection calculus developed above,
this number can be computed in terms of the characters
\[b=L^2,
d=\deg_\pi(L), \omega^2, \sigma, \omega.L, \deg_\pi(\omega)=2g-2,
g=\text{fibre genus;}\] in the generic rational pencil case, all these
characters can be computed by recursion on $d$.\par Suppose now that
$m=3$, where the only relevant $(j.)$ are
\beq (2,1,1), (1,1,2), (2,0,2), (1,2,1),(1,0,3),
(0,3,1),(0,2,2),(0,1,3),    (0,0,4). \eeq
 In each of these cases, it is easy to see that the $G$
 integral evaluates to 1. The $W$ integrals
 may be
evaluated by the calculus developed above. The general
procedure is to proceed inductively, each time transferring the
leftmost factor from the $X\sbr i._B$ from whence it came to
$X\sbr i+1._B$. We will repeatedly be using Corollaries
 \ref{disc-square} and \ref{cube}, as well as
 standard projection formulas (for the symmetrization map). After the first transfer (to $X\sbr 2._B$, we will treat the resulting term as polynomial in $\Gamma\spr 3.$ and break things up according to the power of $\Gamma\spr 3.$ involved.
 The main multiplication rules to be used are the following.
 We will use the notation $[\alpha,...]$, for a base class $\alpha$, in place of $\Gamma\subp{1.}[\alpha,...]$, where $(1.)$ is a trivial (singleton blocks only) distribution.
 We also recall that $\Gamma\spr 3.[\alpha]$ is short for
 $\Gamma\spr 3.[\alpha,1]$).\ls
 {\bf\underline{Multiplication rules}}
\begin{enumerate}

 \item
 $$ [\alpha,\beta,\gamma].\Gamma\spr 3.= 2(\Gamma\spr 3.[\alpha.\beta,\gamma]+\Gamma\spr 3.[\alpha.\gamma,\beta]+
 \Gamma\spr 3.[\beta.\gamma,\alpha])$$
\item  $$\Gamma\spr 3.[\alpha,\beta].\Gamma\spr 3.=
 \Gamma\subp{3}[\alpha.\beta]-\Gamma\spr 3.[\alpha.\omega,\beta]$$
 \end{enumerate}
 \ls
 \flushright{$\lrcorner$}
\flushleft{ The detailed computation follows.}
\begin{enumerate}
 \itemul{(2,1,1)} that is,
 $ L\subp{1}^2(L\subp{2}-\Delta\spr 2.)(L\subp{3}-\Delta\spr{3}.)$.
  First,
  $$\tau_2(L\subp{1}^2(L\subp{2}-\Delta\spr 2.))=
  [L^2,L]-2\Gamma\spr 2.[L^2].$$ Next,
  $$\tau_3(([L^2,L]-2\Gamma\spr 2.[L^2])L\subp{3}=
  [L^2,L\spr 2.]-2\Gamma\spr 3.[L^2,L]=bd^2-bd,$$
$$\tau_3([L^2,L]-2\Gamma\spr 2.[L^2]).\Gamma\spr 3.
=4\Gamma\spr 3.[L^2,L]-2\Gamma\subp{3}[L^2]=2bd-2b$$ Thus the total
is $bd^2-3bd+2b=$\fbox{$b(d-1)(d-2)$}.

\itemul{(1,1,2)}: We treat this as a polynomial in $\Gamma\spr 3.$. The terms are as follows. Degree 0:
\eqsp{
\tau_3(\tau_2(L\subp{1}(L\subp{2}-\Gamma\spr 2.)(\Gamma\spr 2.)^2)+2\tau_2(L\subp{1}(L\subp{2}-\Gamma\spr 2.)\Gamma\spr 2.)
L\subp{3}+\tau_2(L\subp{1}(L\subp{2}-\Gamma\spr 2.))L\subp{3}^2)
=\\ \intertext{(as the $(\Gamma\spr 2.)^2$ doesn't contribute for dimension reasons)}
4\Gamma\spr 3.[L^2,L]+4\Gamma\spr 3.[L.\omega,L]+[L\spr 2.,L^2]-2\Gamma\spr 3.[L,L^2]= 2bd+2dL.\omega+bd^2-bd=\\
bd+bd^2+2dL.\omega
}
degree 1: similarly
\eqsp{
-2\tau_3(2\Gamma\spr 2.[L^2]+2\Gamma\spr 2.[L.\omega]+
[L\spr 2.]L\subp{3}-2\Gamma\spr2.[L]L\subp{3})\Gamma\spr 3.=\\
-2(2\Gamma\spr 3.[L^2]+2\Gamma\spr 3.[L.\omega]+[L\spr 3.]-2\Gamma\spr 3.[L,L])\Gamma\spr 3.=\\
-2(2\Gamma\subp{3}[L^2]+2\Gamma\subp{3}[L.\omega]+6\Gamma\spr 3.[L^2,L]-2\Gamma\subp{3}[L^2]+2\Gamma\spr 3.[L.\omega,L])=\\
-2(2b+2L.\omega+3bd-2b+dL.\omega)=
-2(2L.\omega+3bd+dL.\omega)
}
degree 2:
\eqsp{
([L\spr 2.]-2\Gamma\spr 3.[L])(\Gamma\spr 3.)^2=
(2\Gamma\spr 3.[L^2]+4\Gamma\spr 3.[L,L]-2\Gamma\subp{3}[L]
+2\Gamma\spr 3.[L.\omega])\Gamma\spr 3.=\\
2\Gamma\subp{3}[L^2]+4\Gamma\subp{3}[L^2]-4\Gamma\spr 3.[L.\omega,L]+6\Gamma\subp{3}[L.\omega]+2
\Gamma\subp{3}[L.\omega]= 6b-2dL.\omega+8L.\omega
}
total:\framebox{$
-5bd+bd^2+6b-2dL.\omega+4L.\omega
$}
\itemul{(2,0,2)} This case is similar to (1,1,2) and easier.
The result is
\begin{center}
{\framebox{$-2bd-b(2g-2)+2b$.}}
\end{center}

\itemul{(1,2,1)} This is again quite similar to the (2,1,1) case treated above, and yields

\begin{center}{\framebox{$(bd-2b-L.\omega)(d-2)$}}.
\end{center}

\itemul{(1,0,3)}
Again we consider this as a polynomial in $\Gamma\spr 3.$ and
compute term by term. For degree 0 we have \eqsp{ \tau_3(6\Gamma\spr
2.[L](L\subp 3)^2-6\Gamma\spr 2.[L.\omega]L\spr 3.)= 3bd-3dL.\omega
 }
 degree 1:
 \eqsp{
 -3([L,L^2]+4\Gamma\spr 3.[L,L]-2\Gamma\spr 3.[\omega.L])\Gamma\spr 3.=\\
 -3(
 2\Gamma\spr 3.[L,L^2]+2\Gamma\spr3.[L^2,L]+
 4\Gamma\subp 3 [L^2]-4\Gamma\spr 3.[L.\omega,L]-2
 \Gamma\subp {3} [\omega.L])=\\
 -3(bd+bd+4b-2d\omega.L-2\omega.L)
 }
For degree 2:
\eqsp{
3([L\spr 2.]+2\Gamma\spr 3.[L])(\Gamma\spr 3.)^2=
3(2\Gamma\spr 3.[L^2]+4\Gamma\spr 3.[L,L]+2\Gamma\subp 3[L]-2
\Gamma\spr 3.[\omega.L])\Gamma\spr 3.=\\
3(2\Gamma\subp3[L^2]+4\Gamma\subp3[L^2]-4\Gamma\spr 3.[\omega.L,L]+\-6\Gamma\subp3[\omega.L]-2\Gamma\subp3[\omega.L])=\\
3(2b+4b-2d\omega.L-6\omega.L-2\omega.L)
}

 For degree 3: by Corollary \ref{cube}, we get \eqsp{24\omega.L-d\omega^2
 }
 Summing up, we get \framebox{$-3bd-3dL.\omega+6b+6L.\omega-d\omega^2
 $}

\item[\underline{(0,3,1)}] Expanding, we get
\eqsp{
\tau_2((L\subp{2}-\Delta\spr 2.)^3)=
-6\Gamma\spr 2.[L^2]-6\Gamma\spr 2.[\omega.L]-2(\Gamma\spr 2.)^3=
-3b-3\omega.L-(\omega^2-\sigma)
}
Since this is a point cycle, multiplying by $L\subp{3}-\Delta\spr 3.$
or $L\spr 3.-\Gamma\spr 3.$ again just multiplies the coefficient by
$d-2$, for a total of
\framebox{$(-3b-3L.\omega-(-\sigma+\omega^2))(d-2)$}.

\item[\underline{(0,2,2)}]
Again working degree by degree in $\Gamma\spr 3.$, we get:\par
degree 0:
\eqsp{
\tau_3(2\tau_2((L\subp2^2-2L\subp2\Gamma\spr 2.+
(\Gamma\spr 2.)^2)\Gamma\spr 2.)L\subp3)+
\tau_3(\tau_2(L\subp2^2-2L\subp2\Gamma\spr 2.+
(\Gamma\spr 2.)^2)(L\subp3)^2)=\\
4\Gamma\spr 3.[L^2,L]+8\Gamma\spr 3.[L.\omega,L]+
[\omega^2-\sigma,L]-4\Gamma\spr 3.[L,L^2]-2\Gamma\spr 3.
[\omega,L^2]=\\
4d\omega.L+2d(\omega^2-\sigma)-b(2g-2)
} degree 1:
\eqsp{
-2\tau_3(\tau_2((L\subp2^2-2L\subp2\Gamma\spr 2.+
(\Gamma\spr 2.)^2)\Gamma\spr 2.)+
\tau_2(L\subp2^2-2L\subp2\Gamma\spr 2.+
(\Gamma\spr 2.))L\subp3)\Gamma\spr 3.=\\
-2(2\Gamma\spr 3.[L^2]+4\Gamma\spr 3.[\omega.L]
+[\omega^2+\sigma]+[L^2,L]-4\Gamma\spr 3.[L,L]
-2\Gamma\spr 3.[\omega,L])\Gamma\spr 3.=\\
-2(2b+4\omega.L+2(\omega^2-\sigma)+2bd-4b+2d\omega.L
-2\omega.L+2d\omega^2)=\\
4b-4\omega.L-4(\omega^2-\sigma)-4bd-4d\omega.L-2d\omega^2
}
degree 2:
\eqsp{
([L^2]-4\Gamma\spr 3.[L]-2
\Gamma\spr 3.[\omega]+F_1^{(2:1|0)}+F_1^{(2:0|1)})(\Gamma\spr 3.)^2=\\
(4\Gamma\spr 3.[L^2]+2\Gamma\spr 3.[1,L^2]-4\Gamma\subp3[L]+4\Gamma\spr 3.[L.\omega]-2\Gamma\subp3[\omega]+2\Gamma\spr 3.[\omega^2])\Gamma\spr 3.-6\sigma=\\
4\Gamma\subp3[L^2]+2\Gamma\subp3[L^2]-2\Gamma\spr 3.[\omega,L^2]
+12\Gamma\subp3[L.\omega]+4\Gamma\subp3[L.\omega]+6\Gamma\subp3[\omega^2]
+2\Gamma\subp3[\omega^2]-6\sigma=\\
=4b+2b-b(2g-2)+12L.\omega+4L.\omega+6\omega^2+2\omega^2-6\sigma=
6b-b(2g-2)+16\omega.L+8\omega^2-6\sigma
} The total is \framebox{$
-2d\sigma+10b+12\omega.L+4\omega^2-2\sigma-4bd-2b(2g-2)
$}

\item[\underline{(0,1,3)}] Again we consider this as polynomial in $\Gamma\spr 3.$.
The terms are:\par degree 0:
\eqsp{
\tau_3((L\subp{2}-\Gamma\spr 2.)(3L\subp{3}^2\Gamma\spr 2.
+3 L\subp{3}(\Gamma\spr 2.)^2)&=6\Gamma\spr 3.[L,L^2]
+6\Gamma\spr 3.[\omega,L^2]
-6\Gamma\spr 3.[\omega.L,L]-3[(\omega^2-\sigma),L]=\\
&3db+3b(2g-2)-3d\omega.L-3d(\omega^2-\sigma)
}
\par degree 1:
\eqsp{
-3\tau_3((L\subp{2}-\Gamma\spr 2.)(L\subp{3}^2+2L\subp{3}\Gamma\spr 2.
+(\Gamma\spr 2.)^2))\Gamma\spr 3.=\\
3(-[L,L^2]+2\Gamma\spr 3.[1,L^2]
-4\Gamma\spr 3.[L,L]+4\Gamma\spr 3.[\omega,L]-2\Gamma\spr 3.[\omega.L]
+[\omega^2-\sigma])\Gamma\spr 3.=\\
3(-2bd-b(2g-2)-2b+2d\omega.L+2d\omega^2+2\omega.L
+2(\omega^2-\sigma))
}\par degree 2:
\eqsp{
&3\tau_3((L\subp{2}-\Gamma\spr 2.)(L\subp{3}+\Gamma\spr 2.))(\Gamma\spr 3.)^2
=\\
&3([L,L]-2\Gamma\spr 3.[1,L]+2\Gamma\spr 3.[L]+2\Gamma\spr 3.[\omega])
(\Gamma\spr 3.)^2
=\\
&3((2\Gamma\spr 3.[L^2]+4\Gamma\spr 3.[L,L]+2\Gamma\spr 3.[\omega,L]
-2\Gamma\subp{3}[L]-2\Gamma\spr3.[L.\omega] +2\Gamma\subp{3}[L]
-2\Gamma\spr 3.[\omega^2]+2\Gamma\subp{3}[\omega])
\Gamma\spr 3.=\\
&3(2\Gamma\subp{3}[L^2]-4\Gamma\spr
3.[L.\omega,L]+4\Gamma\subp{3}[L^2]
-2\Gamma\spr 3.[\omega^2,L]+2\Gamma\subp{3}[\omega.L]
+6\Gamma\subp{3}[L.\omega]
-2\Gamma\subp{3}[L.\omega]
-6\Gamma\subp{3}[L.\omega]\\
&-2\Gamma\subp{3}[\omega^2] -6\Gamma\subp{3}[\omega^2])=
 3(6b-2dL.\omega-d\omega^2-8\omega^2)
 }
 degree 3:
 \eqsp{
-\tau_3(L\subp{2}-\Gamma\spr 2.)(\Gamma\spr 3.)^3=
-([L]-2\Gamma\spr 3.)(\Gamma\spr 3.)^3=\\
\intertext{(using \ref{cube} and \ref{deg-pol})}
=12L.\omega-d\omega^2+26\omega^2+2(-6\sigma-3\sigma)+d\sigma
} Total:\framebox{$
-3db-3d\omega.L-d\omega^2+4d\sigma+12b+18\omega.L+8\omega^2-24\sigma
$}
\itemul{(0,0,4)}Proceeding as above, we get: degree 0:
\eqsp{
\tau_3(12(\Gamma\spr 2.)^2(L\subp{3})^2+8(\Gamma\spr 2.)^3L\subp{3})
=-12\Gamma\spr 3.[\omega,L^2]+8\Gamma\spr 3.[\omega^2,L]+
4[\omega^2-\sigma,L]=\\
-6b(2g-2)+4d\omega^2+4d(\omega^2-\sigma).
}
degree 1:
\eqsp{
-4\tau_3((\Gamma\spr 2.+L\subp{3})^3)\Gamma\spr 3.=
-4\tau_3(3\Gamma\spr 2.L\subp{3}^2
+3(\Gamma\spr 2.)^2L\subp{3}+(\Gamma\spr 2.)^3)\Gamma\spr 3.=\\
-4(6\Gamma\spr 3.[1,L^2]-6\Gamma\spr 3.[\omega,L]+[\omega^2-\sigma])\Gamma\spr 3.
=\\ -4(6\Gamma\subp{3}[L^2]-6\Gamma\spr 3.[\omega,L^2]-
6\Gamma\subp{3}[\omega.L]+6\Gamma\spr 3.[\omega^2,L]+2[\omega^2-\sigma])=\\
-24b+12b(2g-2)+24\omega.L-8\omega^2+8\sigma
}
degree 2:\eqsp{
+6\tau_3((L\subp{3}+\Gamma\spr 2.)^2)(\Gamma\spr 3.)^2=
6\tau_3(L\subp{3}^2+2L\subp{3}.\Gamma\spr 2.+(\Gamma\spr 2.)^2)
(\Gamma\spr 3.)^2=\\
6([L^2]+4\Gamma\spr 3.[L]-2\Gamma\spr 3.[\omega]+\half(F^{(2:1|0)}_1
+F^{(2:0|1)}_1))(\Gamma\spr 3.)^2=\\
6(4\Gamma\spr 3.[L^2]+2\Gamma\spr3.[1,L^2]+4\Gamma\subp{3}[L]
-4\Gamma\spr3.[L.\omega]-2\Gamma\subp{3}[\omega]+2
\Gamma\spr 3.[\omega^2]+\half(F^{(2:1|0)}_1
+F^{(2:0|1)}_1)(\Gamma\spr 3.))\Gamma\spr 3.=\\
6(4\Gamma\subp{3}[L^2]+2\Gamma\subp{3}[L^2]-2\Gamma\spr3.[\omega,L^2]
-12\Gamma\subp{3}[L.\omega]
-4\Gamma\subp{3}[L.\omega]
+6\Gamma\subp{3}[\omega^2]
+2\Gamma\subp{3}[\omega^2]-3\sigma)=\\
6(4b+2b-b(2g-2)-12L.\omega-4L.\omega+6\omega^2+6\omega^2+3\sigma)
=6(6b-b(2g-2)-16L.\omega+8\omega^2-3\sigma)
}
degree 3:\eqsp{
-4\tau_3(\Gamma\spr 2.+L\subp{3})(\Gamma\spr 3.)^3=
-4(2\Gamma\spr 3.+[L])(\Gamma\spr 3.)^3=\\
-4(-24\Gamma\subp{3}[\omega.L]+2\Gamma\spr 3.[\omega^2,L]
+2(\Gamma\spr 3.)^4)=-4(-24\omega.L+d\omega^2+2(12\Gamma\subp{3}[\omega^2]
+\Gamma\subp{3}[\omega^2]-6\sigma-3\sigma)=\\
-4(-24\omega.L+d\omega^2-d\sigma+26\omega^2-18\sigma) } degree 4:\eqsp{
\tau_3(\tau_2(1))
(\Gamma\spr 3.)^4=6(13 \omega^2-9\sigma) } Total:\framebox
{$
12b+24\omega.L+14\omega^2 $}\par
{\bf {Grand total}}:
\framebox{${(3d^2-25d+60)b+(-12d+72)L.\omega+(-3d+28)\omega^2
-3b(2g-2)+(3d-20)\sigma}$}

\end{enumerate}
This formula has been obtained by other means by Ethan Cotterill \cite{cot}
\newsubsection{Example: double points}
Let $X/B$ be an arbitrary nodal family and $f:X\to\P^n$ a morphism. Consider the relative double points of $f$, i.e. double
points on fibres. This locus is given on $X\scl 2._B$ as the
degeneracy locus of a bundle map
\[\phi:(n+1)\O\to\Lambda_2(L), L:=f^*\O(1).\]
By Porteus, the virtual fundamental class of this locus
is given by the Segre class $s_n(\Lambda_2(L)^*)$, which equals
\eqspl{}{
\sum\limits_{i=0}^n(L_1)^{n-i}(L_2-\Gamma)^i, \Gamma=\Gamma\scl 2..
}
The powers of $\Gamma$ can be evaluated using Corollary \ref{polpowers,m=2}.
Pushing the result down to $X^2_B$ for simplicity yields
\eqsp{
\sum\limits_{i=0}^nL_1^{n-i}L_2^i-
\sum\limits_{i=0}^nL_1^{n-i}(\sum\limits_{j=1}^i
(\Gamma[\omega^{j-1}]+\sum\limits_{s,k}
\delta_{s*}(\psi_x^{j-2-k}\psi_y^k))
L^{i-j})
}
To describe the direct image of this on $B$, we need some notation. Recall that $\kappa_j=\pi_*(\omega^{j+1})$. Extending this, we may set
\eqspl{}{
\kappa_j(L)=\pi_*(L^{j+1}), \kappa_{i,j}(L,M)=
\pi_*(L^{i+1}M^{j+1}).
} Note that in our case $\kappa_j(L)$ may be interpreted
as the class of the locus of curves meeting a generic $\P^{n-j}$. Also,
for each boundary datum $(T_s, \delta_s, \theta_s)$, $T_s$
admits a map to $\P^n$ via either the $x$ or $y$-section
(the two maps are the same), via which we can pull back $L^j$,
which corresponds to the locus of boundary curves whose node
$\theta_s$ meets $\P^{n-j}$.
 Then pushing the above down to $B$ yields
\eqspl{}{
2m_2=(-1)^n(\sum\limits_{i=1}^{n-1}\kappa_{i-1}(L)
\kappa_{n-i-1}(L)-\kappa_{n-j-1, j-2}(L,\omega)
+\sum\limits_{s,k}
\delta_{s*}(L^{n-j}\psi_x^{j-2-k}\psi_y^k))
}
More generally, for any smooth variety $Y$ of dimension $n$
 and map $f:X\to Y$, one can use the double-point formula of \cite{R}, Th. 3.3ter, p. 1208, to evaluate the
class of the double-point locus in $X^2_B$ in terms of the diagonal
class $\Delta_Y$ on $Y\times Y$ as
\eqspl{}{
2m_2&=(f^2)^*(\Delta_Y)+\sum\limits_{i\geq 1}(-\Gamma^i)c_{n-i}(T_Y)\\
&=(f^2)^*(\Delta_Y)+\sum\limits_{i\geq 1}(
-\Gamma[\omega^{i-1}]+\half\sum\limits_{s,j}\delta_{s*}(\psi_x^{i-j-3}\psi_y^j))
c_{n-i}(T_Y)
}
Applying this set-up to the case $L=\omega$, one would like
in principle to
be able to compute fundamental classes of loci of hyperelliptics
(and more generally, $\frak M^r_d$ loci in $\overline{\frak M}_g$).
The problem is that the naive notion of canonical curve
in $\P^{g-1}$ is ill-behaved over the boundary and requires substantial
modification there. This work is currently in progress.
\bibliographystyle{amsplain}
\bibliography{mybib}
\end{document}

%% file: myheader2.tex
\DeclareMathOperator{\Bl}{B\l} 
\def\inv{^{-1}}

\DeclareMathOperator{\del}{\partial}

\def\refp #1.{(\ref{#1})}

\newcommand{\bsl}{\begin{slide}{}}

\newcommand{\A}{\mathcal{A}}

\newcommand{\Cal}[1]{\mathcal #1}

\newcommand{\un}{{\underline{n}}}
\def\sch{\text{sch}}
\def\sbr #1.{^{[#1]}}
\def\sfl #1.{^{\lfloor #1\rfloor}}

\newcommand\strad{{\mathrm {strad}}}
\newcommand\subp [1] {_{(#1)}}

\newcommand\half{\frac{1}{2}}

\def\bt{\boxtimes}
\def\what{\widehat}
\def\inv{^{-1}}
\def\?{{\bf{??}}}

\def\Hilb{\text{Hilb}}

\def\Proj{\textrm{Proj}}

\def\H{\mathcal H}

\def\A{\Bbb A}

\def\C{\mathbb C}
\def\P{\mathbb P}

\def\R{\mathbb R}
\def\Z{\mathbb Z}

\def\ord{\text{\rm ord}}

\def\Spec{\text{\rm Spec} }

\def\ls{\vskip.25in}

\def\Q{\mathbb Q}

\def\O{\mathcal O}

\def\y{\bar{y}}
\def\bt{\boxtimes}

\def\Sym{\textrm{Sym}}
\def\id{\text{id}}

\def\m{\mathfrak m}

\def\1/2{\frac{1}{2}}

\def\I{\mathcal{ I}}

\def\simto{\stackrel{\sim}{\rightarrow}}
\def\2{{[2]}}
\def\l{\ell}
\def\nl{\newline}

\def\<{\langle}
\def\>{\rangle}

\def\2{{[2]}}
\def\l{\ell}
\def\Proj{\text{Proj}}
\def\sbr #1.{^{[#1]}}
\def\scl #1.{^{\lceil#1\rceil}}
\def\spr #1.{^{(#1)}}
\def\supp{\text{supp}}
\def\subpr#1.{_{(#1)}}

\def\beq{\begin{equation*}}
\def\eeq{\end{equation*}}
\newcommand{\td}{\tilde }

\newcommand{\beql}[2]{\begin{equation}\label{#1}#2\end{equation}}
\newcommand{\beqa}[2]{\begin{eqnarray}\label{#1}#2\end{eqnarray}}
\newcommand{\eqspl}[2]{\begin{equation}\label{#1}
\begin{split}#2\end{split}\end{equation}}
\newcommand{\eqsp}[1]{\begin{equation*}
\begin{split}#1\end{split}\end{equation*}}
\newcommand{\itemul}[1]{\item[\underline{#1}]}
\newcommand{\exseq}[3]{
0\to #1\to #2\to #3\to 0
}
\def\eex{\end{rm}\end{example}}
\newcommand\newsection[1]{\section{#1}\setcounter{equation}{0}
}
\newcommand\newsubsection[1]{\subsection{#1}\setcounter{equation}{0}}

\newtheorem*{mainthm}{Main Theorem}
\newtheorem{thm}{Theorem}[section]

\newtheorem{cor}[thm]{Corollary}
\newtheorem{ex}{Exercise}[section]
\newtheorem{con}{Conjecture.}
\newtheorem{lem}[thm]{Lemma}
\newtheorem{claim}[thm]{Claim}
\newtheorem{prop}[thm]{Proposition}
\newtheorem{propdef}[thm]{Proposition-definition}
\newtheorem{defn}[thm]{Definition}
\theoremstyle{remark}

\newtheorem{rem}[thm]{Remark}
\newtheorem{example}[thm]{Example}

\newcommand{\mm}{m,m-1}